\documentclass[onefignum,onetabnum]{siamart171218}

\usepackage{amsmath, amssymb, mathrsfs}
\usepackage{physics}
\usepackage{lipsum}
\usepackage{amsfonts}
\usepackage{graphicx,mathtools}
\usepackage{algorithmic}
\usepackage{booktabs}
\usepackage[titletoc]{appendix}
\usepackage{bm}

\usepackage{enumitem}
\setlist[enumerate]{leftmargin=.5in}
\setlist[itemize]{leftmargin=.5in}

\newtheorem{Th}{Theorem}[section]

\newtheorem{Cor}[Th]{Corollary}

\newtheorem{Def}[Th]{Definition}
\newtheorem{Rem}[Th]{Remark}
\newtheorem{?}[Th]{Problem}

\def\R{{\mathbb R}}

\def\C{{\mathsf {cpl}}}
\def\N{{\mathbb N}}
\def\E{{\mathbb E}}
\def\R{{\mathbb R}}

\def\P{{\mathbb P}}

\def\V{{\mathbb V}}

\def\A{{\mathcal A}}

\def\cor{{\text {Corr}}}

\def\e{{\varepsilon}}

\def\re{{\text{re}}}

\def\ex{{\text{epr}}}
\def\ext{{\text{ept}}}

\def\mse{{\textsf{MSE}}}
\def\risk{{\textsf{Risk}}}
\def\lrmc{{\textsf{LRMC}}}
\DeclareMathOperator*{\argmax}{arg\,max}
\DeclareMathOperator*{\argmin}{arg\,min}

\headers{A bandit-learning approach to multifidelity approximation}{Y. Xu, V. Keshavarzzadeh, R.M. Kirby, 
and A. Narayan}

\title{A bandit-learning approach to multifidelity approximation\thanks{Submitted to the editors.
\funding{Y. Xu and A. Narayan are partially supported by NSF DMS-1848508. V. Keshavarzzadeh and A. Narayan are partially supported by AFOSR under award FA9550-20-1-0338.
V. Keshavarzzadeh and R.M. Kirby acknowledge that their part of this research was sponsored by ARL under cooperative agreement number W911NF-12-2-0023.}}}

\author{Yiming Xu\thanks{Department of Mathematics, and Scientific Computing and Imaging Institute, University of Utah
  (\email{yxu@math.utah.edu}, \email{akil@sci.utah.edu}).}
\and Vahid Keshavarzzadeh\thanks{Scientific Computing and Imaging Institute, University of Utah
  (\email{vkeshava@sci.utah.edu}).}
 \and  Robert M. Kirby\thanks{School of Computing, and Scientific Computing and Imaging Institute, University of Utah
  (\email{kirby@cs.utah.edu}.)}
\and Akil Narayan\footnotemark[2]}

\usepackage{amsopn}

\begin{document}

\maketitle

\begin{abstract}
Multifidelity approximation is an important technique in scientific computation and simulation. 
In this paper, we introduce a bandit-learning approach for leveraging data of varying fidelities to achieve precise estimates of the parameters of interest. 
Under a linear model assumption, we formulate a multifidelity approximation as a modified stochastic bandit and analyze the loss for a class of policies that uniformly explore each model before exploiting it. 
Utilizing the estimated conditional mean-squared error, we propose a consistent algorithm, adaptive Explore-Then-Commit (AETC), and establish a corresponding trajectory-wise optimality result.   
These results are then extended to the case of vector-valued responses, where we demonstrate that the algorithm is efficient without the need to worry about estimating high-dimensional parameters.  
The main advantage of our approach is that we require neither hierarchical model structure nor \textit{a priori} knowledge of statistical information (e.g., correlations) about or between models.  Instead, the AETC algorithm requires only knowledge of which model is a trusted high-fidelity model, along with (relative) computational cost estimates of querying each model.
Numerical experiments are provided at the end to support our theoretical findings.
\end{abstract}



\section{Introduction}

Computational models are ubiquitous tools in different aspects of science and engineering, from the discovery of new constitutive laws in physics and mechanics to the design of novel systems such as multifunctional materials. Such models are primarily developed to describe the system of interest accurately such that they can replace time-consuming real-life experiments. In addition to accuracy, an important aspect to consider is the computational cost of model evaluations, which typically increases as the accuracy of the model increases. The trade-off between cost and accuracy constitutes a fundamental challenge in computational science that has been the subject of active research. This challenge has given rise to a widely studied class of methods that involves several computational models with various levels of accuracy and cost. In different contexts and formulations, such methods are called multifidelity, multilevel, multiresolution, etc. The common theme among all these methods is the availability of several models that describe the same system of interest but with varying accuracy and computational cost. 

\subsection{Multifidelity models}
The notion of fidelity may have different meanings in different problems. 
A typical scenario corresponds to the level of discretization for solving PDEs, where the finest mesh is referred to as the high-fidelity model and coarser meshes are lower fidelity models. 
Other common multifidelity settings include the interpolation and regression models \cite{Forrester09,Forrester08},  projection-based reduced-order models \cite{Sirovich87,Rozza08,Gugercin08,Benner15}, machine learning models~\cite{Vapnik98,Cortes95,Chang11} and other reduced-order models~\cite{Ng16,Majda10}. 
The multifidelity setting considered in this paper is a general abstraction that includes low-fidelity models instantiated as coarse discretizations, via reduced-order models (e.g., reduced-basis methods) or other types of emulators such as Gaussian processes \cite{Hesthaven16,Rasmussen06}. 

Combining the low-fidelity models with the high-fidelity model to accelerate computational efficiency is a fruitful idea in engineering science \cite{Peherstorfer_2018_survey}. It can be construed as a model management strategy with three main categories of approaches: 1) Adaptation of the low-fidelity model to the high-fidelity model by the use of high-fidelity information; approaches based on model correction fall into this category~\cite{Eldred2004,Alexandrov98,Alexandrov01}, 2) fusion of multiple fidelity modes such as control variate approaches~\cite{Bratley87,Hammersley64,Nelson87,Koutsourelakis09,Gorodetsky_2020}; and 3) filtering approaches such as importance sampling where low-fidelity models serve as filters and determine when to use the high-fidelity model~\cite{Christen05,Fox97,Narayan14,Peherstorfer14_SIAM_conf}. 
The method developed in this paper belongs to the realm of fusion approaches, where we assume in particular that a subset of low-fidelity models can be combined linearly to effectively predict the value of the high-fidelity model.

\subsection{Bandit learning}
The idea of bandit learning first appeared in a study \cite{Thompson_1933} of treatment design problems in clinical trials.
It was later developed as a powerful tool to study sequential decision-making problems in an uncertain environment.
The primary goal of bandit learning is to find an adaptive strategy that produces close-to-optimal decisions by minimizing the loss or regret. 
The uncertainty in an environment can be modeled using different mechanisms, such as random feedback \cite{Lai_1985, auer2002finite,dani2008stochastic,Rusmevichientong_2010}, adversarial manipulation \cite{Auer_2002,auer1995gambling,bubeck2012towards}, Markov chains \cite{anantharam1987asymptotically, gittins1979bandit, jaksch2010near}, etc., giving rise to a wealth of bandit setups that find use in numerous applications in practice \cite{bouneffouf2019survey, zhao2019multi}. 
Regret, which plays a similar role as a loss function in statistical learning, can be chosen in various forms, often depending on the specific application of interest. 
For example, cumulative loss \cite{auer2002finite,Auer_2002} is widely used as regret to measure the overall performance of a strategy, and a discounted version \cite{gittins1979bandit} can be used to account for the elapsed time during the strategy. In other situations where only the final decision matters, the loss at a fixed (final) step is adopted \cite{bubeck2009pure,audibert2010best,garivier2016optimal}. 
Despite the diversity of choices for regret, the overall goal is quite thematic: find a strategy that performs similarly to an oracle by investing a reasonable amount of computational complexity in a procedure that adaptively learns.  
The key philosophy behind bandit learning thus lies in efficiently learning a competitive strategy and then utilizing the learned knowledge for future use.     
A comprehensive treatment of the subject can be found in, for instance, \cite{lattimore2020bandit, bubeck2012regret}.

\subsection{Related work}\label{ssec:related}
Our approach in this article centers around using non-hierarchical linear regression for constructing a multifidelity estimator. The idea of using linear regression for multifidelity approximation is not completely new. 
For instance, recent work has introduced a general linear regression framework to study telescoping-type estimators in multifidelity inference \cite{Schaden_2020, schaden2020asymptotic}. This work treats different fidelities as factor models with heterogeneous noise structures. Assuming that the covariance matrix of the noise is known, the authors proposed a re-weighted least squares estimator which is shown to be the best linear unbiased estimator for the quantity of interest (under a fixed allocation choice). 
This framework offers a uniform perspective on many well-studied multifidelity and multilevel estimators \cite{giles2008multilevel,giles2015multilevel, Gorodetsky_2020,Peherstorfer_2016} and provides certain limits on their theoretical performance. However, the construction utilizes oracle information about model correlations, which are often not known beforehand and must first be learned from the data, placing practical restrictions on the immediate deployment of such procedures. There is limited work on multifidelity methods with non-hierarchical dependence, i.e., methods that do not necessarily assume an ordering of models based on cost/accuracy \cite{gorodetsky_mfnets_2020}, and algorithms for identifying non-hierarchical relationships are an area of active research.

In \cite{Zhang_2018}, the authors proposed to compute a high-fidelity model using a linear combination of low-fidelity surrogates corrected by a deterministic discrepancy function. 
This approach is easy to implement and allows classical tools to be used in the analysis.  
Nevertheless, certain drawbacks exist: 
The assumption on the discrepancy function can be restrictive when noise pollutes data in a random fashion, 
and realizing the optimal performance of the method requires selecting the `best' surrogate model in practice, which is a nontrivial task when the computational budget is limited.

The work \cite{kandasamy2016multi} has a similar title to this paper, but the setup and goals considered there are completely different.

\subsection{Contributions of this paper}
This paper provides techniques to address some of the challenges identified in Section \ref{ssec:related} that limit the applicability of current approaches.
Under a type of linear model assumption (that is slightly stronger than what is assumed in \cite{Schaden_2020}), we introduce a bandit-learning approach that, under a budget constraint, facilitates computational learning of the relation between different models before an exploitation choice is committed.
Our algorithm seeks to identify the best selection of regressors by investing a portion of the budget in an exploration phase. This selection is then utilized in an exploitation phase to construct a surrogate for the high-fidelity model, which is significantly less expensive but exhibits comparable accuracy.  

We make three contributions in this paper:
\begin{itemize}
\item 
We formulate multifidelity approximation as a bandit-learning type problem under a linear model assumption. 
In particular, we introduce a class of policies called uniform exploration policies and define their associated loss (regret).  
We also derive an asymptotic formula for their loss as the computational budget goes to infinity. 
\item
  We propose an adaptive Explore-Then-Commit (AETC) algorithm (Algorithms \ref{alg:aETC}) that automatically identifies a trade-off point between exploration and exploitation. 
We prove that the estimator produced by the algorithm is consistent and asymptotically matches the best regression model with optimal exploration.  In particular, this algorithm requires only model cost information and identification of a high-fidelity model. No model correlations or statistics are needed as input.
\item
We initially consider scalar-valued responses, but subsequently generalize to vector-valued high-fidelity models, and demonstrate that the algorithm is efficient despite some potentially high-dimensional parameter estimation procedures. 

\end{itemize} 
Our numerical results strongly support our theoretical findings. 
Our methodology enjoys great generality as it does not require any particular knowledge of the models, i.e., hierarchical structure or correlation statistics. The procedure requires only the identification of a trusted high-fidelity model, the ability to query the models themselves, and a relative cost estimate of each model relative to the high-fidelity model.
It is worth mentioning that our approach can be extended to estimate more complicated statistics such as the cumulative distribution functions (CDF) of QoIs \cite{xu2021budget}, but this requires stronger assumptions.  

The rest of the paper is organized as follows. 
In Section \ref{<}, we introduce the abstract setup of multifidelity approximation with which we will be working in the rest of the paper, followed by a brief review of the key ideas behind stochastic bandits. 
In Section \ref{blp}, we formulate multifidelity approximation as a modified bandit-learning problem, and in Section \ref{ttt}, we prove the consistency of the estimators used during the exploitation phase and derive a nonasymptotic convergence rate. 
In Section \ref{alg}, we propose an adaptive algorithm, AETC, based on the estimated conditional mean-squared errors from the analysis and establish a trajectory-wise optimality result for it.  
In Section \ref{sec:vec}, we extend our results developed in the previous sections to the case of vector-valued high-fidelity models and justify the efficiency of the estimation procedures where high-dimensional parameters are involved.  
In Section \ref{num}, we provide a detailed study of the novel AETC algorithm via numerical experiments that verifies our theoretical statements and demonstrates the utility of the AETC algorithm. 
In Section \ref{concl}, we conclude by summarizing the main results in the paper.

\section{Background}\label{<}
\subsection{Notation}\label{ps}
Fix $n\in\N$. 
We let $Y, X_1, \cdots, X_n$ be random vectors (possibly of different dimensions) representing the high-fidelity model and $n$ surrogate low-fidelity models, respectively.  
Let $c_i$ ($i\in [n]$) and $c_0$ be the respective cost of sampling $X_i$ and $Y$.  
For example, in parametrized PDEs, $X_i$ is the approximate solution given by the $i$-th solver under random coefficients of a PDE, and $Y$ is the solution given by the most accurate solver, which is assumed to be the ground truth. 
The cost $c_i$ corresponds to the computational time.  We make no assumptions about the accuracy or costs of $X_i$ relative to those of $X_{i+1}$. In particular, the index $i$ does not represent an ordering based on cost, accuracy, or hierarchy.
For $X_i$ ($i\in [n]$) and $Y$, let $f_i$ and $f$ be output quantity of interest maps that bring $X_i$ and $Y$, respectively, to $f_i(X_i)\in\R^{k_i}$ and $f(Y)\in\R^{k_0}$, which are the quantities of practical interest. 
Understanding the mean of $f(Y)$ is crucial for many simulation problems in engineering. 

Computing the expectation of $f(Y)$ may not seem challenging at all: Under mild assumptions on the distribution of $Y$, a Monte Carlo (MC) procedure exploits the law of large numbers using an ensemble of independent samples to accurately estimate the mean of $f(Y)$. 
However, in practice, obtaining numerous samples of $Y$ can be computationally expensive if $c_0$ is large. This motivates the idea of estimating $\E[f(Y)]$ using low-fidelity model outputs $f_i(X_i)$, which often contain some information about $f(Y)$ and are cheaper to sample. 
For instance, when $f_i(X_i)$ and $f(Y)$ are scalars and the correlation between them is high, one can construct an unbiased estimator for $\E[f(Y)]$ by taking a telescoping sum of the low-fidelity models plus a few samples of the ground truth $f(Y)$. Optimizing over the allocation parameters under a minimum variance criterion leads to the Multi-Fidelity Monte-Carlo (MFMC) estimator.  
It is shown in \cite{Peherstorfer_2016} that the MFMC method outperforms MC by a substantial margin on several test datasets. However, successful implementation of the method requires both knowledge of statistical information of and between the models, as well as a hierarchical structure, neither of which may be known beforehand. More precisely, in MFMC, a training stage is required to learn correlations between models, and guidance on how to perform this training is heuristic. Our procedure is a principled approach that effects a similar type of training in an automated way and is accompanied by theoretical guarantees.
In the rest of this section, we introduce a general framework in terms of learning the high-fidelity model from its surrogates. 
Our method utilizes ideas from bandit learning and therefore is less reliant on any existing knowledge of the models. 
We will need a general linear model assumption that is slightly stronger than the generic correlation assumption and has been widely used in statistics. 
For simplicity, we assume $k_i=1$ for $i\in [n]$ and denote
\begin{align} 
X^{(i)}: = f_i(X_i)\in\R.
\end{align} 
We will first deal with the case where $f(Y)$ is a scalar, i.e., $k_0=1$. 
The result will be generalized to the vector-valued responses in Section \ref{sec:vec}. 
 
\subsection{Linear regression}\label{lastt}
Let $S\subseteq  [n]$ be a selection of low-fidelity models with $|S| = s > 0$.  
Suppose $f(Y)$ and $\{X^{(i)}\}_{i\in S}$ satisfy the following \emph{linear model assumption}:
\begin{align}
f(Y) &= X_S^T\beta_S + \e_S,\label{1}
\end{align}
where $$X_S = (1,(X^{(i)})_{i\in S})^T\in\R^{s+1}$$ is the regressor vector, which includes a constant (intercept) term; $\beta_S\in\R^{s+1}$ is the coefficient vector; and $\e_S\in\R$ is model noise, which we assume is independent of $X_S$. 
For convenience, we assume in the following discussion that $\e_S$ is a centered \emph{sub-Gaussian} random variable with variance $\sigma_S^2$.

Note that even though \eqref{1} considers only linear interactions with $X_S$, it is less restrictive than it appears.
Indeed, a key assumption implied by \eqref{1} is that the conditional expectation of $f(Y)$ given $X_S$ (i.e., $\E[f(Y)|X_S]$) is a linear function of $X_S$. 
By definition, $\E[f(Y)|X_S]$ is a measurable function of $X_S$, which can be approximated using, e.g., polynomials of $X_S$ on any compact set under mild regularity conditions.
When the linearity assumption is violated, one can add higher order terms of $X_S$ as new regressors to fit a larger linear model to mitigate the model misspecification effect. 
This procedure will not incur any additional computational cost.
Nevertheless, identifying an optimal set of appropriately expressive regressors (features) is a much harder problem that goes beyond the scope of this paper.

Our goal is to seek an efficient estimator for $\E[f(Y)]$, so taking the expectation of \eqref{1} yields
\begin{align}
\E[f(Y)] = \E[X_S^T]\beta_S.\label{2}
\end{align}
The term $\E[X_S]$ on the right-hand side of \eqref{2} can be estimated by averaging the independent joint samples of $X_S$, which is the same as the Monte Carlo applied to the conditional expectation $\E[f(Y)|X_S]$, similar to a model-assisted approach in sampling statistics \cite{fuller2011sampling}.

What is gained from sampling from $f(Y)|X_S$ instead of $f(Y)$? The two major advantages are:
\begin{itemize}
\item
The $N$-sample Monte Carlo estimator for $\E[f(Y)]$ based on sampling $f(Y)$ has mean-squared error $\V[f(Y)]/N$, where $\V[ \cdot ]$ denotes the variance.
The same procedure via sampling $f(Y)|X_S$ has mean-squared error $\V[\E[f(Y)|X_S]]/N$. 
Since conditioning does not increase variance, the numerator of the latter is bounded by the numerator of the former for any fixed $N$.
\item 
Given a fixed budget, when $\sum_{i\in S}c_i\ll c_0$, the number of affordable samples for $f(Y)|X_S$ is much larger than $f(Y)$, which can significantly reduce the variance of the estimator.  
\end{itemize}
Both observations provide some heuristics that \eqref{2} may give rise to a better estimator for $\E[f(Y)]$ under certain circumstances. 

In practice, neither the best model index set $S$ (which is referred to as \textit{model} $S$ in the rest of the article) nor the corresponding $\beta_S$ is known. 
Thus, we cannot avoid expending some resources to estimate $\beta_S$ for every model $S\subseteq  [n]$ before deciding which model can best predict $\E[f(Y)]$. One possible way to achieve this is via bandit learning.

\begin{table}
  \begin{center}
  \resizebox{\textwidth}{!}{
    \renewcommand{\tabcolsep}{0.4cm}
    \renewcommand{\arraystretch}{1.3}
    {\scriptsize
    \begin{tabular}{@{}cp{0.8\textwidth}@{}}
    \toprule
      $X^{(i)}$ & the $i$-th regressor \\ 
      $f(Y)$ & the response variable (vector) \\
      $k_i, k_0$ & the dimension of $X^{(i)}$ and $f(Y)$ \\
      $m$ & the number of samples for exploration\\
      $n$ & the number of regressors \\
      $B$ & total budget\\
      $S$ & the model index set (indices of regressors used in regression)\\
       $(c_i)_{i=0}^n/c_{\ex}/c_{\ext}(S)$ & cost parameters/total cost/cost of model $S$\\
       $N_S$ & the affordable samples in model $S$\\
      $\beta_S$& coefficients (matrix) of model $S$ \\
      $X_{\ex,\ell}$& the $\ell$-th sample in the exploration phase \\
      $Z_S$ & the design matrix in the exploration phase in model $S$ \\
      $\e_S/\eta_S$ & the noise variable/vector in model $S$ \\
      $x_S$ & the mean of the regressors in model $S$ \\
      $\Sigma_S$& the covariance matrix of the regressors in model $S$ \\
      $\sigma_S^2$ & the variance in model $S$ (single response) \\
      $Q$ & the weight matrix in defining the $Q$-weighted risk \\
      $\Gamma_S$ & the covariance matrix of the noise in the case of vector-valued response \\
      \bottomrule
    \end{tabular}
  }
    \renewcommand{\arraystretch}{1}
    \renewcommand{\tabcolsep}{12pt}
  }
  \end{center}
\caption{\small Notation used throughout this article.}\label{tab:notation}
\end{table}


\subsection{Stochastic bandits}\label{rbl}
This section provides a brief overview of the philosophy behind bandit learning, and in particular, we focus on stochastic bandits, which is the subfield most relevant to our procedure. 
 
Consider a \textit{multiarmed bandit} setup: we are faced with $k$ slot machines or \textit{arms}, where $k\in\N$ is fixed. 
At each integer time $t$, a player pulls exactly one arm and this player will receive a reward. 
Rewards generated from the same arm at different times are assumed to be independent and identically distributed; rewards generated from different arms are independent and their distributions are stationary in time.  
Given a fixed (time) horizon $N$, the goal of stochastic bandit learning is to seek a policy with small `regret'.
A \emph{policy} $\pi = (\pi_1, \cdots, \pi_N)\in [k]^N$
is a random sequence of choices adapted to the natural filtration generated by past observations and actions. 
`Regret' is a special type of loss function that measures the quality of a policy.  
In the classical setup, the regret of a policy $\pi$ is defined as 
\begin{align}
R_N(\pi): &= N\mu_{\max} - \E\left[\sum_{t\in [N]}z_{t}(\pi)\right] \nonumber\\
&= \sum_{i\in [k]}T_N(i)\Delta_i& T_N(i) = \E\left[\#\{t\leq N, \pi_t = i\}\right],\label{blr}
\end{align}
where $z_t(\pi)$ is the reward received at time $t$ under policy $\pi$ (i.e., corresponds to reward $\pi_t$ at time $t$), $\mu_i$ is the expectation of the reward distribution of arm $i$, $\mu_{\max} : = \max_{i\in [k]}\mu_i$, and $\Delta_i: = (\mu_{\max}-\mu_i)$ is the \textit{suboptimality} gap of arm $i$. 
Intuitively, \eqref{blr} measures the difference between the average total reward under $\pi$ compared to the best (oracle) mean reward, which is generally not known.

In practice, we assume that the true reward distributions are not available, and so it is necessary to expend effort estimating the expected reward of each arm. 
A simple idea to estimate these average rewards is to pull each arm a fixed number of times $m$, and use the gathered data to estimate the mean rewards.   
Such a process is called an \textit{exploration} stage in bandit learning. 
Based on the estimated rewards, we can then make a decision about which arm performs the best and repeatedly select it for the remaining time (the \textit{exploitation} phase) until the horizon is reached. 
A direct combination of the exploration and exploitation gives the Explore-Then-Commit (ETC) algorithm, shown in Algorithm \ref{alg:bETC}.

\begin{algorithm}
\hspace*{\algorithmicindent} \textbf{Input}: $m$: the number of exploration on each arm \\
    \hspace*{\algorithmicindent} \textbf{Output}:  $\pi=(\pi_t)_{t\in [N]}$
 \caption{Explore-then-Commit (ETC) algorithm for stochastic bandits} 
 \begin{algorithmic}[1]
 \IF{$t\leq mk$}
 \STATE{$\pi_t =  \lceil t\bmod {k}\rceil$}
 \ELSE
 \STATE{$\pi_t = \argmax_{i\in [k]}\widehat{\mu}_{i}(mk)$, where $\widehat{\mu}_{i}(mk)=\frac{1}{m}\sum_{t=m\cdot (i-1)+1}^{m\cdot i}z_{t}(\pi)$}
 \ENDIF
\end{algorithmic}
\label{alg:bETC}
\end{algorithm}
The success of the ETC algorithm closely depends on the input parameter $m$, which dictates how many times each arm is pulled in the exploration phase, and is therefore indicative of the cost of this phase. A small $m$ may result in poor estimation (and thus poor exploration), and in this case, there is a considerable chance that the chosen exploitation strategy will be suboptimal. A large $m$, on the other hand, leaves little room for exploitation, and regret will then be dominated by the exploration phase. A good choice of $m$ lies in finding the right trade-off point between exploration and exploitation. When $k=2$ and reward distributions are sub-Gaussian, a near-optimal $m$ can be explicitly computed \cite[Chapter 6]{lattimore2020bandit}.  
However, such explicit expressions often rely on the knowledge of the suboptimality gaps $\Delta_i$, which are not available in general. 
More useful algorithms which can be viewed as adaptive generalizations of the ETC include the Upper-Confidence Bound (UCB) algorithm \cite{auer2002finite} and the Elimination algorithm \cite{Auer_2010}. 

One generalization of stochastic bandits pertaining to the multifidelity approximation problem of our interest is budget-limited bandits \cite{tran2010epsilon}, where each arm has a cost for pulling, and the number of pulls is constrained by a total budget instead of the time horizon. 
Similar algorithms as well as the regret analysis in the stochastic bandits can be carried out in the budget-limited setup \cite{tran2010epsilon, tran2012knapsack}.


\section{A bandit-learning perspective of the multifidelity problem}\label{blp}

In this section, we demonstrate that multifidelity approximation fits into a modified framework of bandit learning, which we exploit to develop an algorithm. 
Since our multifidelity goal is different from that in classical stochastic bandits, both the action set and the loss function (regret) must be tailored. We begin with the former. 

\subsection{Action sets and uniform exploration policies}
Our goal is to construct a regression-based MC estimator for $\E[f(Y)]$ using the best regression model. 
However, the relationship between $f(Y)$ and $X_S$ is unknown, requiring us to first estimate the coefficient vector $\beta_S$ in \eqref{2} for each $S$ and then decide which model is optimal. 
As a consequence, the action set contains both the exploration and exploitation options for each $S\subseteq  [n]$.   
Denote by $\A$ the set of actions. Then, we can write $\A$ as 
\begin{align}
\A = \left\{a_{\ex}(S), a_{\ext}(S)\right\}_{S\subseteq  [n]}
\end{align}
where 
\begin{align*}
&a_{\ex}(S): \text{collect a sample of $(X_S,f(Y))$}\\
&a_{\ext}(S): \text{use the remaining budget to sample $X_S$ }.
\end{align*}
As opposed to stochastic bandits, actions here are highly correlated.
For example, exploration will not be allowed whenever an exploitation action is triggered (which exhausts the budget), separating exploration and exploitation into two distinct phases. 
In addition, action $a_{\ex}(S)$ yields data for every $S'$ satisfying $S'\subseteq  S$. 
For simplicity, we specialize our action set to consider only \emph{uniform exploration policies}, i.e, each exploration yields a sample of all regressors, incurring the cost
\begin{align}
c_{\ex} := \sum_{i=0}^nc_i.\label{expc}
\end{align}
This procedure is similar to the ETC algorithm (Algorithm \ref{alg:bETC}) where arms are explored to the same degree during exploration.

Let $B>0$ be a given, fixed total budget, and let $m> n+1$ be the number of exploration actions. (We will specify how $m$ is chosen later.)
The budget $B_{\ex}$ spent on exploration under a uniform exploration policy $\pi$ and the budget $B_{\ext}$ remaining for exploitation are given by
\begin{align*}
  B_{\ex} &= c_{\ex}m, & B_{\ext} &= B - B_{\ex}.
\end{align*}
During exploration, each action yields a sample of the form
\begin{align*}
&X_{\ex,\ell} =\left (1, X^{(1)}_\ell, \cdots, X^{(n)}_\ell, f(Y_\ell)\right)^T& \ell\in [m], 
\end{align*}
where the subscript $\ell$ is a sampling index. For $S\subseteq  [n]$, let $X_{\ex,\ell}|_S\in \R^{s+1}$ (including the intercept term) and $X_{\ex,\ell}|_Y \in \R$ be the restriction of $X_{\ex,\ell}$ to model $S$ and $f(Y)$, respectively. The coefficient vector $\beta_S$ in \eqref{1} can be estimated by a standard least squares procedure:
\begin{align}
  \widehat{\beta}_S&= Z_S^\dagger X_\ex|_Y \label{lscoef}
\end{align} 
where 
\begin{align*}
  Z_S &= \left(X_{\ex, 1}|_S, \cdots, X_{\ex, m}|_S\right)^T, & 
  X_{\ex}|_Y &= \left(X_{\ex,1}|_Y, \ldots X_{\ex,m}|_Y \right)^T
\end{align*}
is the design matrix and data vector, respectively, and $Z_S^\dagger$ is the Moore-Penrose pseudoinverse of $Z_S$.
For simplicity, we will assume that the design matrix has full rank in the following discussion, so that $Z_S^\dagger: =  \left(Z_S^TZ_S\right)^{-1}Z_S^T$.
In exploitation, one selects an action $a_{\ext}(S)$ for some $S\subseteq  [n]$ and uses \eqref{2} to build an MC estimator for $\E[f(Y)]$. 
The true coefficient vector $\beta_S$ is unknown but can be replaced by the estimate $\widehat{\beta}_S$, yielding the 
Linear Regression Monte-Carlo (LRMC) estimator associated with model $S$:
\begin{align}
\lrmc_S = \frac{1}{N_S}\sum_{\ell\in [N_S]}X_{S,\ell}^T\widehat{\beta}_S,\label{LRMC}
\end{align}
where $N_S$ is the number of affordable samples to exploit model $S$:
\begin{align*}
&N_S = \left\lfloor\frac{B_{\ext}}{c_{\ext}(S)}\right\rfloor =  \left\lfloor\frac{B - c_{\exp}m}{c_{\ext}(S)}\right\rfloor & c_{\ext}(S) := \sum_{i\in S}c_i,
\end{align*}
and $X_{S,\ell}$ are i.i.d. samples of $X_S$ which are independent of the samples in the exploration stage. 
Here we choose not to reuse samples from the exploration phase during the exploitation process for convenience of analysis. 

\begin{Rem}\label{rem:recycle}
A similar analysis of the LRMC estimator that reuses or recycles the exploration data can be carried out but with more complicated notation.
For most applications that we consider, where $m/N_S\ll 1$ (since the high-fidelity model is often substantially more expensive than the low-fidelity emulators), reuse has little impact on the estimator in general. 
Regardless of analysis, one can always recycle exploration samples in a practical setting; our numerical results show that recycling exploration samples has negligible impact on the performance of our procedure for the examples we have tested.
See Appendix \ref{ap:0} for a theoretical justification, and Figure \ref{fig:1} in Section \ref{num} for numerical evidence.
\end{Rem}  
We will show in Section \ref{ttt} that for fixed $S\subseteq  [n]$, $\lrmc_S$ defined in \eqref{LRMC} is, almost surely, a consistent estimator for $\E[f(Y)]$, and we will provide a convergence rate. 

\subsection{Loss function}
The loss function used in bandit learning is called regret, which is often defined by the reward difference between a policy and an oracle. 
In our case, it is more convenient to define loss as a quantity that we wish to minimize.
Note that the output of a uniform exploration policy $\pi$ is an LRMC estimator \eqref{LRMC},
where the selected model $S$ satisfies $\pi_{m+1} = a_{\ext}(S)$.  
One way to measure the immediate quality of \eqref{LRMC} is through the following conditional mean-squared error (MSE) on $\widehat{\beta}_S$: 
\begin{align}
\mse_S|_{\widehat{\beta}_S} = \E\left[\left(\lrmc_S - \E[f(Y)]\right)^2\big | \widehat{\beta}_S\right].\label{def:mse}
\end{align}
The $\lrmc_S$ alone, despite being an unbiased estimator (which is easily verified using independence), is not a sum of i.i.d. random variables unless conditioned on $\widehat{\beta}_S$. Once conditioned, $\lrmc_S$ becomes a sum of i.i.d. random variables that is a biased estimator for $\E[f(Y)]$.
Define the following statistics of $X_S$,
\begin{align*}
  x_S &= \E [X_S], & \Sigma_S = \mathrm{Cov} [X_S].
\end{align*}
By writing $\mse_S|_{\widehat{\beta}_S}$ using the bias-variance decomposition, we obtain
\begin{align}
\mse_S|_{\widehat{\beta}_S} &= (x_S^T(\widehat{\beta}_S-\beta_S))^2 + \V\left[\lrmc_S\big |\widehat{\beta}_S\right] \nonumber\\
& = (\widehat{\beta}_S-\beta_S)^Tx_Sx_S^T(\widehat{\beta}_S-\beta_S) + \frac{1}{N_S}\widehat{\beta}_S^T\Sigma_S\widehat{\beta}_S ,\label{ana}
\end{align}
where $\V[ \cdot  | \widehat{\beta}_S ]$ is the $\widehat{\beta}_S$-conditional variance operator.
Note that \eqref{ana} decomposes the loss incurred in the exploration and exploitation phases by the bias term and variance term, respectively.
The \emph{average conditional MSE} of the LRMC is defined as \eqref{ana} averaged over the randomness of the model noise $\e_S$ in exploration:
\begin{align}
\overline{\mse}_S|_{\widehat{\beta}_S} &:= \E_{\e_S}\left[\mse_S|_{\widehat{\beta}_S}\right] \nonumber\\
&= \frac{1}{N_S}\left[\beta_S^T\Sigma_S\beta_S + \sigma_S^2\tr(\Sigma_S(Z_S^TZ_S)^{-1})\right] + \sigma_S^2\tr(x_Sx_S^T(Z_S^TZ_S)^{-1}).\label{dep}
\end{align}
Note that $\overline{\mse}_S|_{\widehat{\beta}_S}$ defined in \eqref{dep} is random, and one could alternatively define it by averaging all the randomness (both $\e_S$ and $Z_S$). However, this approach is more difficult to analyze due to the term $\E[(Z_S^TZ_S)^{-1}]$. 
In the following discussion, we will use \eqref{dep} as the \emph{loss function} associated with the uniform exploration policy $\pi$ to examine its average quality in terms of estimating $\E[f(Y)]$.

\section{Consistency of the LRMC estimator}\label{ttt}

In this section, we will show that the LRMC estimators constructed in the previous section are consistent estimators for $\E[f(Y)]$. 
The result can be derived as a corollary of a nonasymptotic estimate of the convergence rate of the $\widehat{\beta}_S$-conditional MSE.
The following definition will be used in the subsequent analysis:
\begin{Def}[$\alpha$-Orlicz norm]
The $\alpha$-Orlicz ($\alpha\geq 1$) norm of a random variable $W$ is defined as 
\begin{align*}
\|W\|_{\psi_\alpha}: = \inf\{C>0: \E[\exp(|W|^\alpha/C^\alpha)]\leq 2\}.
\end{align*}
\end{Def}

\begin{Th}\label{thm:conss}
Fix $S\subseteq  [n]$.
Suppose $X_S$ satisfies
\begin{align}
&\max\left\{\sup_{\|\theta\|_2=1}\E[\langle X_S,\theta\rangle^4]^{1/4}, \left\|\|X_S\|_2\right\|_{\psi_1}\right\} \leq K<\infty&K\geq 1, \label{...}
\end{align}
where $\|\cdot\|_{\psi_1}$ is the $1$--Orlicz norm, and
$\Lambda_S :=  \E[X_SX_S^T] = x_Sx_S^T +\Sigma_S $ is invertible. 
Then, for large $m$, with probability at least $1-m^{-2}$, 
\begin{align}
&\mse_S|_{\widehat{\beta}_S} \lesssim \frac{\beta_S^T\Sigma_S\beta_S}{N_S}+(s+1)\sigma_S^2\frac{\log m}{m},\label{sd}
\end{align}
where the implicit constant in $\lesssim$ is independent of $m$.  
\end{Th}
\begin{proof}
See Appendix \ref{ap:1}.
\end{proof}

Theorem \ref{thm:conss} implies the following consistency result of the LRMC estimator, which follows immediately from \eqref{sd} and an application of the Borel-Cantelli lemma:

\begin{Cor}[consistency of the LRMC]
  Let $\{B_k\}_{k=1}^\infty$ satisfy $\lim_{k \uparrow \infty} B_k = \infty$, and let $m_k, N_k$ be the respective exploration round and exploitation samples satisfying $\min (m_k, N_k)\to\infty$. 
For fixed $S\subseteq  [n]$, denote $\lrmc_{S,k}$ the LRMC estimator under budget $B_k$, and $\widehat{\beta}_{S,k}$ the corresponding estimator for $\beta_S$.  
Then, for almost every sequence $\{\widehat{\beta}_{S,k}\}_{k=1}^\infty$, $\lrmc_{S,k}\xrightarrow{\P}\E[f(Y)]$ as $k\to\infty$. 
\end{Cor}

Without a nonasymptotic convergence rate, a strong consistency result for the LRMC estimator can be established using existing results in \cite[Theorem 1]{Lai_1982}. In fact, a sufficient condition for the estimator $\widehat{\beta}_S$ (hence $\widehat{\sigma}_S^2$) to be strongly consistent is that the largest and smallest eigenvalues of $(Z^T_SZ_S)^{-1}$, $\lambda_{\max}$ and $\lambda_{\min}$, satisfy $\lambda_{\min}\to\infty$ and $\log\lambda_{\max}/\lambda_{\min}\to 0$ as $m\to\infty$ a.s., which is verifiable under the condition \eqref{...}.

\section{Algorithms}\label{alg}

In this section, we first analyze the asymptotic loss associated with uniform exploration policies.
Then, we propose an adaptive ETC (AETC) algorithm based on the estimated average MSEs, which selects the optimal exploration round and the best model to commit for a large budget almost surely.  
For convenience, we will ignore all integer rounding effects in the rest of the discussion. 

\subsection{Oracle loss of uniform exploration policies}
A uniform exploration policy spends the first $m$ rounds on exploration and then chooses a fixed model $S$ among the subsets of $[n]$ for exploitation. 
Fixing $B$ and for each $S$, the optimal exploration $m_S(B)$ balances the terms in expression \eqref{dep}.
The best uniform exploration policy is the one that stops exploring at time $m_S(B)$ and then picks model $S$ for exploitation, where model $S$ has the smallest average conditional MSE computed at the optimal exploration round $m_S(B)$. To develop computable expressions, we will be mainly concerned with the case when $B$ tends to infinity. 

To find the optimal $m$ for fixed $S$, note that for sufficiently large $m$, the law of large numbers tells us that almost surely from \eqref{dep}, 
\begin{align}
\overline{\mse}_S|_{\widehat{\beta}_S}& = \frac{1}{N_S}\left(\beta_S^T\Sigma_S\beta_S + \sigma_S^2\tr(\Sigma_S(Z_S^TZ_S)^{-1})\right) + \sigma_S^2\tr(x_Sx_S^T(Z_S^TZ_S)^{-1})\nonumber\\
&\simeq \frac{c_{\ext}(S)}{B-c_{\ex}m}\left(\beta_S^T\Sigma_S\beta_S + \frac{1}{m}\sigma_S^2\tr(\Sigma_S\Lambda_S^{-1})\right) + \frac{1}{m}\sigma_S^2\tr(x_Sx_S^T\Lambda_S^{-1})\nonumber\\
&\simeq \frac{\beta_S^T\Sigma_S\beta_S}{B-c_{\ex}m}c_{\ext}(S) + \frac{\sigma_S^2\tr(x_Sx_S^T\Lambda_S^{-1})}{m},\label{optm}
\end{align}
where $A_1(m)\simeq A_2(m)$ means that $A_1(m)/A_2(m)\to 1$ as $m\to\infty$ with probability 1. 
Denoting
\begin{align} 
&k_1(S) =c_{\ext}(S)\beta_S^T\Sigma_S\beta_S & k_2(S) = \sigma_S^2\tr(x_Sx_S^T\Lambda_S^{-1})\leq  (s+1)\sigma_S^2,\label{yiyi}
\end{align}
the asymptotically best $m$ for exploiting model $S$ can be found by minimizing \eqref{optm} :
\begin{align}
m_S = \argmin_{1<m<B/c_{\ex}}\frac{k_1(S)}{B-c_{\ex}m} + \frac{k_2(S)}{m} = \frac{B}{c_{\ex}+\sqrt{\frac{c_{\ex}k_1(S)}{k_2(S)}}},\label{mode}
\end{align}
where the latter equality can be computed explicitly since \eqref{optm} is a strictly convex function of $m$ in its domain. 
The $\overline{\mse}_S|_{\widehat{\beta}_S}$ corresponding to $m = m_S$ is  
\begin{align}
\overline{\mse}^*_S|_{\widehat{\beta}_S} &=  \frac{(\sqrt{k_1(S)}+\sqrt{c_{\ex}k_2(S)})^2}{B}\propto \left(\sqrt{k_1(S)}+\sqrt{c_{\ex}k_2(S)}\right)^2.\label{yuc}
\end{align}
The best model is the one that has the smallest $\overline{\mse}^*_S|_{\widehat{\beta}_S}$ under the optimal exploration round: 
\begin{align}
  S^* = \argmin_{S\subseteq  [n]} \overline{\mse}^*_S|_{\widehat{\beta}_S} = \argmin_{S\subseteq  [n]} \left(\sqrt{k_1(S)}+\sqrt{c_{\ex}k_2(S)}\right)^2,\label{>}
\end{align}
where the right-hand side is assumed to have a unique minimizer. 
The uniform exploration policy $\pi^*$ with exploration round count $m = m_{S^*}$ (more precisely $(\lfloor m_{S^*}\rfloor)$) and exploitation action $a_{\ext}(S^*)$ achieves the smallest asymptotic loss among all uniform exploration policies; such a policy is referred to as a \emph{perfect uniform exploration policy}. 
Note that determination of this optimal policy requires oracle information, thus cannot be directly applied in practice. We will provide a solution for this in the next section.

\begin{Rem}
For every $S\subseteq  [n]$, one can verify that the numerator in \eqref{yuc} is bounded by $2(k_1(S) + c_\ex k_2(S))$. On the other hand, the MC estimator using the high-fidelity samples alone has MSE of a similar form with numerator $c_0\V[f(Y)]$ (Section \ref{lastt}). 
Taking the quotient of the two quantities yields
\begin{align*}
\frac{2(k_1(S) + c_\ex k_2(S))}{c_0\V[f(Y)]} & \stackrel{\eqref{yiyi}}{\leq} \frac{2[c_{\ext}(S)\beta_S^T\Sigma_S\beta_S+c_\ex(s+1)\sigma_S^2]}{c_0\V[f(Y)]}\\
&\stackrel{\eqref{1}}{\leq} 2\left[\frac{c_\ext(S)}{c_0} + \frac{c_\ex(s+1)}{c_0}\frac{\sigma_S^2}{\V[f(Y)]}\right].
\end{align*}
This can be unconditionally bounded by $2(n+1)(n+2)$ for all $S\subseteq  [n]$, and can be significantly smaller than $1$ if both
\begin{align} 
&\frac{c_\ext(S)}{c_0}\ll 1 &\frac{\sigma_S^2}{\V[f(Y)]}\ll 1. \label{beck}
\end{align} 
This implies that the best uniform exploration policies are at most a constant (which only depends on $n$) worse than the classical MC, and can be significantly better if \eqref{beck} is satisfied for at least one $S$ (which is often the case in practice). 
\end{Rem}

\subsection{An adaptive ETC algorithm}\label{adpp}

Finding a uniform exploration policy is impossible without oracle access to $\beta_S, \Sigma_S, x_S, \sigma^2_S$ and $\Lambda^{-1}_S$.
These quantities, despite being unknown, can be estimated at any particular time $t$ in exploration ($t>s+1$): 
\begin{align}
\widehat{\beta}_S(t) & = Z_S^\dagger X_\ex|_Y \ (\text{using the first $t$ samples})&\widehat{x}_S(t) &= \frac{1}{t}\sum_{\ell\in [t]}X_{\ex, \ell}|_S \label{e1}\\
\widehat{\Sigma}_S(t) &= \frac{1}{t-1}\sum_{\ell\in [t]}(X_{\ex, \ell}|_S-\widehat{x}_S(t))(X_{\ex, \ell}|_S-\widehat{x}_S(t))^T\nonumber\\
\widehat{\sigma}^2_S(t)& = \frac{1}{t-s-1}\sum_{\ell\in [t]}\left(X_{\ex,\ell}|_Y -X_{\ex,\ell}|_S\widehat{\beta}_S \right)^2&\widehat{\Lambda}_S^{-1}(t) & = \left(\frac{1}{t}Z_S^TZ_S\right)^{-1}\nonumber
\end{align} 
Note that $\widehat{\sigma}^2_S(t)$ is an unbiased estimator for $\sigma_S^2$ when the noise is Gaussian. 
Plugging \eqref{e1} into \eqref{dep} and ignoring the higher order term $\sigma_S^2\tr(\Sigma_S(Z_S^TZ_S)^{-1})$
yields the sample estimator for $\overline{\mse}_S|_{\widehat{\beta}_S}$ with exploration round $m$ at time $t$:
\begin{align}
\widehat{\overline{\mse}}_S|_{\widehat{\beta}_S}(m;t) = \frac{c_{\ext}(S)\widehat{\beta}_S^T(t)\widehat{\Sigma}_S(t)\widehat{\beta}_S(t)}{B-c_{\ex}m} + \frac{\widehat{\sigma}_S^2(t)\tr(\widehat{x}_S(t)\widehat{x}_S^T(t)\widehat{\Lambda}_S^{-1}(t))}{m}.\label{001}
\end{align}
We call \eqref{001} the \emph{empirical average conditional MSE} of model $S$, which we subsequently call the empirical MSE.
Correspondingly, we define more empirical estimates of quantities that were previously defined in terms of oracle information:
\begin{align}
k_{1,t}(S) &=c_{\ext}(S)\widehat{\beta}^T_S(t)\widehat{\Sigma}_S(t)\widehat{\beta}_S(t),\ \ \ \ \ \ \ k_{2,t}(S) = \widehat{\sigma}_S^2(t)\tr(\widehat{x}_S(t)\widehat{x}_S^T(t)\widehat{\Lambda}_S^{-1}(t)),\label{0001}\\
m_S(t) &= \frac{B}{c_{\ex}+\sqrt{\frac{c_{\ex}k_{1,t}(S)}{k_{2,t}(S)}}},\ \ \ \ \ \ \ \overline{\mse}_S^*|_{\widehat{\beta}_S}(t) = \frac{(\sqrt{k_{1,t}(S)}+\sqrt{c_{\ex}k_{2,t}(S)})^2}{B}.\nonumber
\end{align}

To mimic the idea of a perfect uniform exploration policy, we start by collecting the smallest necessary number of exploration samples to make the above estimation possible, which requires $|S|+2$ exploration rounds. 
For $t\geq |S|+2$, we use the collected exploration data to compute \eqref{e1}, \eqref{001}, \eqref{0001}, from which we can estimate the best empirical MSE for each model $S$ .
At this point, if $m_S(t)>t$, then more exploration is needed for model $S$, and the optimal empirical MSE is approximately $\overline{\mse}_S^*|_{\widehat{\beta}_S}(t) = \widehat{\overline{\mse}}_S|_{\widehat{\beta}_S}(m_S(t); t)$.
Otherwise, the best trade-off point has passed and stopping immediately yields the minimal empirical MSE, which equals $\widehat{\overline{\mse}}_S|_{\widehat{\beta}_S}(t; t)$. 
An algorithm could use the estimated optimal expected MSEs to determine which model is currently most favorable for exploitation. But another exploration round should transpire if $m_S(t)$ for this favorable model is larger than the current step $t$.

During implementation, we use the above procedure once $t$ exceeds a relatively small number, i.e., the maximum number of regressors plus one. 
The estimation error of the second term in \eqref{001} is of constant order and cannot be ignored at the beginning. 
If it becomes pathologically close to zero during the initial stages of the algorithm, exploration may terminate after only a few rounds.
To combat this deficiency, we employ a common regularization in bandit learning:
For every $S$, we additively augment $k_{2,t}(S)$ with a small regularization parameter $\alpha_t$, which decays to $0$ as $t\to\infty$. 
For large $B$, adding $\alpha_t$ encourages the algorithm to explore at the beginning, but the encouragement becomes negligible as $\alpha_t\to 0$.  
Putting the ideas together yields the following adaptive ETC (AETC) algorithm for multifidelity approximation:
\medskip

\begin{algorithm}
\hspace*{\algorithmicindent} \textbf{Input}: $B$: total budget, $c_i$: cost parameters, $\alpha_t\downarrow 0$: regularization parameters \\
    \hspace*{\algorithmicindent} \textbf{Output}: $(\pi_t)_t$
 \begin{algorithmic}[1]
\STATE compute the maximum exploration round $M = \lfloor B/c_{\ex}\rfloor$
   \FOR{$t \in [n+2]$}{ 
 \STATE $\pi_t = a_{\ex}([n])$
 }
\ENDFOR
 \WHILE{$n+2\leq t\leq M$}
 \FOR{$S\subseteq  [n]$}{
 \STATE{compute $k_{1,t}(S), k_{2,t}(S)$ using \eqref{e1}, \eqref{0001}, and set $k_{2,t}(S) \gets k_{2,t}(S) +\alpha_t$}
 \STATE{compute $m_S(t)$ using \eqref{0001}}
 \STATE{compute the optimal expected MSE using \eqref{001}: $h_S(t) =\widehat{\overline{\mse}}_S|_{\widehat{\beta}_S}(m_S(t)\vee t; t)$}
}
 \ENDFOR
 \STATE {find the optimal model $S^*(t) = \argmin_{S\subseteq  [n]}h_S(t)$}
\IF{$m_{S^*(t)}(t)>t$}
        \STATE { $\pi_{t+1} = a_{\ex}([n])$ and $t \gets t + 1$}
    \ELSE
        \STATE { $\pi_t = a_{\ext}(S^*(t))$ and $t \gets M+1$}
     \ENDIF
 \ENDWHILE
 \end{algorithmic}
\caption{AETC algorithm for multifidelity approximation (single-valued case)} 
\label{alg:aETC}
\end{algorithm}

\begin{Rem}
Algorithm \ref{alg:aETC} is an adaptive version of the ETC algorithm (Algorithm \ref{alg:bETC}). To better comprehend Algorithm \ref{alg:aETC}, and in particular the policy actions $\pi$ that it implicitly defines, note that the first `for' loop (step 2-4) involves collecting a minimum number of samples to make the parameters that will be used estimable. The `while' loop (step 5-17) determines the exploration rate adaptively based on the estimated optimal mean-squared errors. (Exploration/exploitation rate refers to the number of samples used for exploration/exploitation in AETC.) In particular, the inner `for' loop (step 6-10) computes for each $S\subseteq  [n]$ the current expected optimal mean-squared error, and step 11 selects the best model based on the estimated values. The `if -- else' procedure (step 12-16) then decides if more exploration is needed based on the information of the selected best model. 
\end{Rem}

\subsection{AETC consistency and optimality}
Algorithm \ref{alg:aETC} ensures that both exploration and exploitation go to infinity as $B\to\infty$. (See Appendix \ref{ap:2}.)
As a consequence, almost surely, the model chosen by Algorithm \ref{alg:aETC} for exploitation converges to $S^*$ and the corresponding exploration is asymptotically optimal.   

\begin{Th}\label{thm:AETC}
Let $m(B)$ be the exploration round chosen by Algorithm \ref{alg:aETC} under budget $B$, and $S(B)$ be the model for exploitation, i.e., $S(B) = S^*(m(B))$.  
Suppose \eqref{...} holds uniformly for all $S\subseteq  [n]$, and $\lim_{t\to\infty}\alpha_t = 0$. 
Then, with probability $1$, 
\begin{subequations}
\begin{align}\label{rr1}
  \lim_{B\to\infty}\frac{m(B)}{m_{S^*}} &= 1, \\
  \lim_{B\to\infty}S(B) &= S^*,\label{rr2}
\end{align}
\end{subequations}
where $S^*$ and $m_{S^*}$ are the model choice and exploration round count given oracle information defined in \eqref{>} and \eqref{mode}, respectively.
\end{Th}

\begin{proof}
See Appendix \ref{ap:2}. 
\end{proof}

Theorem~\ref{thm:AETC} concludes that as the budget goes to infinity, both the exploration round count and the chosen exploitation model will converge to the optimal ones given by the perfect uniform exploration policies for almost every realization. However, this does not imply that the loss of the corresponding policy $\pi$ asymptotically matches the loss of the perfect uniform exploration policy. Indeed, the formula defining the loss in \eqref{dep} assumes $m$ is deterministic and cannot be applied when $m=m(B)$ is chosen in an adaptive (and random) manner. 

Selecting $\alpha_t$ is often considered an art in bandit learning, and for us is mostly used for theoretical analysis. 
In practice, we observe that it is sufficient to choose $\alpha_t$ as a sequence with exponential decay in $t$.

\begin{Rem}\label{rrrm}
  Let $N(B)$ denote the number of affordable samples for exploiting under budget $B$. 
\eqref{rr1} implies that both $m(B)$ and $N(B)$ diverge as $B$ goes to infinity.  
In addition, one can see from the proof that $S^*(t) = S^*$ for sufficiently large $t$ almost surely. 
These combined with the strong law of large numbers imply that the estimators produced by Algorithm \ref{alg:aETC} are almost surely consistent.   
\end{Rem}

\section{Vector-valued responses and efficient estimation}\label{sec:vec}

\subsection{Vector-valued high-fidelity models}
We now generalize the results in the previous section to the case of vector-valued responses, i.e., $k_0>1$.  
Let
\begin{align*}
&f(Y) = (f^{(1)}(Y), \cdots, f^{(k_0)}(Y))^T\in\R^{k_0}
\end{align*}
denote the response vector. 
The linear model assumption in \eqref{1} is modified as follows:
\begin{align}
&f(Y) = \beta_SX_S + \e_S, \label{newass}
\end{align}
where $\e_S = (\e^{(1)}_{S}, \cdots, \e^{(k_0)}_{S})^T\sim (0, \Gamma_S)$
is a centered multivariate sub-Gaussian distribution with covariance matrix $\Gamma_S$, and 
$\beta_S = (\beta^{(1)}_{S}, \cdots, \beta^{(k_0)}_{S})^T \in \R^{k_0 \times (s+1)}$ 
is the coefficient matrix for model $S$. 
For fixed $S$, one may use \eqref{lscoef} to estimate the coefficients $\beta_S^{(j)}$ by $\widehat{\beta}_S^{(j)}$ for $j\in [k_0]$, and the corresponding LRMC estimators can be built similarly as in \eqref{LRMC}.  
To generalize the MSE, we consider a common class of quadratic risk functionals to measure the quality of the LRMC estimators. 
Let $Q\in\R^{q\times l}$, and let $\widehat{\theta}$ be an estimator for $\theta\in\R^{l}$ where $q$ is arbitrary. 
The \emph{$Q$-risk} of $\widehat{\theta}$ is defined as 
\begin{align}
\risk(\widehat{\theta}) = \E\left[\left\|Q(\widehat{\theta}-\theta)\right\|_2^2\right] = \E\left[(\widehat{\theta}-\theta)^TQ^TQ(\widehat{\theta}-\theta)\right],
\end{align}
where we suppress the notational dependence of $\risk$ on $Q$. We use the following conditional $Q$-risk on the estimated coefficient matrix $\widehat{\beta}_S$ in place of \eqref{def:mse}:

\begin{align}
\risk_S|_{\widehat{\beta}_S}: = \risk(\lrmc_S)|_{\widehat{\beta}_S} = \E\left[\left\|Q(\lrmc_S-\E[f(Y)])\right\|_2^2|\widehat{\beta}_S\right]\label{newreg}
\end{align}
where above $Q \in \R^{q \times k_0}$ for arbitrary $q$.
A similar result as Theorem \ref{thm:conss} can be obtained for the LRMC estimator under \eqref{newreg}:

\begin{Th}\label{thm:consv}
Under the same conditions as in Theorem \ref{thm:conss} and assumption \eqref{newass},
the following result holds: For large $m$, with probability at least $1-m^{-2}$,
\begin{align}
\risk_S|_{\widehat{\beta}_S} \lesssim \frac{1}{N_S}\tr(\Sigma_S\beta_S^TQ^TQ\beta_S) + (s+1)\tr(Q\Gamma_SQ^T)\frac{\log m}{m},\label{sddd}
\end{align}
where the implicit constant in $\lesssim$ is independent of $m$.  
\end{Th}

Other results in the previous section (including Algorithm \ref{alg:aETC} and Theorem \ref{thm:AETC}) can also be generalized to the vector-valued case;
since the ideas are similar, we do not restate them here. 
An extended discussion of the results and the proof of Theorem \ref{thm:consv} can be found in Appendix \ref{ap:3}.
The corresponding algorithm is Algorithm \ref{alg:aETC_mul}.


\subsection{Efficient estimation}
Estimation for vector-valued responses is more challenging due to the presence of high-dimensional parameters.  
Nevertheless, we are only interested in some functionals of the high-dimensional parameters, i.e., 
$\tr(Q\Gamma_SQ^T)$ for instance instead of $\Gamma_S$ itself. For these functionals, the plug-in estimators are relatively accurate for $t\gtrsim s+1$. 
See Appendix \ref{ap:3} for a discussion that theoretically establishes this point.


\section{Numerical simulations}\label{num}

In this section we demonstrate performance of the AETC algorithms (Algorithms \ref{alg:aETC} and \ref{alg:aETC_mul}). 
The regularization parameters $\alpha_t$ are set as $\alpha_t = 4^{-t}$ in all simulations, except in one case (first plot in Figure \ref{fig:more2}) where we investigate how the accuracy of AETC depends on $\alpha_t$.
We will utilize four methods to solve multifidelity problems: 
\begin{itemize}[leftmargin=2.0cm]
  \item[(MC)] Classical Monte Carlo, where the entire budget is expended over the high-fidelity model.
  \item[(MFMC)] The multifidelity Monte Carlo procedure from \cite{Peherstorfer_2016}, which is provided with oracle correlation information in Table \ref{145} and Table \ref{146}.
  \item[(AETC)] Algorithm \ref{alg:aETC} (for scalar responses) or \ref{alg:aETC_mul} (for vector responses).
  \item[(AETC-re)] Algorithm \ref{alg:aETC} or \ref{alg:aETC_mul}, but we recycle samples from exploration when computing exploitation estimates. See Remark \ref{rem:recycle}.
\end{itemize}
To evaluate results, we compute and report an empirical mean-squared error over 500 samples. Since both the AETC and AETC-re produce random estimators due to the exploration step, the experiment (including both exploration and exploitation) is repeated $200$ times with the $0.05$-$0.50$-$0.95$-quantiles recorded.

\subsection{Multifidelity finite element approximation for parametric PDEs}\label{ppde}
In the first example, we investigate the performance of our approach on multifidelity parametric solutions of linear elastic structures. We consider two scenarios associated with different geometries of the domain, namely \emph{square} and \emph{L-shape} structures. The geometry, boundary conditions, and loading in these structures are shown in Figure~\ref{fig:struct}.
\begin{figure}[htbp]
\begin{center}
\includegraphics[width=0.45\textwidth]{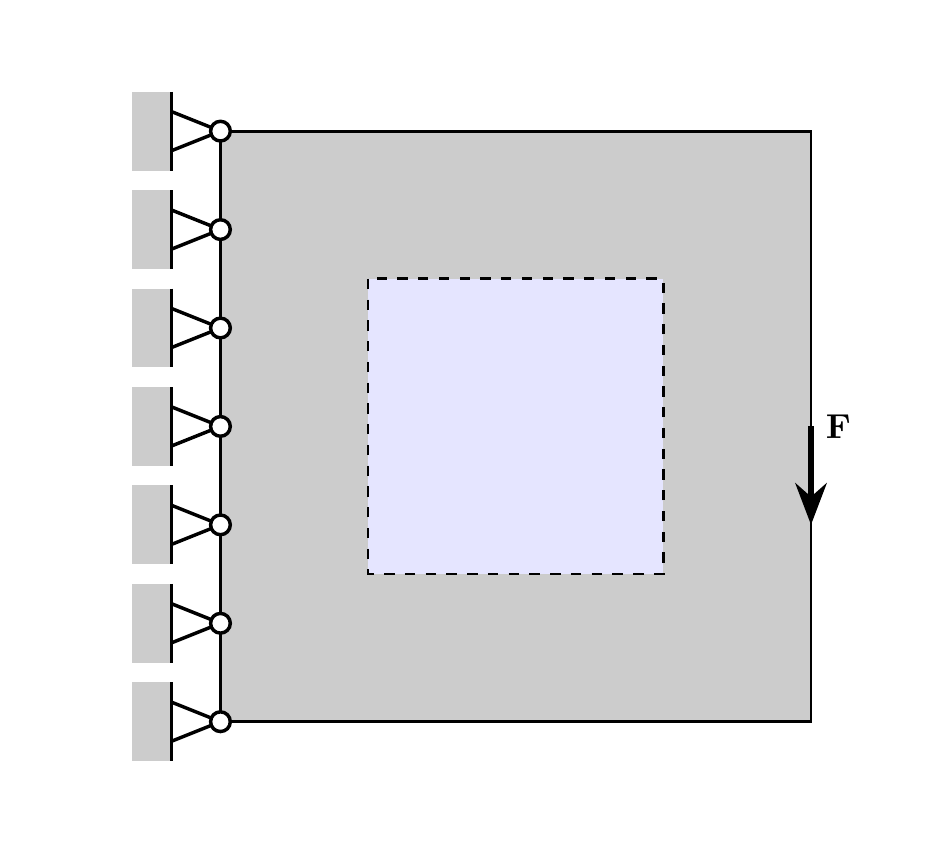}\hspace{1cm}
\includegraphics[width=0.45\textwidth]{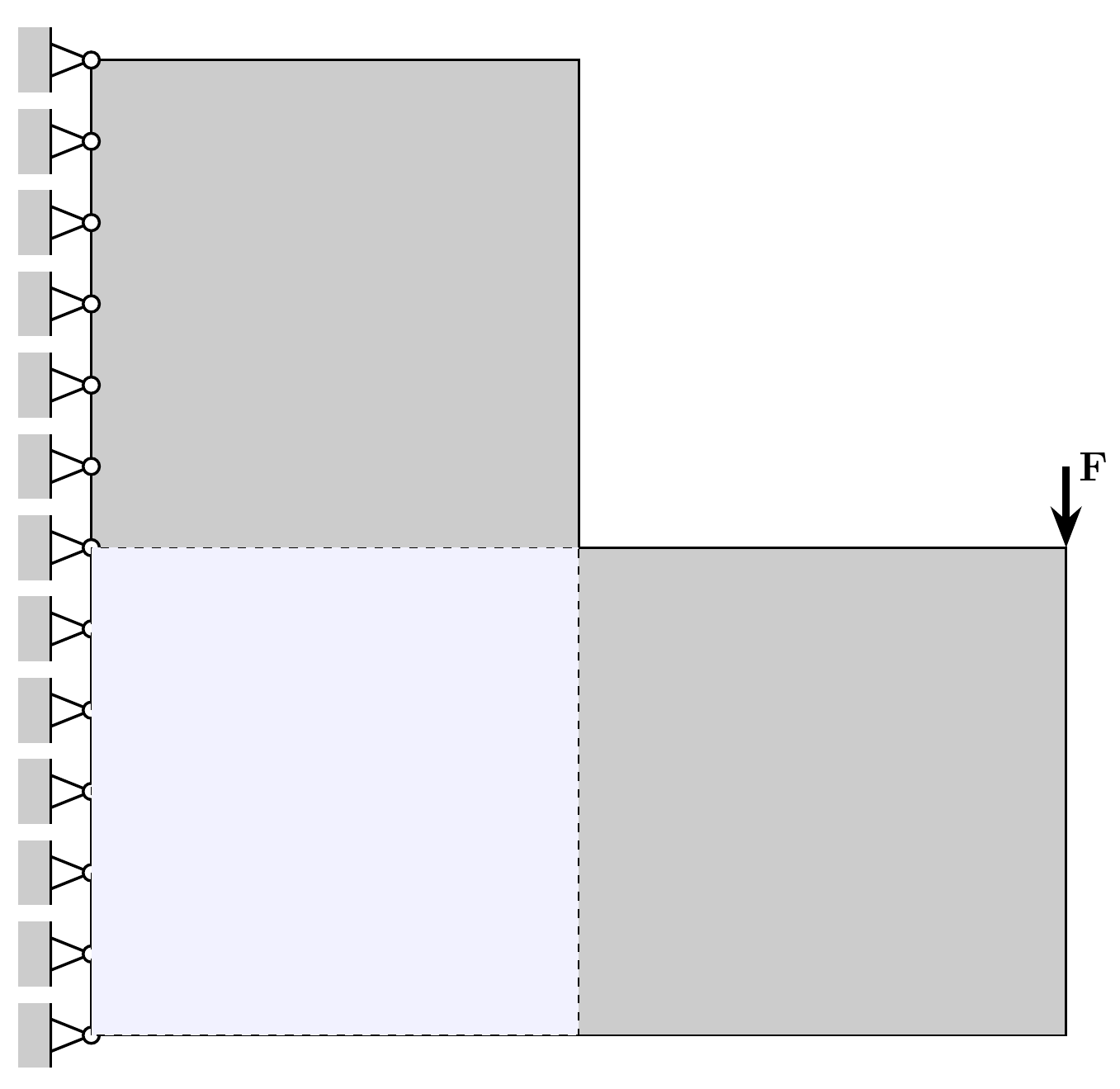}
\caption{\small Geometry, boundary conditions, and loading for two linear elastic structures, the square and the L-shape. The regions in which the vector-valued quantities are investigated are highlighted with the light blue color.}\label{fig:struct}
  \end{center}
\end{figure}

The model is an elliptic PDE that governs motion in linear elasticity. In a parametric setting, we introduce the parametric/stochastic version of this governing equation. 
In a bounded spatial domain $D$ with boundary $\partial D$, the parametric elliptic PDE is 
\begin{equation}\label{elliptic_pde_stochastic}
     \begin{cases}
        \nabla \cdot (E(\bm p,\bm x) \nabla u(\bm p,\bm x) ) = f(\bm x) & \forall (\bm p,\bm x) \in \mathcal{P} \times D \\
        u(\bm p,\bm x)  = 0  & \forall (\bm p,\bm x) \in \mathcal{P} \times \partial D\\
        \end{cases}
\end{equation}
where $E$, the elasticity matrix dependent on state variables $\bm x$, is parameterized with the random variables $\bm p$, and $f$ is the forcing function. We assume that $\bm p \in \mathcal{P}$ is a $d$-dimensional random variable with independent components $\{p_i\}_{i=1}^d$.
The solution of this parametric PDE is the displacement $u \equiv u(\bm p, \bm x):\mathcal{P} \times D \rightarrow \R$. The displacement is used to compute the scalar quantity of interest, the \emph{compliance}, which is the measure of elastic energy absorbed in the structure as a result of loading,
\begin{align}
  \C \coloneqq \int_D \nabla u(\bm p, \bm x)^T E(\bm p, \bm x) \nabla u(\bm p, \bm x) d \bm x
\end{align}
The linear elastic structure is subjected to plane stress conditions and the uncertainty is considered in the material properties, namely the elastic modulus, which manifests in the model through the random parameters $\bm p$. 
To model the uncertainty, we consider a random field for the elastic modulus via a Karhunen-Lo\'{e}ve (KL) expansion:
\begin{equation}\label{KL_numerical}
E(\bm p,\bm x) = E_0(\bm x) + \delta \left [ \sum_{i=1}^{d} \sqrt{\lambda_i} E_i(\bm x) p_i \right ]
\end{equation}
where $\delta=0.5, E_0(\bm x)=1$ are constants, and the random variables $p_i$ are uniformly distributed on $[-1,1]$ and the eigenvalues $\lambda_i$ and basis functions $E_i$ are taken from the analytical expressions for the eigenpairs of an exponential kernel on $D=[0,1]^2$. The Poisson's ratio is $\nu=0.3$, and we initially use $d = 4$ parameters. We solve the partial differential equation \eqref{elliptic_pde_stochastic} for each fixed $\bm p$ via the finite element method with standard bilinear square isotropic finite elements on a rectangular mesh. 

In this example, we form a multifidelity hierarchy through mesh coarsening: Consider $n=7$ mesh resolutions with mesh sizes $h = \{1/(2^{8-L})\}_{L=1}^{7}$ where $L$ denotes the level. The mesh associated with $L=1$ yields the most accurate model (highest fidelity), which is taken as the high-fidelity model in our experiments. To utilize the notation presented earlier in this article, our potential low-fidelity regressors is formed from the compliance computed form various discretizations,
\begin{align*}
  X^{(i)} &= \C^{(i)}\in\R, & i&\in [6],
\end{align*} 
where $\C^{(i)}$ is the compliance of the solution computed using a solver with mesh size $h = 2^{i-7}$.
We provide more details about the discretization and the uncertainty model of this section in Appendix \ref{app:num-setup-1}.

The cost for each model is the computation time, which we take to be inversely proportional to the mesh size squared, i.e., $h^2$. (This corresponds to employing a linear solver of optimal linear complexity.)
We normalize cost so that the model with the lowest fidelity has unit cost. 
The total budget $B$ ranges from $10^5$ to $4\times 10^5$, and our simulations increment the budget over this range by $0.5\times 10^5$ units.
The budget range considered here is sufficient to communicate relevant results.
We refer to \cite{xu2021budget} for a further study of the square domain over a larger budget range of $10^5$ to $10^7$.

\subsubsection{Scalar high-fidelity output}\label{revision1}
In the first experiment, the response variable is chosen as the compliance of the high-fidelity model, i.e., 
\begin{align*}
f(Y) = \C^{(0)}\in\R.
\end{align*} 
Therefore, for AETC we use Algorithm \ref{alg:aETC}. 
The oracle statistics of $f(Y)$ and $X^{(i)}$ are computed over $50000$ independent samples, and shown in illustrated in Table \ref{145}.

\medskip
\begin{table}[htbp]
\begin{center}
\small
\begin{tabular}{l*{6}{c}r}
Models              & $f(Y)$ &$X^{(1)}$ & $X^{(2)}$ & $X^{(3)}$ & $X^{(4)}$& $X^{(5)}$  & $X^{(6)}$ \\
\hline
$\text{Corr}(\cdot, f(Y))$ & 1 & 0.998 & 0.992 & 0.976 & 0.940 & 0.841 & -0.146\\
Mean            & 9.641 & 9.197 & 8.749 & 8.287 &  7.782 & 7.141 & 6.160\\
Standard deviation            & 0.127 & 0.113 & 0.099 & 0.086 &  0.072 & 0.052 & 0.027\\
Cost            & 4096& 1024 & 256 & 64 & 16 &  4 & 1 \\
\end{tabular} 
\bigskip

\begin{tabular}{l*{6}{c}r}
Models              & $f(Y)$ &$X^{(1)}$ & $X^{(2)}$ & $X^{(3)}$ & $X^{(4)}$& $X^{(5)}$  & $X^{(6)}$ \\
\hline
$\text{Corr}(\cdot, f(Y))$ & 1 & 0.999 & 0.995 & 0.980 & 0.932 & 0.733 & -0.344\\
Mean            & 25.940 & 24.256 & 22.902 & 21.180 &  19.054 & 16.072 & 12.275\\
Standard deviation            & 0.290 & 0.242 & 0.195 & 0.149 &  0.104 & 0.061 & 0.070\\
Cost            & 4096& 1024 & 256 & 64 & 16 &  4 & 1 \\
\end{tabular} 
\caption{\small Oracle information of $f(Y)$ and $X^{(i)}$ (approximated with $3$ digits) in the case of the square domain (\textbf{top}) and the L-shape domain (\textbf{bottom}).}\label{145}
\end{center}
\end{table}

In both cases, the expensive models are more correlated with the high-fidelity model than the cheaper ones. 
Such hierarchical structure is often a necessary assumption for the implementation of multifidelity methods, the MFMC, for instance. 

Accuracy results for various multifidelity procedures are shown in the first two plots in Figure \ref{fig:1}. As the budget goes to infinity, the model selected by AETC converges to the minimizer given by \eqref{>}, which can be explicitly computed using oracle information. 
Particularly, solutions to \eqref{>} in the case of the square domain and the L-shape domain are respectively as follows:
\begin{itemize}
\item Square domain: $f(Y)\sim X^{(4)}+X^{(5)}+X^{(6)}+\text{intercept}$ 
\item L-shape domain: $f(Y)\sim X^{(3)}+X^{(4)}+X^{(5)}+X^{(6)}+\text{intercept}$. 
\end{itemize}
The regression coefficients for the square domain and the L-shape domain are  6.151 ($X^{(4)}$), -6.509 ($X^{(5)}$), 1.444 ($X^{(6)}$), -0.640 (intercept) and 5.537 ($X^{(3)}$), -6.606 ($X^{(4)}$), 2.559 ($X^{(5)}$), -0.437 ($X^{(6)}$), -1.210 (intercept), respectively.
Note that there is a canceling effect (change in sign of coefficients) for the more expensive regressors due to high correlations.
Neither limiting model above will be used by the MFMC estimator as the relationship between cost and correlation does not satisfy the assumption \cite[condition (20)]{Peherstorfer_2016}. For every budget level $B$ under our test, we compute the empirical probability (as a percentage) that the limiting model is selected by AETC and plot this in the rightmost panel in Figure \ref{fig:1}. 
\begin{figure}[htbp]
\begin{center}
\includegraphics[width=0.32\textwidth]{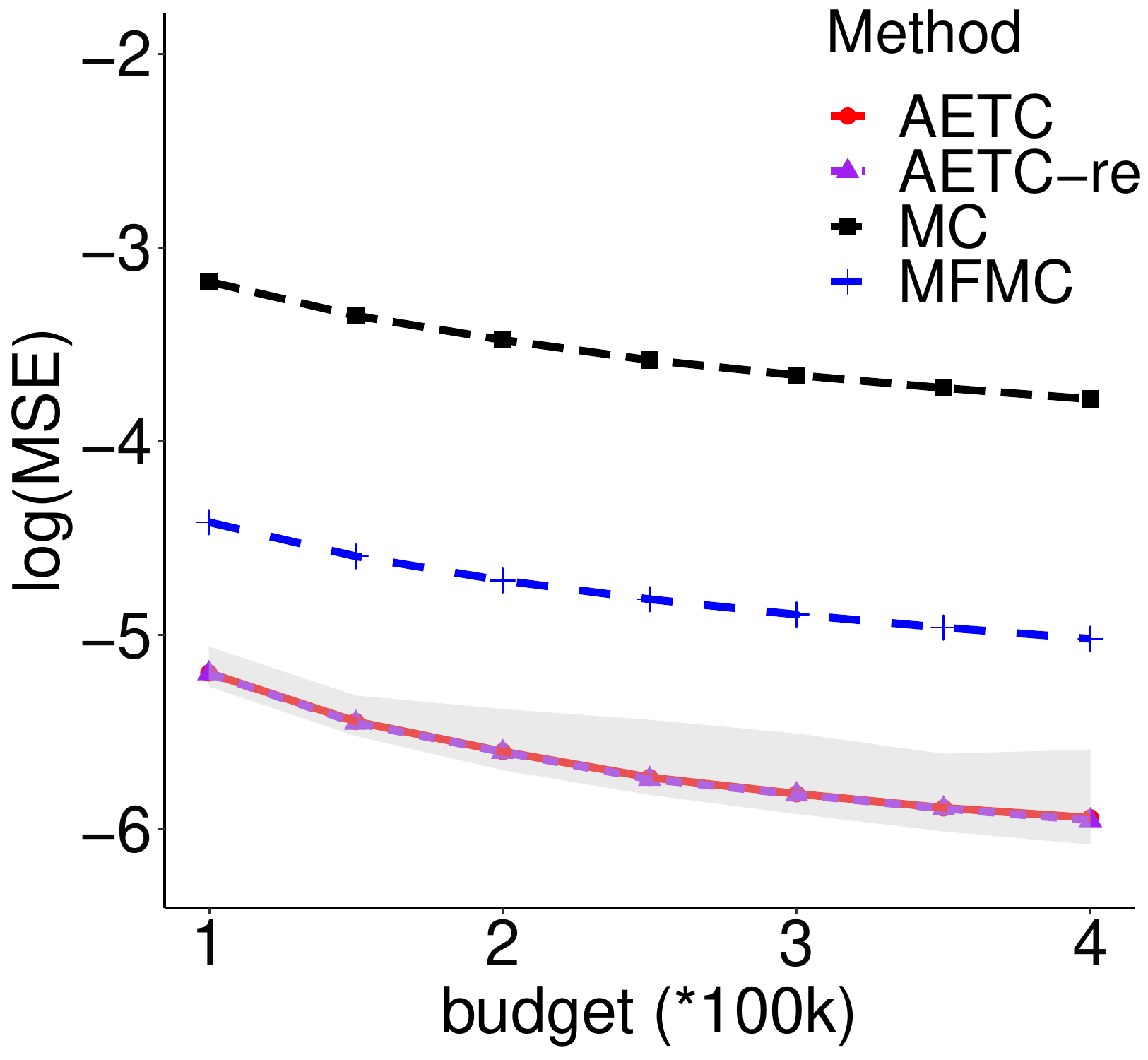}\ \ 
\includegraphics[width=0.32\textwidth]{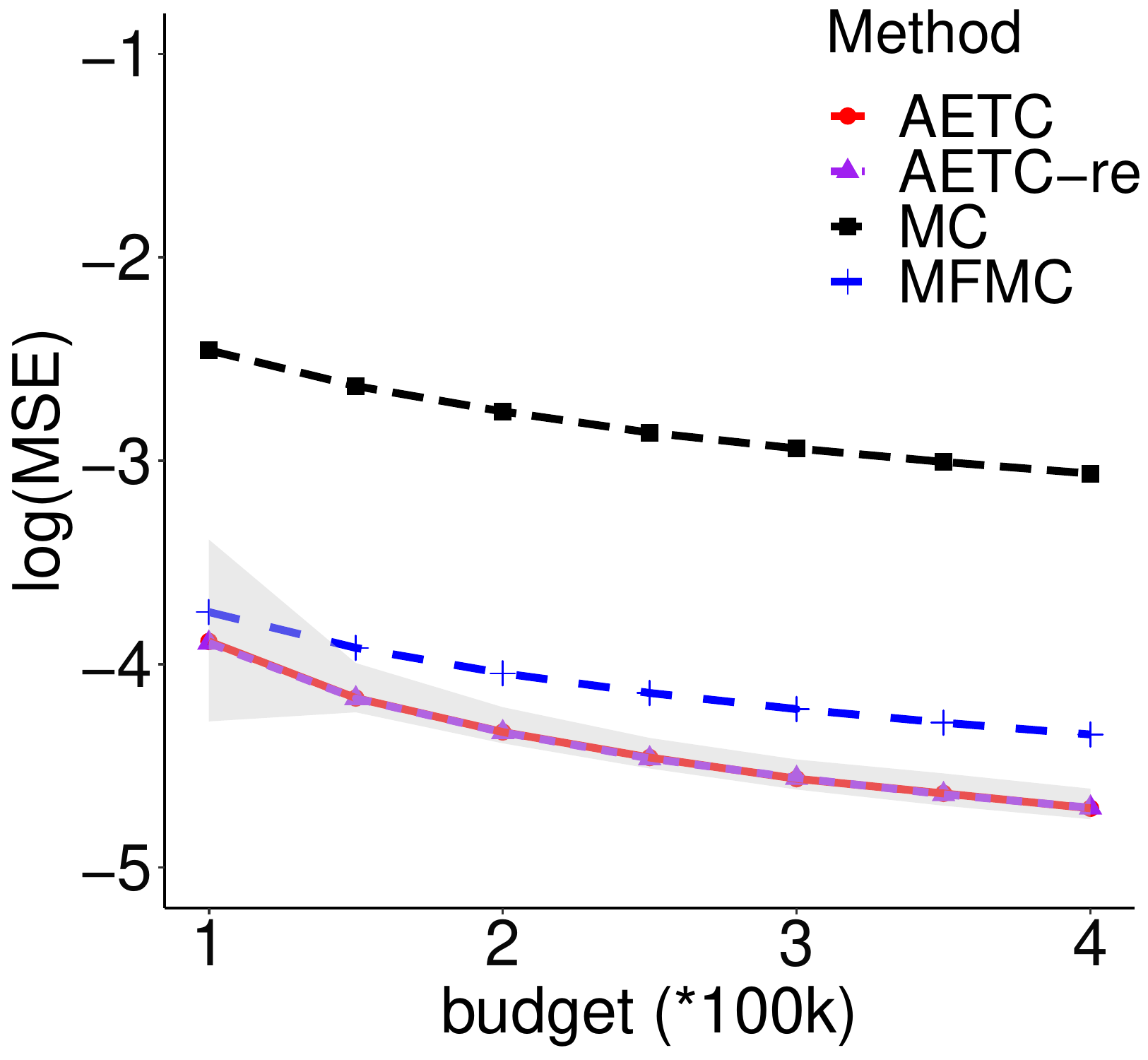}\ \ 
\includegraphics[width=0.32\textwidth]{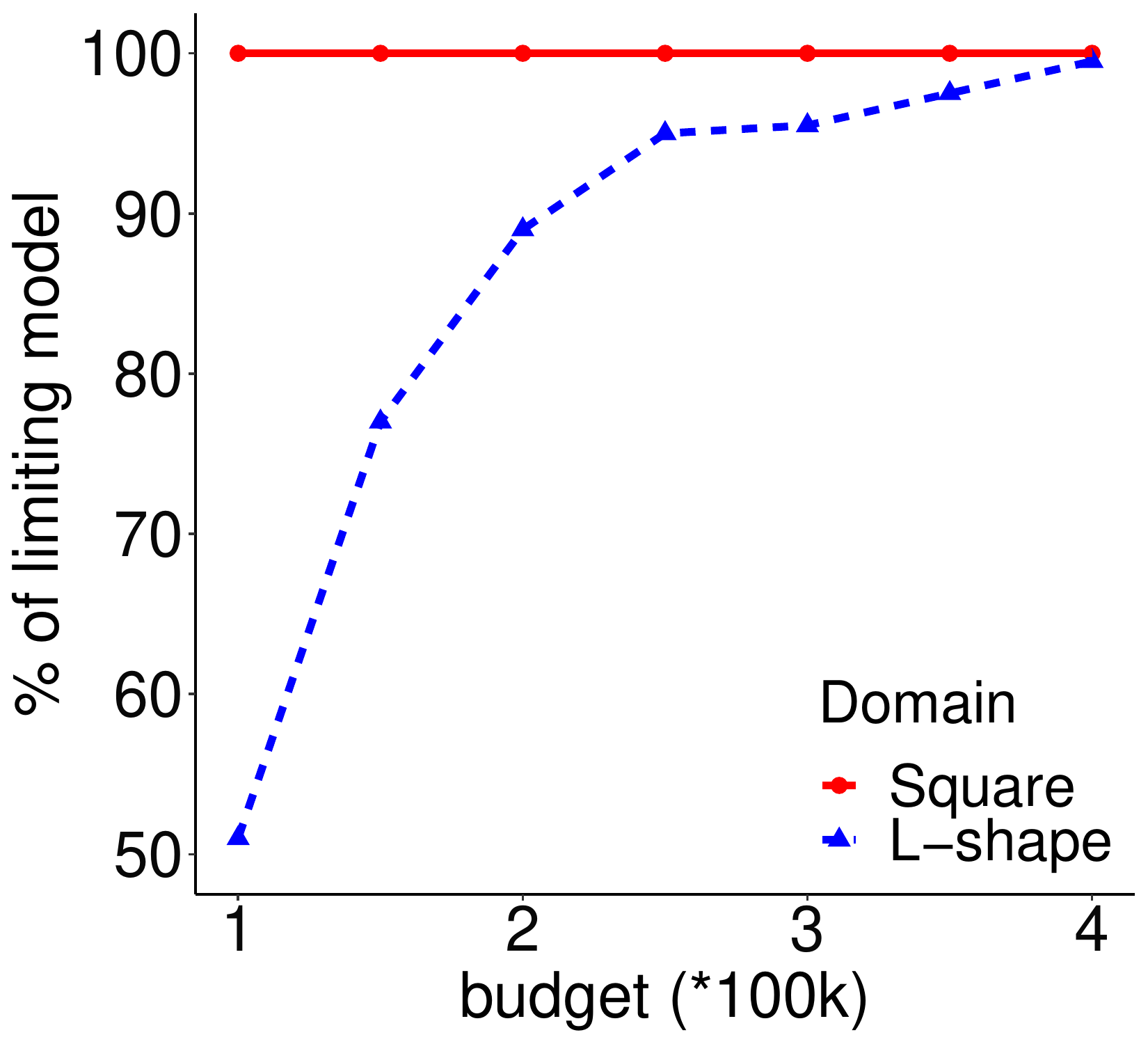}
\caption{\small Comparison of the ($\log_{10}$) mean-squared error of the LRMC estimator given by the AETC algorithm, the AETC algorithm reusing the exploration samples (AETC-re),
the MC estimator, and the MFMC estimator as the total budget increases from $10^5$ to $4\times 10^5$ in the case of the square domain (\textbf{left}) and the L-shape domain (\textbf{middle}).  
The $0.05$-$0.50$-$0.95$-quantiles are plotted for the LRMC estimator to measure its uncertainty. 
  We also compute the probability (plotted as a percentage) that the AETC algorithm selects the limiting model given by \eqref{>} (\textbf{right}).}
  \label{fig:1}
  \end{center}
\end{figure}

Figure \ref{fig:1} shows that for both domain geometries, the AETC algorithm outperforms both the MC and the MFMC by a notable margin, even though the latter has access to oracle correlation statistics that are not provided to AETC.
Little difference between AETC and AETC-re is visible, which is not surprising since the number of the exploitation samples is much larger than the number of exploration rounds.
The mean-squared error of AETC is smaller in the case of the square domain than in the L-shape domain under the same budget, for which a possible explanation is that the variance of the surrogate models in the former is smaller than in the latter (Table \ref{145}), making the exploration procedure more efficient in the square domain case. 
Finally, note that as the budget goes to infinity, the mean-squared error of the AETC algorithm decays to zero, and the frequency of the AETC exploiting the limiting model given by \eqref{>} converges to $1$, verifying the asymptotic results in Theorem \ref{thm:AETC}.
The statistics of the learned regression coefficients of the limiting model in both cases when $B = 4\times 10^5$ are provided in Figure \ref{fig:beta}.

In this example, despite a deterministic nonlinear relationship between $f(Y)$ and $X^{(i)}$, the conditional expectation $\E[f(Y)|X_S]$ does seem to admit a reasonable approximation using an appropriately constructed linear combinations of the low-fidelity model outputs in the selected subsets by AETC.
In particular, it can be seen from Figure \ref{fig:beta} that the regression coefficients learned by AETC on the limiting model, which are estimated based on less than $25$ random exploration samples in both cases, are concentrated around their true values computed using the oracle dataset (50000 independent samples). The concentration effect is more prominent in the square domain case due to a larger exploration rate. 

\begin{figure}[htbp]
\begin{center}
\includegraphics[width=0.38\textwidth]{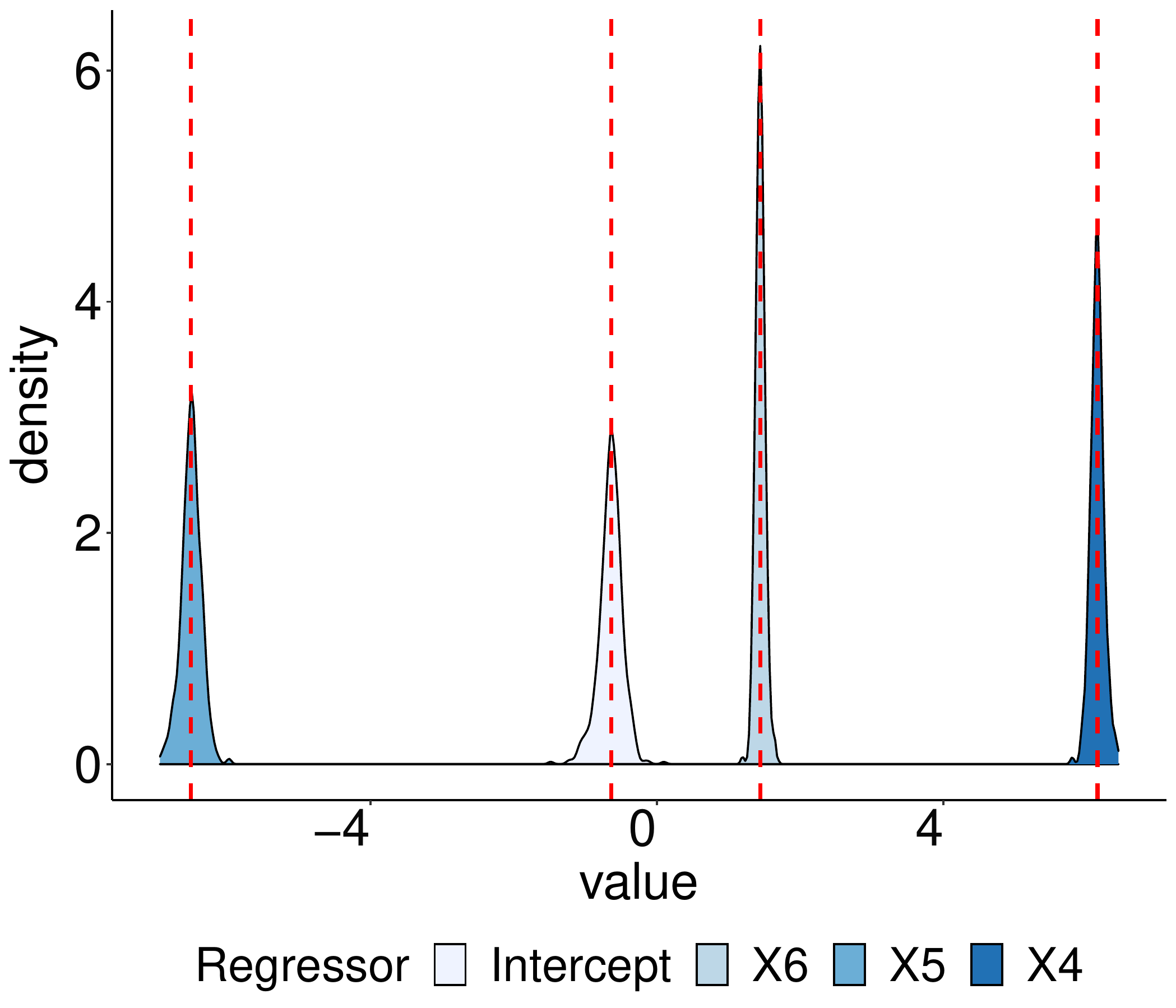}\hspace{1.5 cm}
\includegraphics[width=0.38\textwidth]{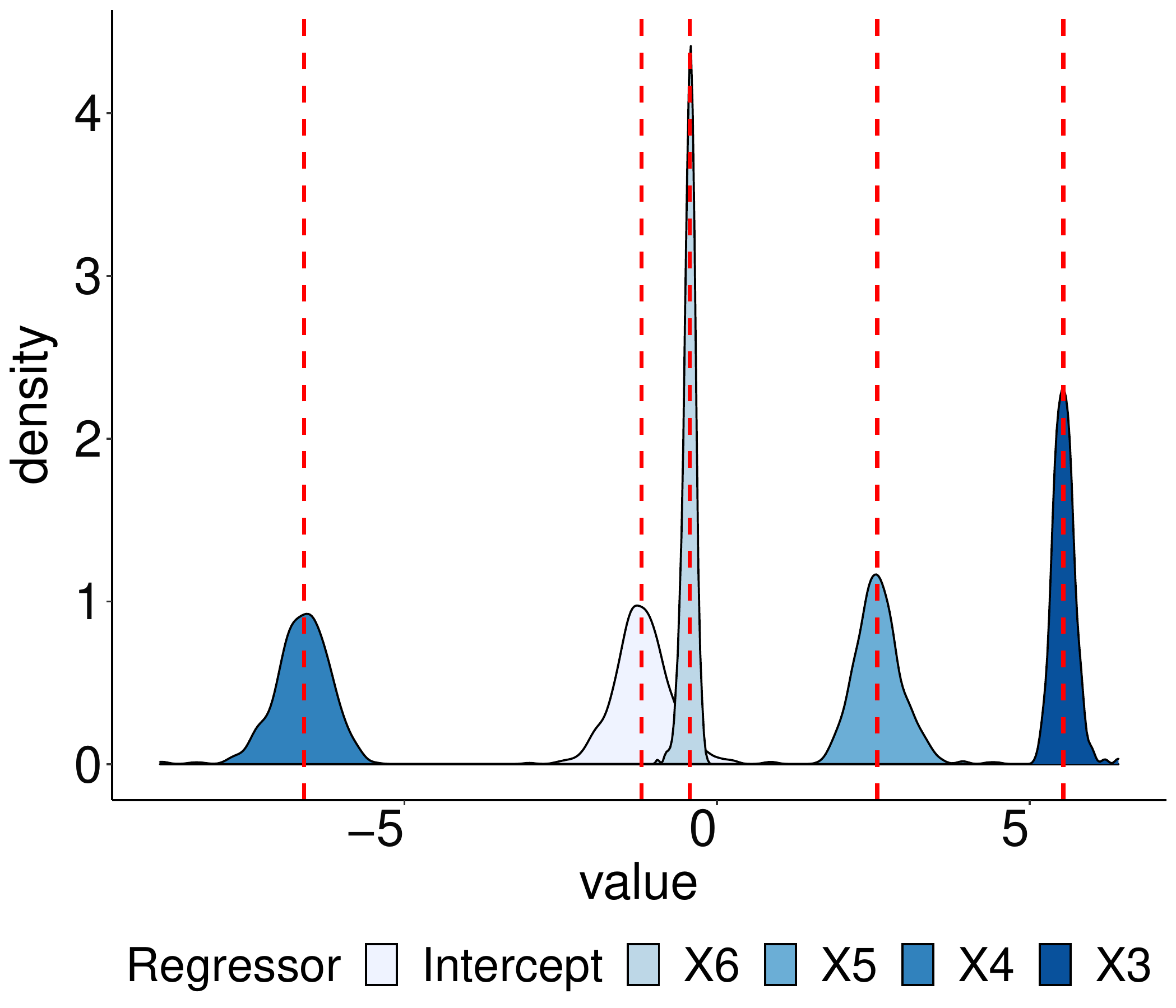}
\caption{\small Densities of the estimated coefficients of the limiting model learned by AETC under budget $B = 4\times 10^5$ in $200$ experiments in the case of the square domain (\textbf{left}) and the L-shape domain (\textbf{right}). Red dashed lines correspond to the true values of the coefficients estimated using the oracle dataset.}
  \label{fig:beta}
  \end{center}
\end{figure}

  
In the rest of this section, we focus for simplicity on the square domain case to further investigate the performance of AETC. 
We will take a much larger KL truncation dimension $d= 20$ to inject more uncertainty into the models. 
To better demonstrate the behavior of the algorithms in the asymptotic regime, we use an increased budget range compared to the previous simulation, i.e., the total budget $B$ ranges from $2\times 10^5$ to $12\times 10^5$, with budget increment $2\times 10^5$ units.
The oracle statistics of $f(Y)$ and $X^{(i)}$ are computed over $50000$ independent samples, and shown in Table \ref{146}:

\medskip
\begin{table}[htbp]
\begin{center}
\small
\begin{tabular}{l*{6}{c}r}
Models              & $f(Y)$ &$X^{(1)}$ & $X^{(2)}$ & $X^{(3)}$ & $X^{(4)}$& $X^{(5)}$  & $X^{(6)}$ \\
\hline
$\text{Corr}(\cdot, f(Y))$ & 1 & 0.998 & 0.989 & 0.966 & 0.914 & 0.789 & -0.135\\
Mean            & 9.645 & 9.200 & 8.752 & 8.290 &  7.784 & 7.142 & 6.160\\
Standard deviation            & 0.137 & 0.120 & 0.104 & 0.089 &  0.073 & 0.052 & 0.027\\
Cost            & 4096& 1024 & 256 & 64 & 16 &  4 & 1 \\
\end{tabular} 
\end{center}
\caption{\small Oracle information of $f(Y)$ and $X^{(i)}$ (approximated with $3$ digits) in the case of the square domain with randomness dimension $d = 20$.}\label{146}
\end{table}

We first conduct a standard analysis for AETC, including an average mean-squared error comparison with the other two methods and the convergence to the limiting model.
As before, we use oracle statistics to compute the limiting model via \eqref{>}, which in this case is given by $f(Y)\sim X^{(3)}+X^{(4)}+X^{(5)}+X^{(6)}+\text{intercept}$. Note that this model is different from the one in the previous example where $d = 4$. A plausible explanation, as can be inspected from Table \ref{145} and \ref{146}, is that the low-fidelity models are less correlated with the high-fidelity model when $d$ is large, thus AETC leans toward utilizing more low-fidelity models to achieve equally accurate predictions.
Moreover, we include two additional plots to summarize the statistics (median) of the exploration/exploitation rate in AETC as opposed to the optimal rates \eqref{yuc} computed using the oracle statistics, as well as a plot of the estimation given by AETC along $100$ trajectories at different budget levels. The results are reported in Figure \ref{fig:more1}.

\begin{figure}[htbp]
\begin{center}
\includegraphics[width=0.32\textwidth]{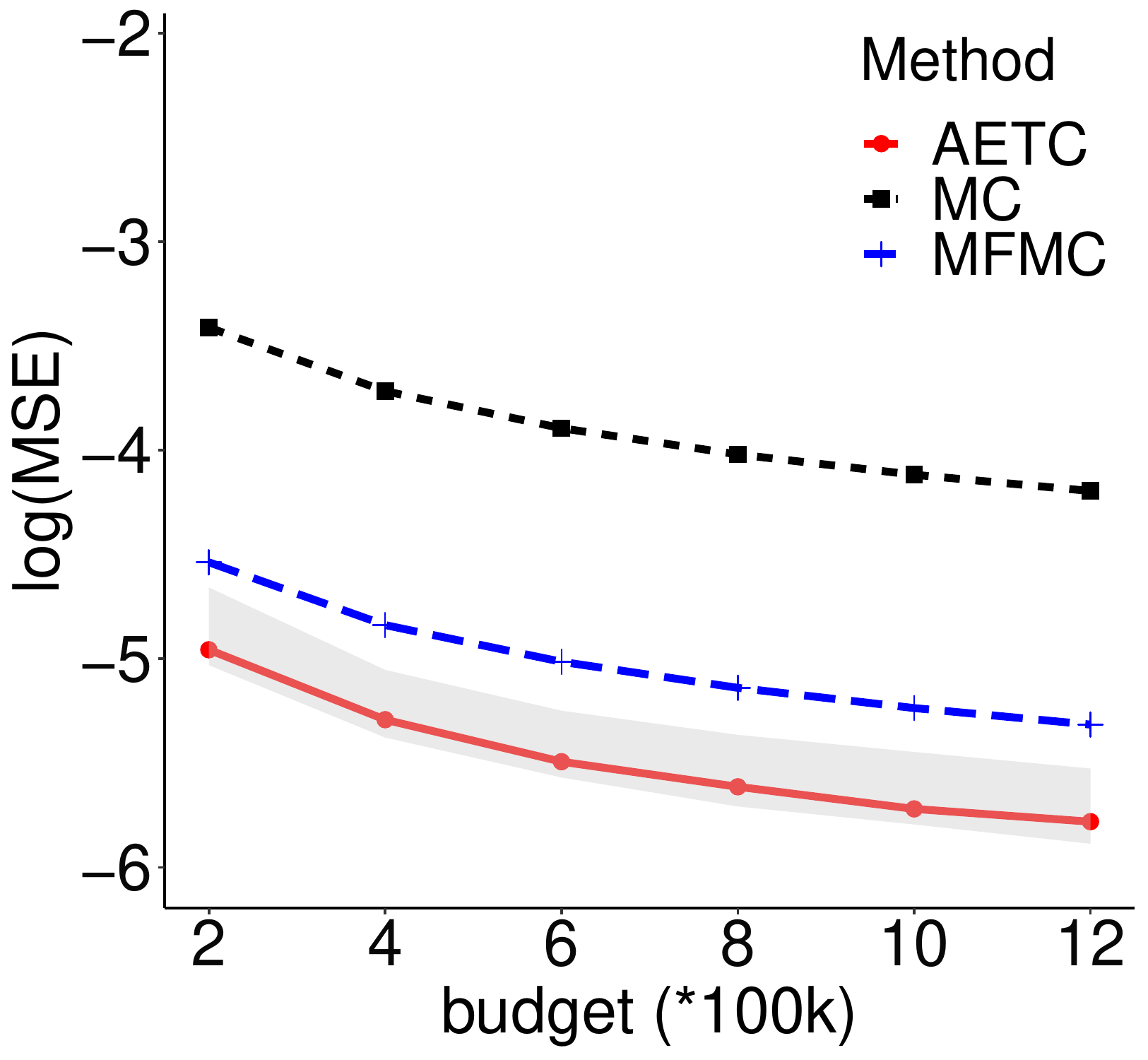}\hspace{2.5 cm}
\includegraphics[width=0.32\textwidth]{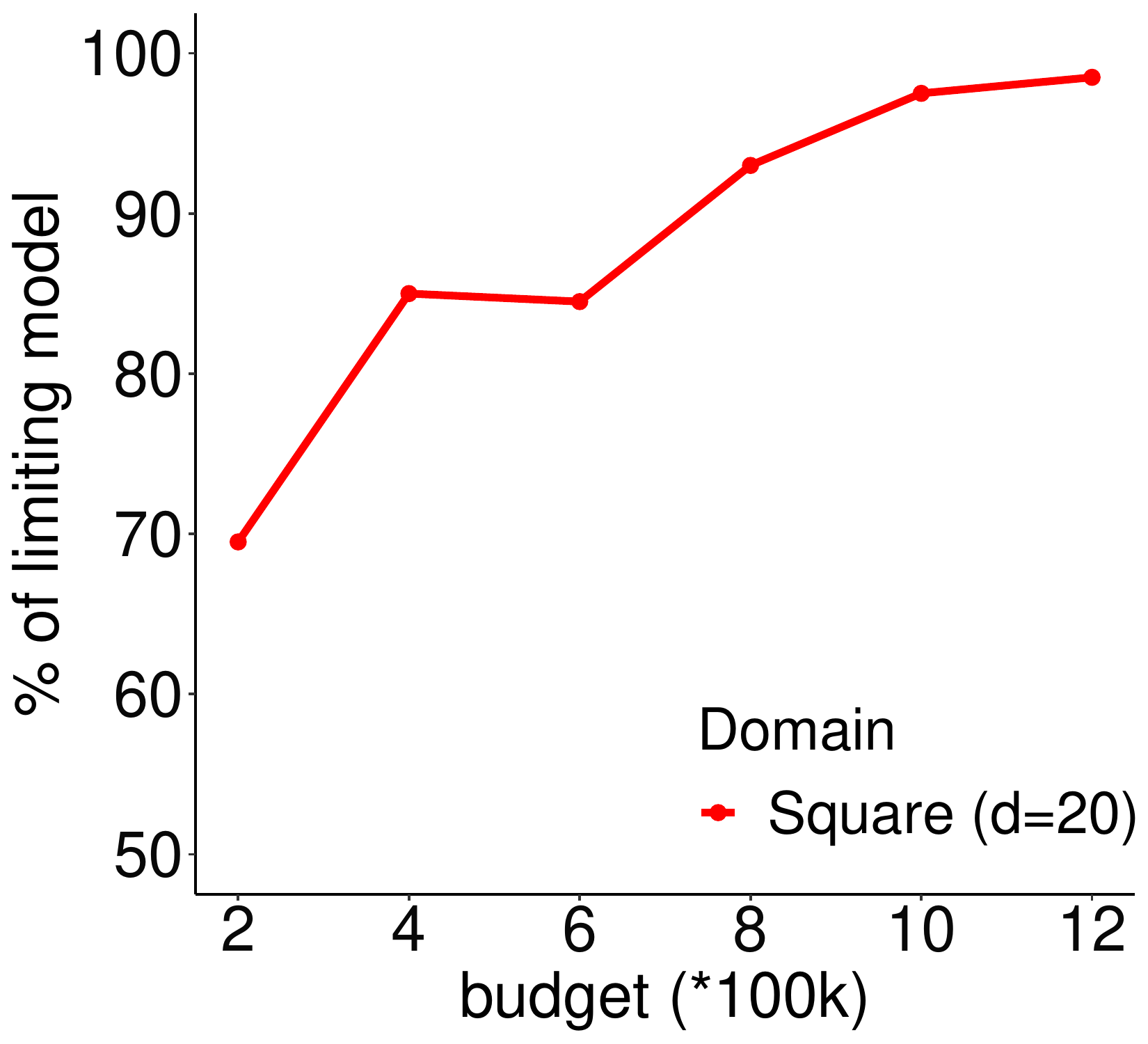}\\
  \noindent\makebox[\linewidth]{\rule{\textwidth}{0.6 pt}}
  \vspace{0.3 cm}
  
\includegraphics[width=0.32\textwidth]{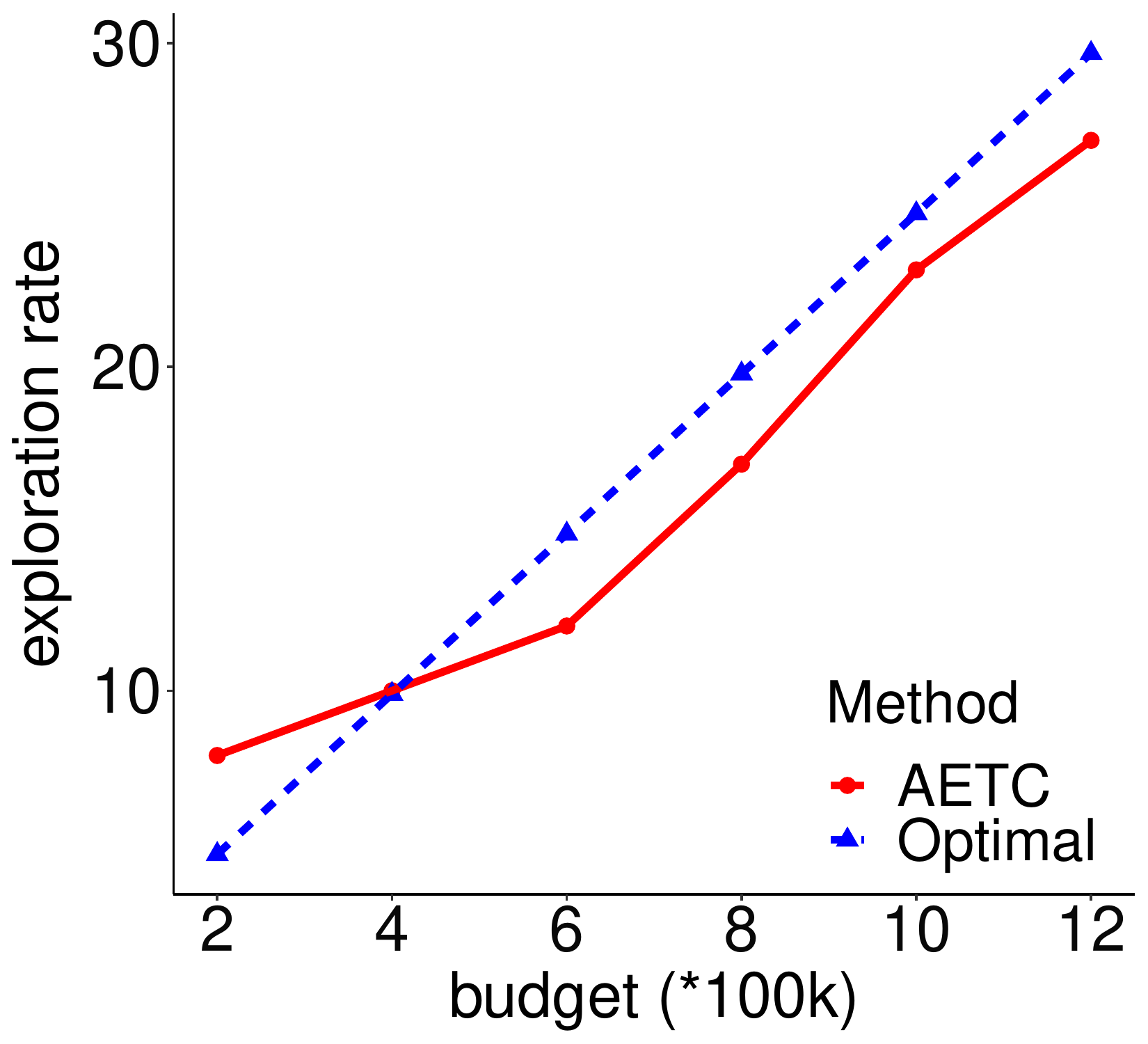}\ \ 
\includegraphics[width=0.32\textwidth]{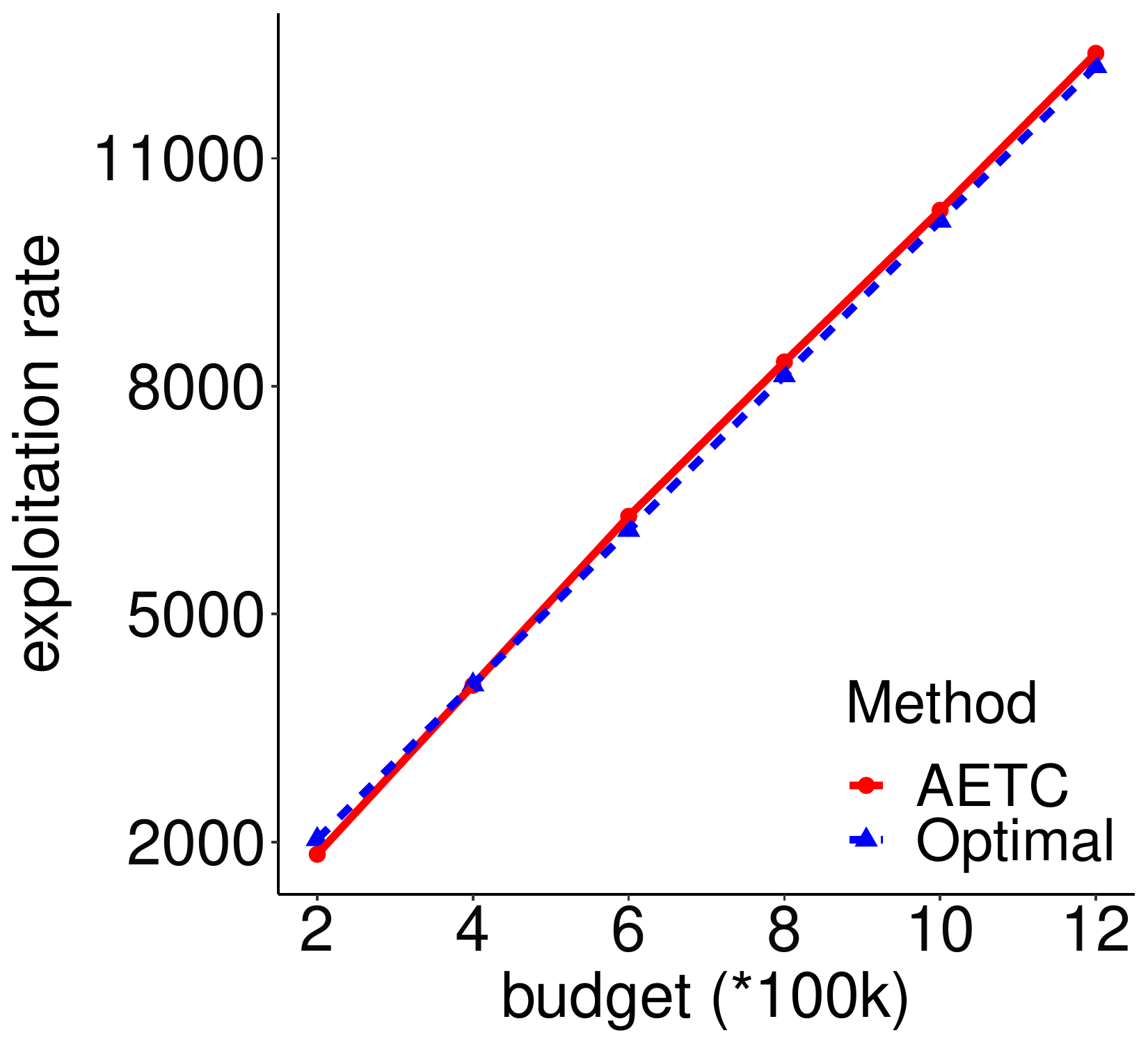}\ \   
\includegraphics[width=0.32\textwidth]{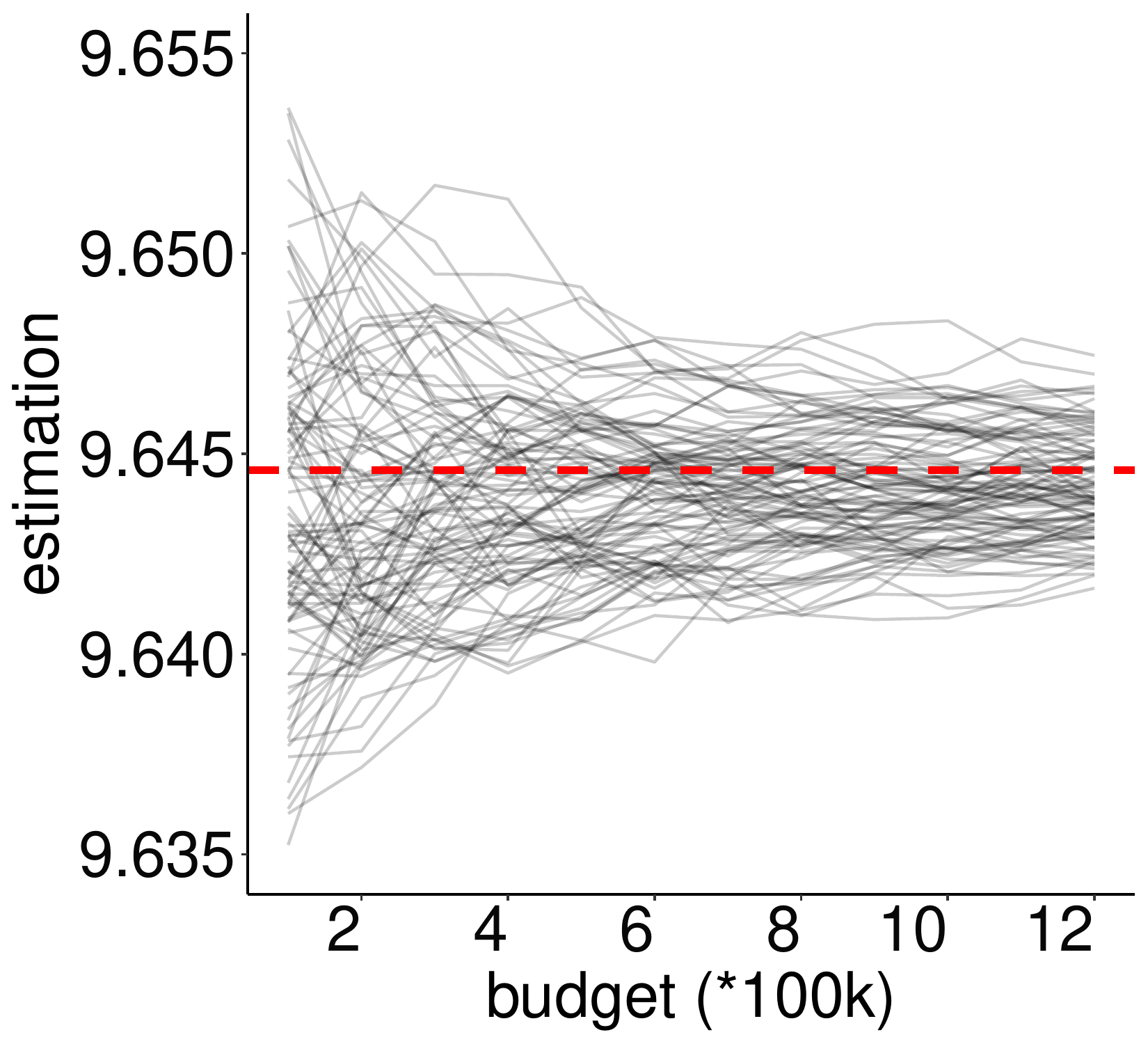}
\caption{\small Comparison of the ($\log_{10}$) mean-squared error of the LRMC estimator given by the AETC algorithm, the MC estimator, and the MFMC estimator as the total budget increases from $2\times 10^5$ to $12\times 10^5$ in the case of the square domain with KL truncation parameter $d=20$ (\textbf{top left}).  
The $0.05$-$0.50$-$0.95$-quantiles are plotted for the LRMC estimator to measure its uncertainty. 
  We also compute the probability (plotted as a percentage) that the AETC algorithm selects the limiting model given by \eqref{>} (\textbf{top right}), and the median of the exploration (\textbf{bottom left}) and exploitation (\textbf{bottom middle}) rate in AETC at different budget. Finally, we provide the estimated values for $\E[f(Y)]$ of the LRMC estimation in AETC at different budget along 100 random trajectories, with the red dashed line referring to the true value of $\E[f(Y)]$ (\textbf{bottom right}). }
  \label{fig:more1}
  \end{center}
\end{figure}

Figure \ref{fig:more1} shows that AETC consistently outperforms the other two methods as in Figure \ref{fig:1} and the expected model convergence. Both the exploration and exploitation rates in AETC asymptotically match the optimal rates computed using the oracle statistics, verifying the statement \eqref{rr1} in Theorem \ref{thm:AETC}. Moreover, by plotting the AETC estimation along $100$ random trajectories (with multiple evaluations at different budgets), we notice that the LRMC estimator in the AETC algorithm converges to $\E[f(Y)]$ as the budget goes to infinity. This observation is consistent with Remark \ref{rrrm}.

We next investigate how the performance of AETC depends on the regularization parameters $\alpha_t$.
In addition to the value $\alpha_t = 4^{-t}$ previously used, we consider two alternative choices $\alpha_t = 2^{-t}$ and $\alpha_t = 8^{-t}$ corresponding to a more and less active exploration strategy, respectively, in the early stage of learning.
We apply AETC with different regularization parameters to the same training dataset under different total budgets and record the corresponding average mean squared errors.
The simulation results are reported in Figure \ref{fig:more2}.

It can be seen from Figure \ref{fig:more2} that for a large budget, the three exponential-decaying choices for $\alpha_t$ yield similar accuracy results. For a small budget, however, the larger $\alpha_t$ is slightly more accurate, likely due to a more aggressive early-stage exploration enforced by the regularization.

\begin{figure}[htbp]
\begin{center}
\includegraphics[width=0.35\textwidth]{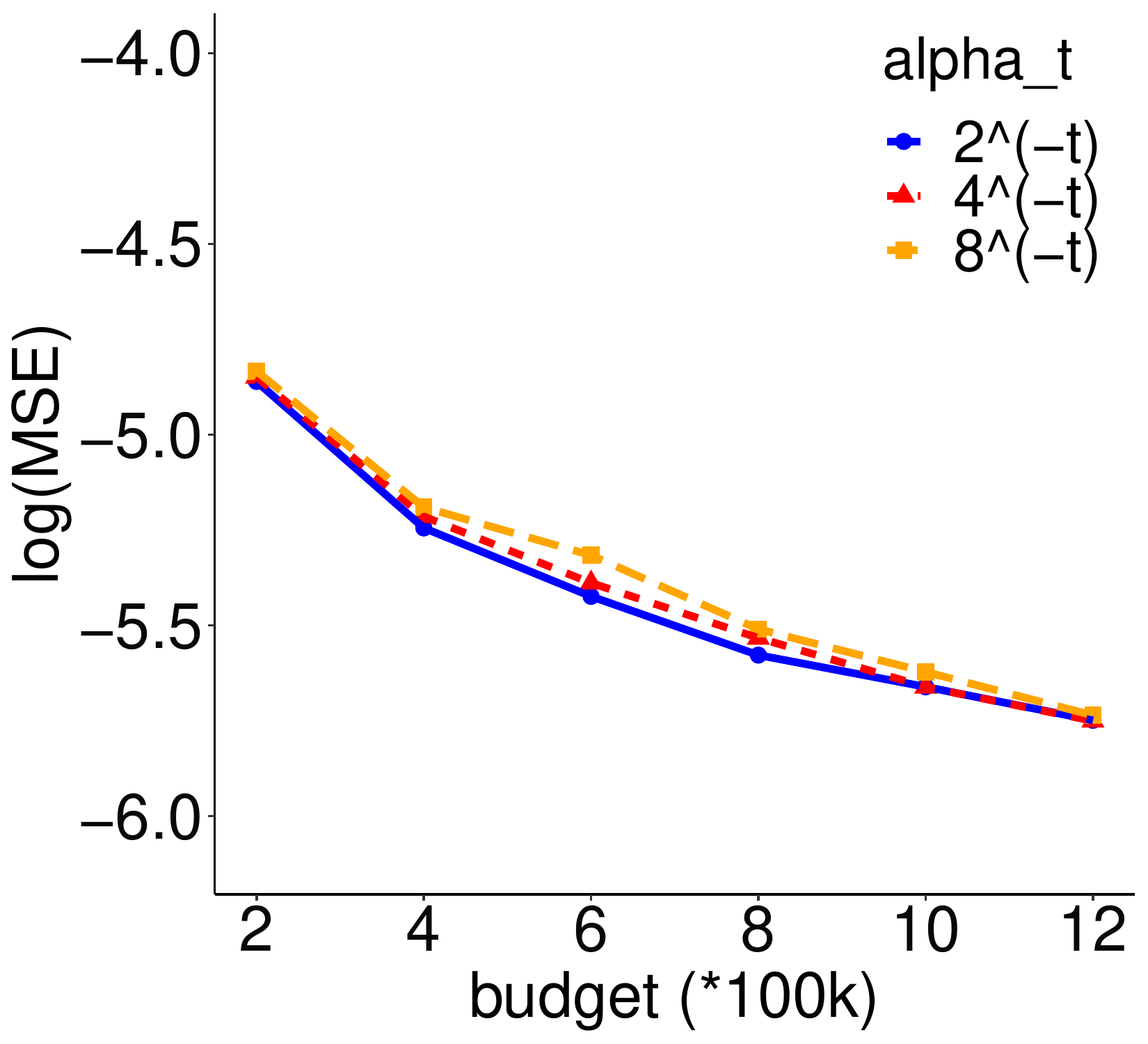}
\caption{\small Comparison of the average ($\log_{10}$) mean-squared error of the AETC algorithm at three different regularization parameters ($\alpha_t = 2^{-t}, 4^{-t}, 8^{-t})$ as the total budget increases from $2\times 10^5$ to $12\times 10^5$ in the case of the square domain with KL truncation $d=20$. }
  \label{fig:more2}
  \end{center}
\end{figure}


\subsubsection{Vector-valued high-fidelity output}
The same experiment is repeated when $f(Y)$ is taken as a vector-valued response, and hence we utilize Algorithm \ref{alg:aETC_mul}. 
We take $f(Y)$ to be defined by $9$ randomly selected components of the magnitude of the discrete solution (displacement) given by the high-fidelity model within a region shown with the highlighted (blue) color in Figure~\ref{fig:struct}, i.e.,
\begin{align*}
&f(Y) := \bar{\bm u}^{(0)}|_T &   \bar{\bm u}^{(0)} = \sqrt{\left(\bm {u}^{(0)}_x\right)^2 + \left(\bm {u}^{(0)}_y \right)^2}\in\R_+^{2601}
\end{align*}
where $T$ is a subset of coordinates in the highlighted region in Figure~\ref{fig:struct} with $|T| = 9$, and the arithmetic operations above are taken componentwise. The full spatial domain of the structures depicted in Figure \ref{fig:struct} is in $[0,1]^2$, and the highlighted regions are squares with the size $0.5 \times 0.5$ in the center and lower left corner of the square and L-shape structures, respectively. In order to compute the vector-valued quantity at the same points across all resolutions, we evaluate $\bar{\bm u}$ on a fixed $51 \times 51$ mesh across different resolutions, where the evaluations are accomplished by using the continuous solution from the finite element approximation.

For the square domain, $T$ is taken as a randomly selected $3\times 3$ block of pixels from the $51\times 51$ solution vector corresponding to the highlighted region in Figure \ref{fig:struct}.
For the L-shape domain, $T$ is taken as $9$ randomly sampled components from the $51\times 51$-dimensional solution vector corresponding to the highlighted region in Figure \ref{fig:struct}. 
For both cases, $Q$ is the identity matrix $I_9$, so the $Q$-risk defined in \eqref{newreg} is simply the sum of the mean-squared error of each response. 
We show oracle information of the correlations between $X^{(i)}$ and $f^{(j)}(Y)$, $i\in [6], j\in [9]$ in Figure \ref{corplot}.

\begin{figure}[htbp]
\begin{center}
\includegraphics[width=0.48\textwidth]{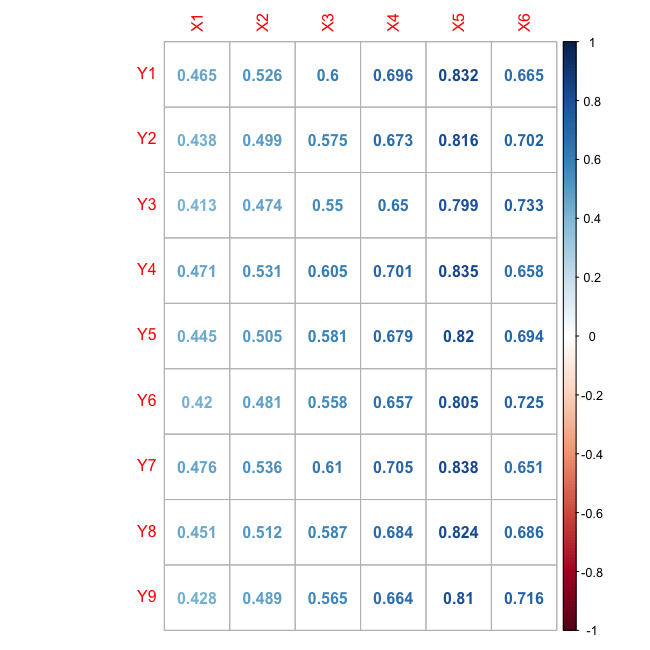}\hspace{0.4cm}
\includegraphics[width=0.48\textwidth]{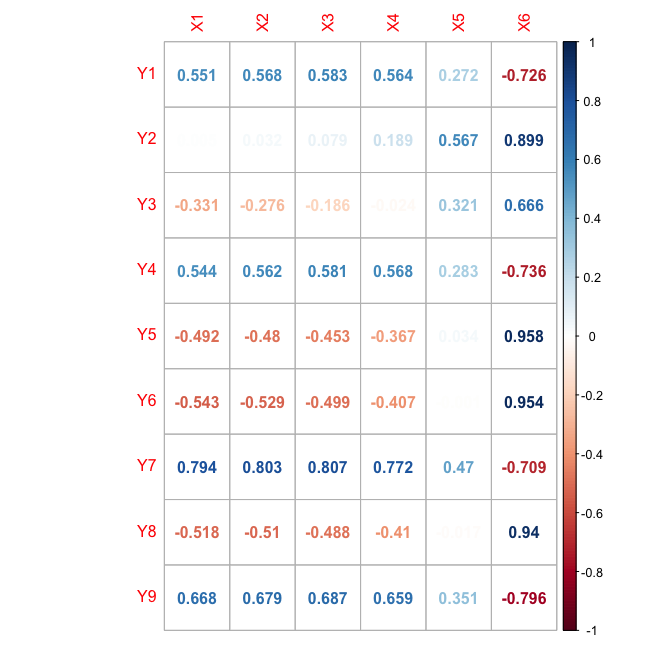}
\caption{\small Correlations between $X^{(i)}$ and $f^{(j)}(Y)$, $i\in [6], j\in [9]$ in the case of the square domain (\textbf{left}) and the L-shape domain (\textbf{right}). The entry associated to indices (Xi, Yj) is the value of $\cor(X^{(i)}, f^{(j)}(Y))$.}   \label{corplot}
   \end{center}
\end{figure}

For both domain structures, cheaper models have some strong correlations with the high-fidelity model, implying that the cost-correlation hierarchical structure (an assumption for the MFMC) is violated. We thus compare only the MC estimator with the AETC (for vector-valued high-fidelity output). The oracle limiting models to which AETC will converge are:
\begin{itemize}
\item Square domain: $f(Y)\sim X^{(4)} + X^{(5)}+X^{(6)}+\text{intercept}$
\item L-shape domain: $f(Y)\sim X^{(3)}+X^{(4)}+X^{(5)}+X^{(6)}+\text{intercept}$. 
\end{itemize}
The details of the results are given in Figure \ref{fig:11}, which are consistent with the conclusion drawn from Figure \ref{fig:1}. 

\begin{figure}[htbp]
\begin{center}
\includegraphics[width=0.32\textwidth]{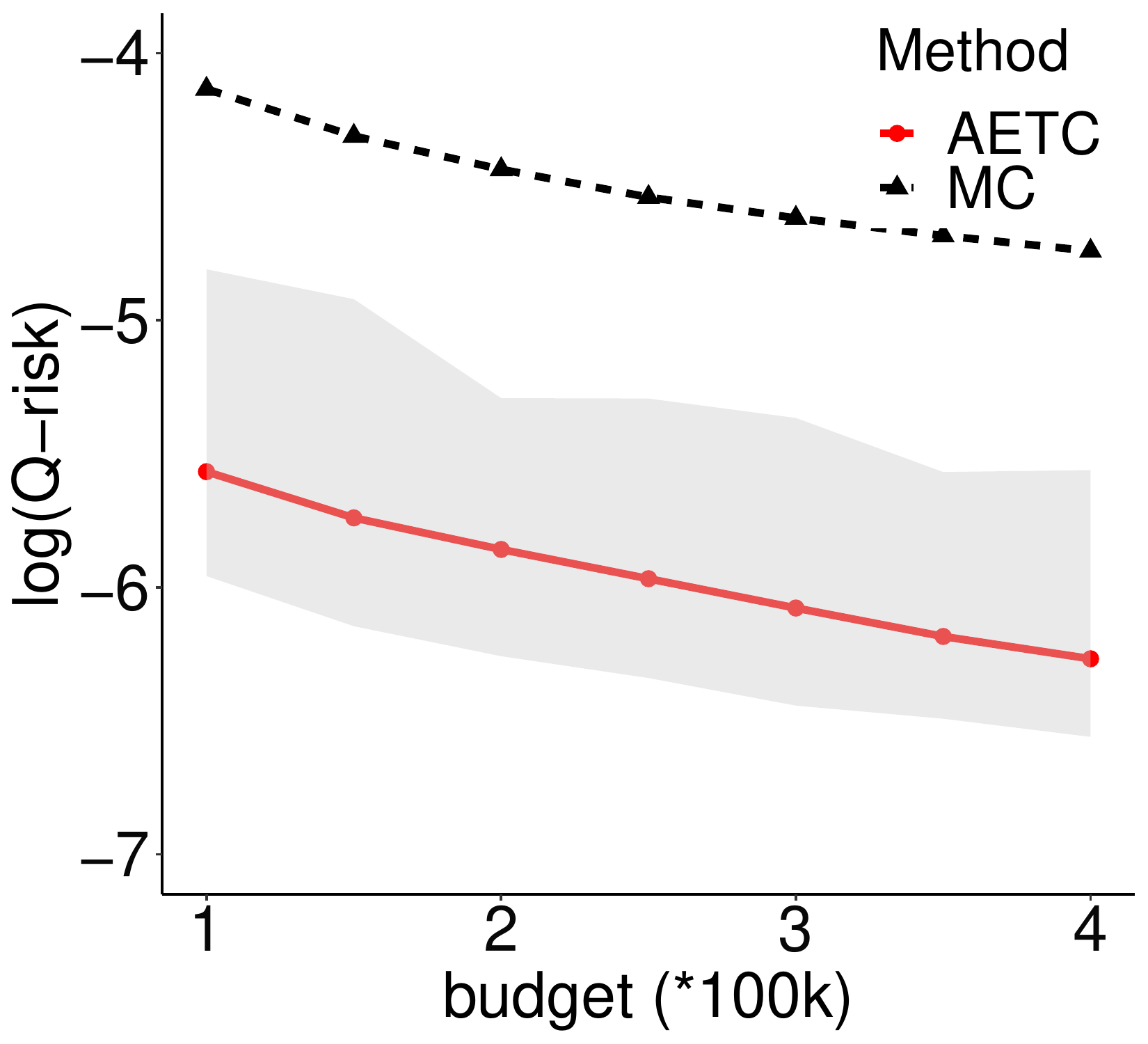}\ \ 
\includegraphics[width=0.32\textwidth]{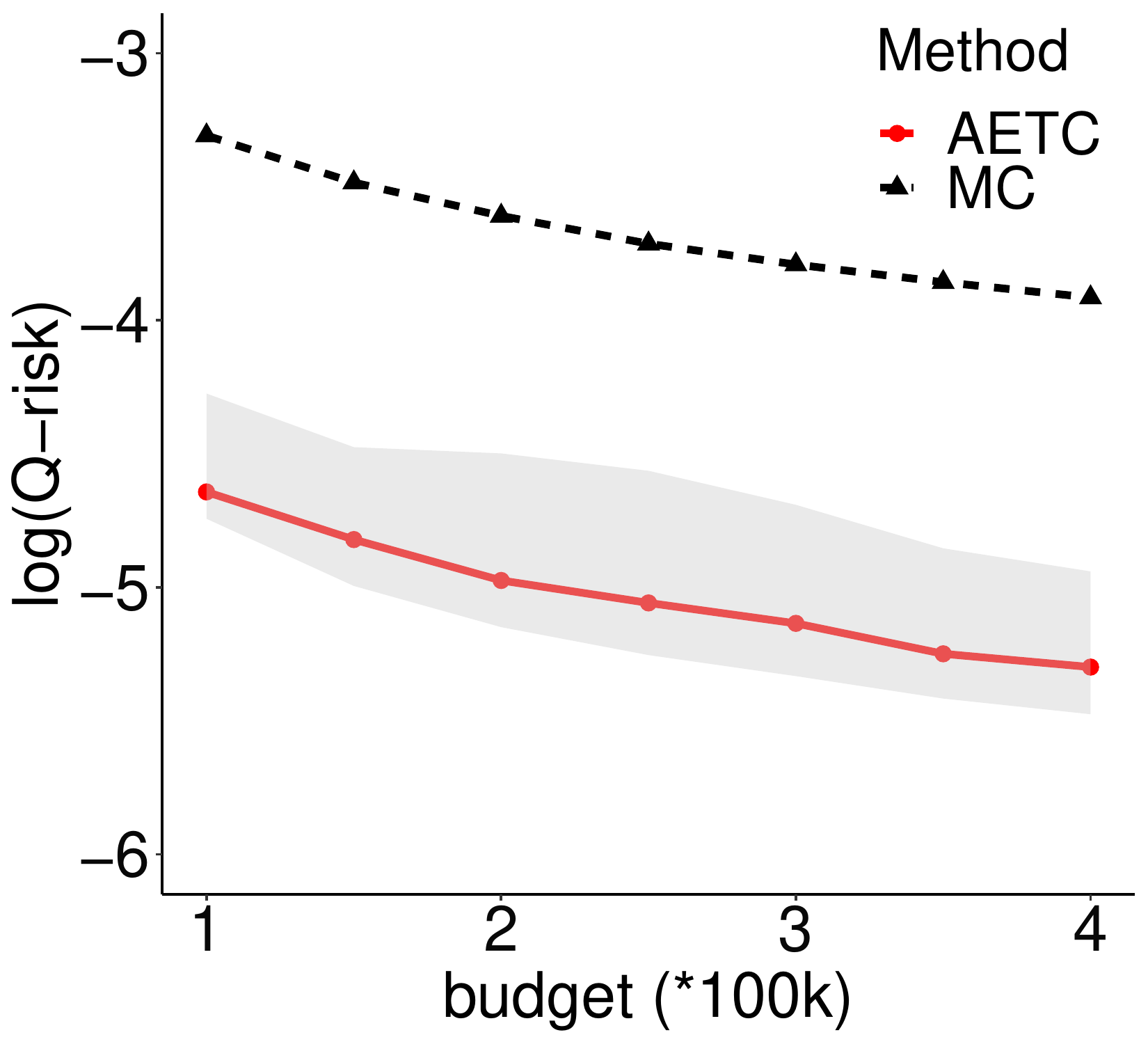}\ \ 
\includegraphics[width=0.32\textwidth]{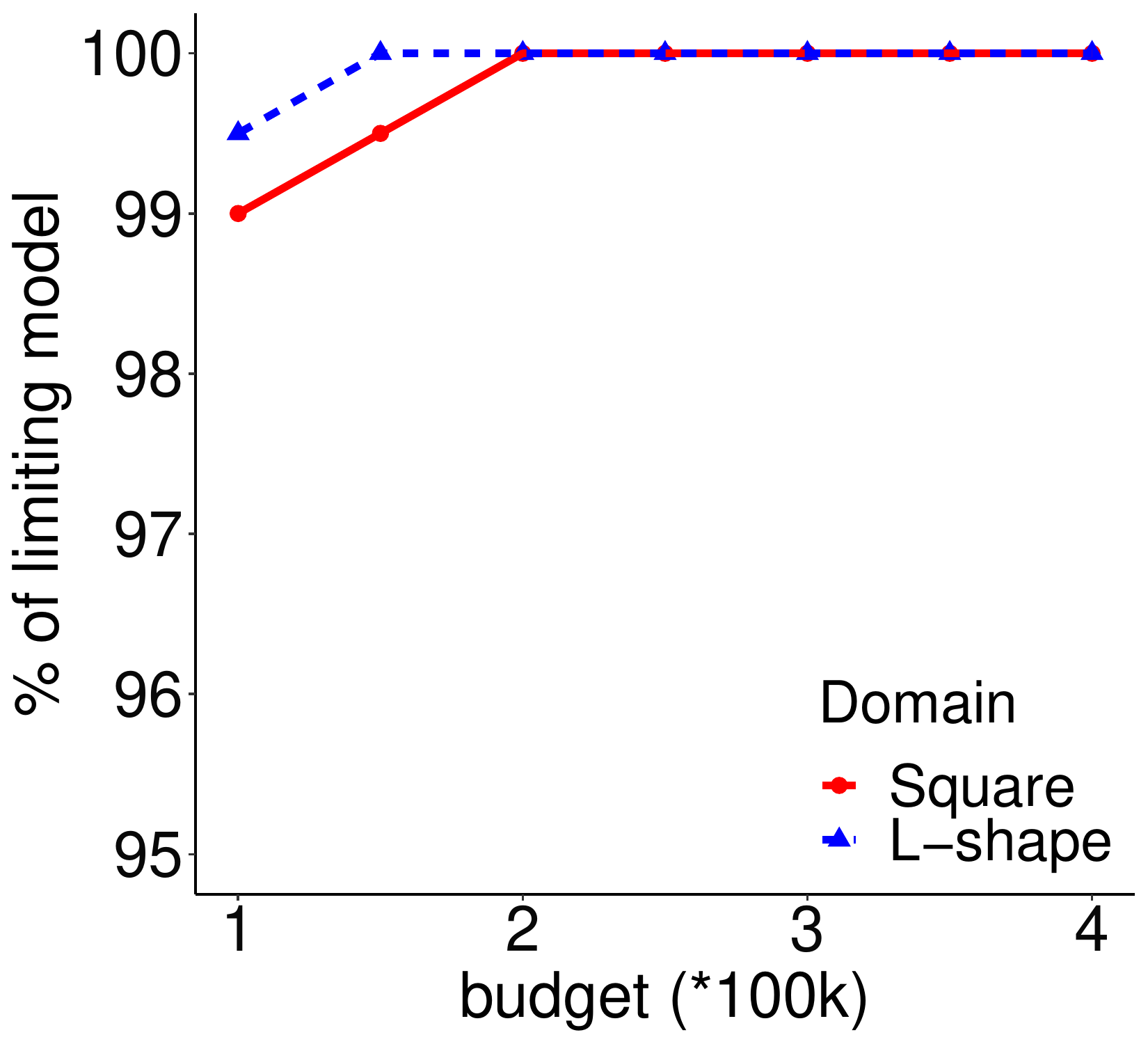}
\caption{\small Comparison of the ($\log_{10}$) $Q$-risk of the LRMC estimator given by the AETC algorithm and the MC estimator as the total budget increases from $10^5$ to $4\times 10^5$ in the case of the square domain (\textbf{left}) and the L-shape domain (\textbf{middle}).  
The $0.05$-$0.50$-$0.95$-quantiles are plotted for the LRMC estimator to measure its uncertainty. 
  We also compute the probability (plotted as a percentage) that the AETC algorithm selects the limiting model given by \eqref{><} (\textbf{right}).}
  \label{fig:11}
  \end{center}
\end{figure}

\subsection{Mixture of reduced-order models}\label{ssec:rom-gp}

In this example, we consider a multifidelity problem using the same model as in the previous section on the square domain, but we generate
different types of reduced-order models (or surrogate models) which approximate the compliance of the solution to \eqref{elliptic_pde_stochastic}.  
We will use the compliance computed via the Finite Element Method on a $2^6\times 2^6$ mesh as the ground truth, i.e., the surrogate model $X^{(1)}$ in the first table in Table~\ref{145} is treated as the high-fidelity model $Y$. 
We explore two classes of choices for reduced-order models: 
\begin{itemize}
  \item Gaussian process (GP) emulators. Low-fidelity regressors $X^{(i)}$ for $i \in [6]$ are generated as the mean of GP emulators built on compliance data from $Y$. We use an exponential covariance kernel and optimize hyperparameters by maximizing the log-likelihood. For $n_T = 10, 100, 1000$ training points in parameter space, this defines models $X^{(1)},~X^{(3)},~X^{(5)}$, respectively. We then select non-optimal hyperparameters with the same training data $n_T = 10, 100, 1000$, which defines models $X^{(2)},~X^{(4)},~X^{(6)}$, respectively. The cost of these models is given by the cost of training, optimization, and evaluation the GP averaged over $2 \times 10^4$ different values of $\bm p$. We perform each experiment five times and report the average time (cost) for each model in Table~\ref{1451}.
  \item Projection-based model reduction with proper orthogonal decomposition (POD).  We generate $k$ POD basis functions from high-fidelity displacement data, use this to form a rank-$k$ Galerkin projection of the finite element formulation, and models $X^{(i)}$ for $i = (7, \ldots, 12)$ are the compliances computed from the projected systems of rank $k = (1, 2, 3, 4, 5, 10)$, respectively. The cost of this procedure is taken as \textit{only} the cost of solving the rank-$k$ projected system and does \textit{not} include the time required to collect POD training data or the time required to compute the POD modes.
\end{itemize}
More details about the experimental setup above are given in Appendix \ref{app:num-setup-2}. The oracle statistics and costs for the GP and POD low-fidelity models are shown in Table \ref{1451}. Note that here we have not normalized the cost relative to the low-fidelity model.  

\begin{table}[htbp]
\begin{center}
\small
\begin{tabular}{l*{6}{c}r}
Models              & $f(Y)$ &$X^{(1)}$ & $X^{(2)}$ & $X^{(3)}$ & $X^{(4)}$& $X^{(5)}$ \\
\hline
$\text{Corr}(\cdot, f(Y))$ & 1 & 0.993 & 0.414 & 1-2e-05 & 0.401 & 1-1e-06 \\
Mean            & 9.197 & 9.195 & 9.147 & 9.197 & 9.045 & 9.197 \\
Standard deviation    &  0.113      & 0.122 & 0.556 & 0.113 & 0.335 &  0.113  \\
Cost            & 9233.69& 0.31 & 0.31 & 2.31 & 2.33 &  30.79  \\
\eqref{yuc} with $S=\{i\}$  & --& 4.202 & 203.025 & 0.358 & 214.676 &  3.574  \\
\end{tabular} 
\bigskip

\begin{tabular}{l*{6}{c}r}
Models            & $X^{(6)}$  & $X^{(7)}$ &$X^{(8)}$ & $X^{(9)}$ & $X^{(10)}$ & $X^{(11)}$ & $X^{(12)}$  \\
\hline
$\text{Corr}(\cdot, f(Y))$ & 1-2e-04 & 0.999 & 1-2e-04 & 1-5e-05 & 1-8e-06 & $\approx 1$ & $\approx 1$ \\
Mean           & 9.196 & 9.189 & 9.194 & 9.195 & 9.197 & 9.197 & 9.197  \\
Standard deviation    & 0.114    & 0.113 & 0.113 & 0.113 & 0.113 &  0.113 &  0.113 \\
Cost         & 31.29  & 0.18& 0.45 & 0.68 & 0.85 &  1.03 & 2.02 \\
\eqref{yuc} with $S=\{i\}$  & 6.226 & 0.732 & 0.243 & 0.180 & 0.137 & \textbf{0.119} &0.230  \\
\end{tabular} 
  \caption{\small Oracle information of $f(Y)$ and $X^{(i)}$. The numerator of the asymptotic average conditional MSE estimate \eqref{yuc} is also shown.}\label{1451}
\end{center}
\end{table}

We now apply the (scalar response) AETC algorithm to this multifidelity setup, with the total budget ranging from $10^5$ to $2\times 10^5$, incremented by $0.2\times 10^5$ in our experiments.
We have $12$ low-fidelity models in total, and exhausting all of them for selection would require complexity on the order of $2^{12}-1$.
Since the AETC algorithm generally produces an efficient combination of relatively cheap regressors, we set the maximal number of regressors in the model to be $5$ to accelerate computation, i.e., we explore only for $s = |S| \leq 5$. 
(We could have used the full model, which would not incur too much an increase in computation time but would require taking more exploration samples to start with, which is a waste of resources.)

The performance of AETC is compared to the direct MC estimator applied to $f(Y)$, and the results are reported in the first plot in Figure \ref{fig:111}. The figure shows that AETC is more efficient than the MC estimator by a substantial margin in terms of the mean-squared error.
To better see how this is reflected in practice, we fix the budget to be $B = 10^5$, run $200$ experiments of both AETC and the MC, and compute the difference between the estimated values of $\E[f(Y)]$ and the ground truth. This is illustrated in the second plot in Figure \ref{fig:111}.  
Since this multifidelity setup contains surrogates obtained from different types of methods, we investigate which are chosen by AETC for exploitation.  
This information is given in the last plot in Figure \ref{fig:111}. 
The most frequent models chosen by the AETC are between model $X^{(10)}$ and $X^{(11)}$, both of which have near-perfect correlation with $f(Y)$ with only a moderate cost. 
In fact, $X^{(11)}$ is also the oracle limiting model given by \eqref{>}, although the convergence is rather slow due to the competitor $X^{(10)}$, which has an extremely high correlation but a slightly cheaper cost.   
The other models which have been selected by the AETC are the other POD models as well as their combinations with models $X^{(1)}$ and $X^{(2)}$. 
These models correspond to the lowest fidelity models in each method of approximation and are often cheap to sample from with reasonably high correlation.

\begin{figure}[htbp]
\begin{center} 
\includegraphics[width=0.32\textwidth]{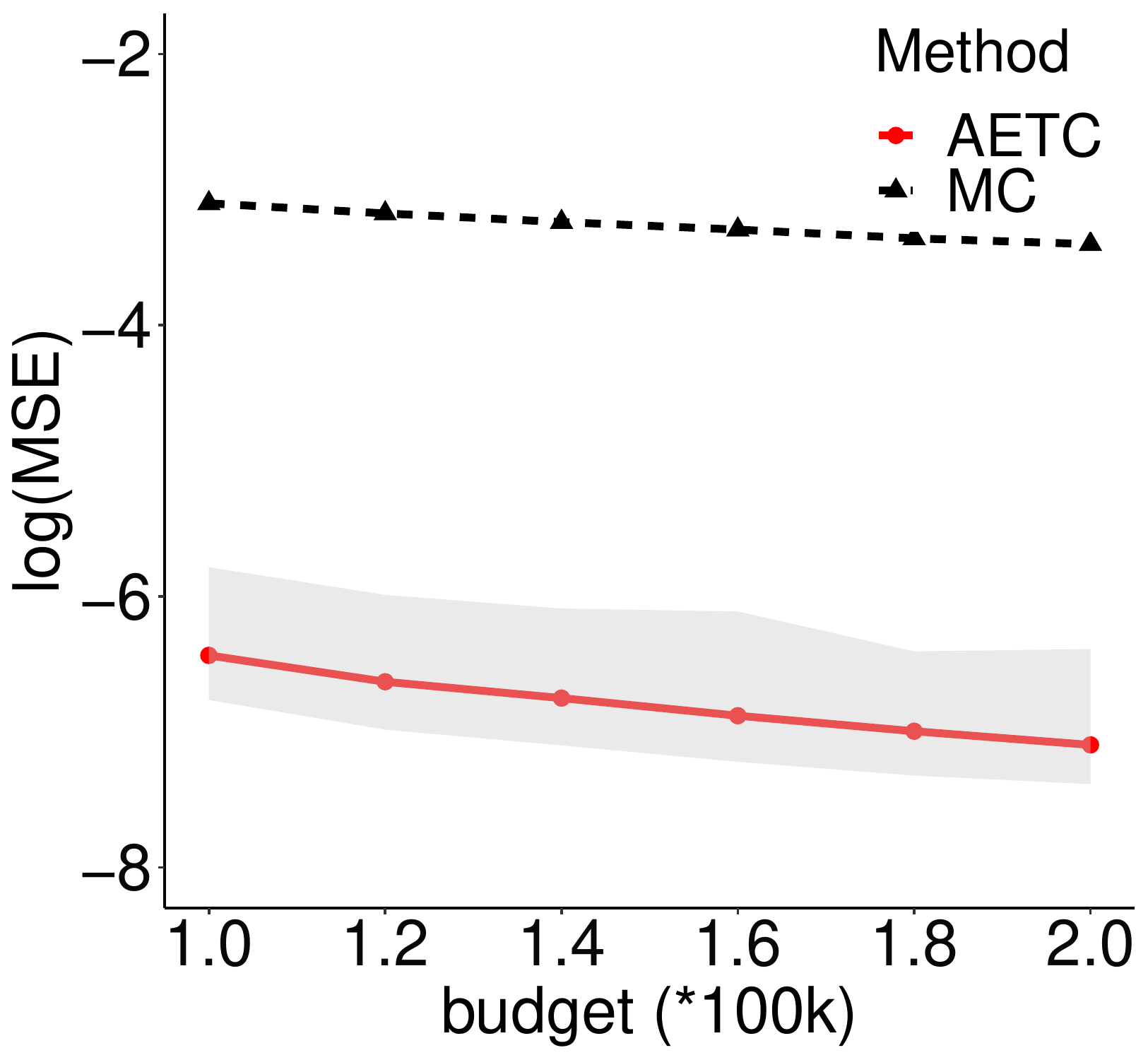}\ \ 
\includegraphics[width=0.32\textwidth]{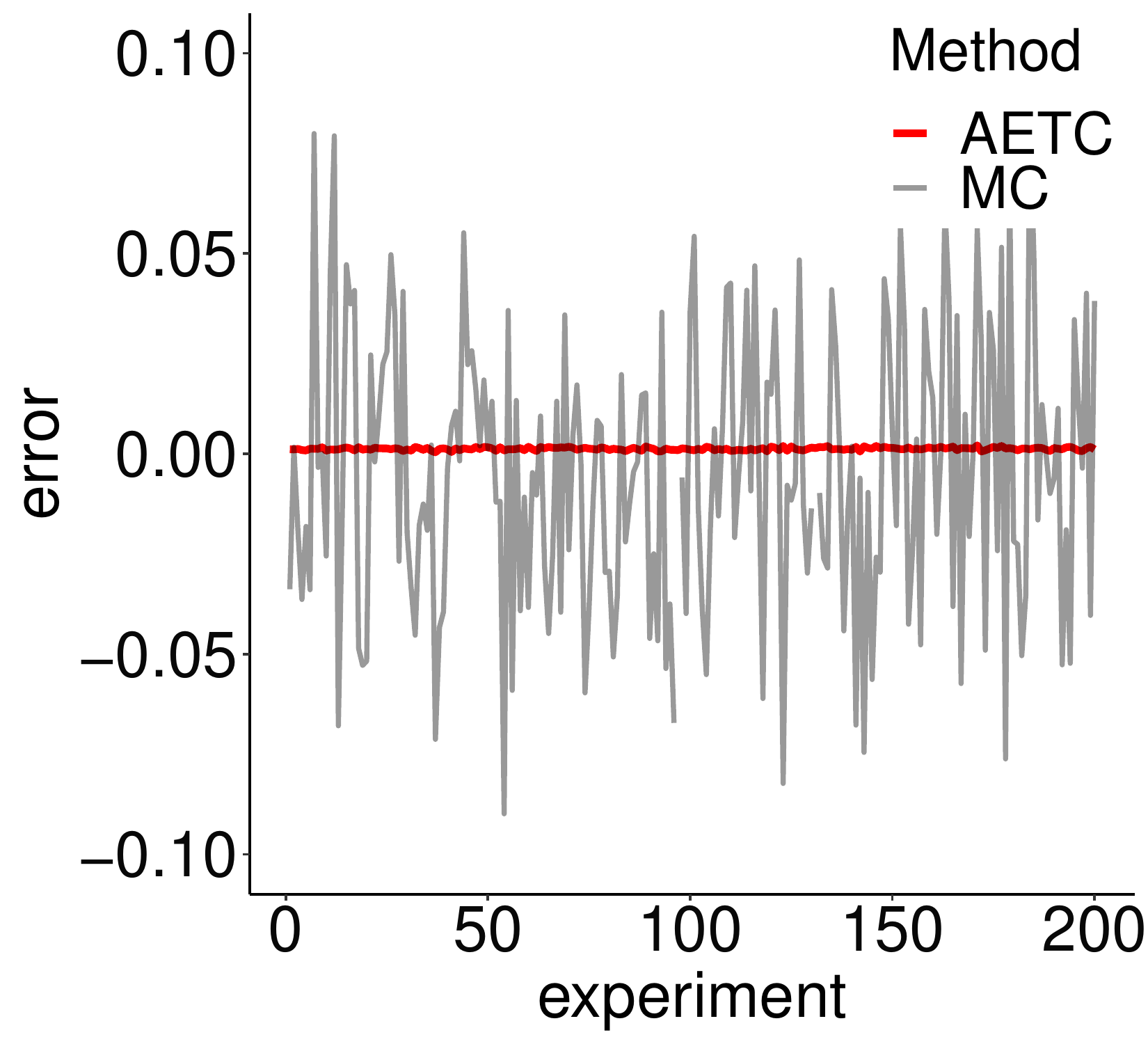}\ \ 
\includegraphics[width=0.32\textwidth]{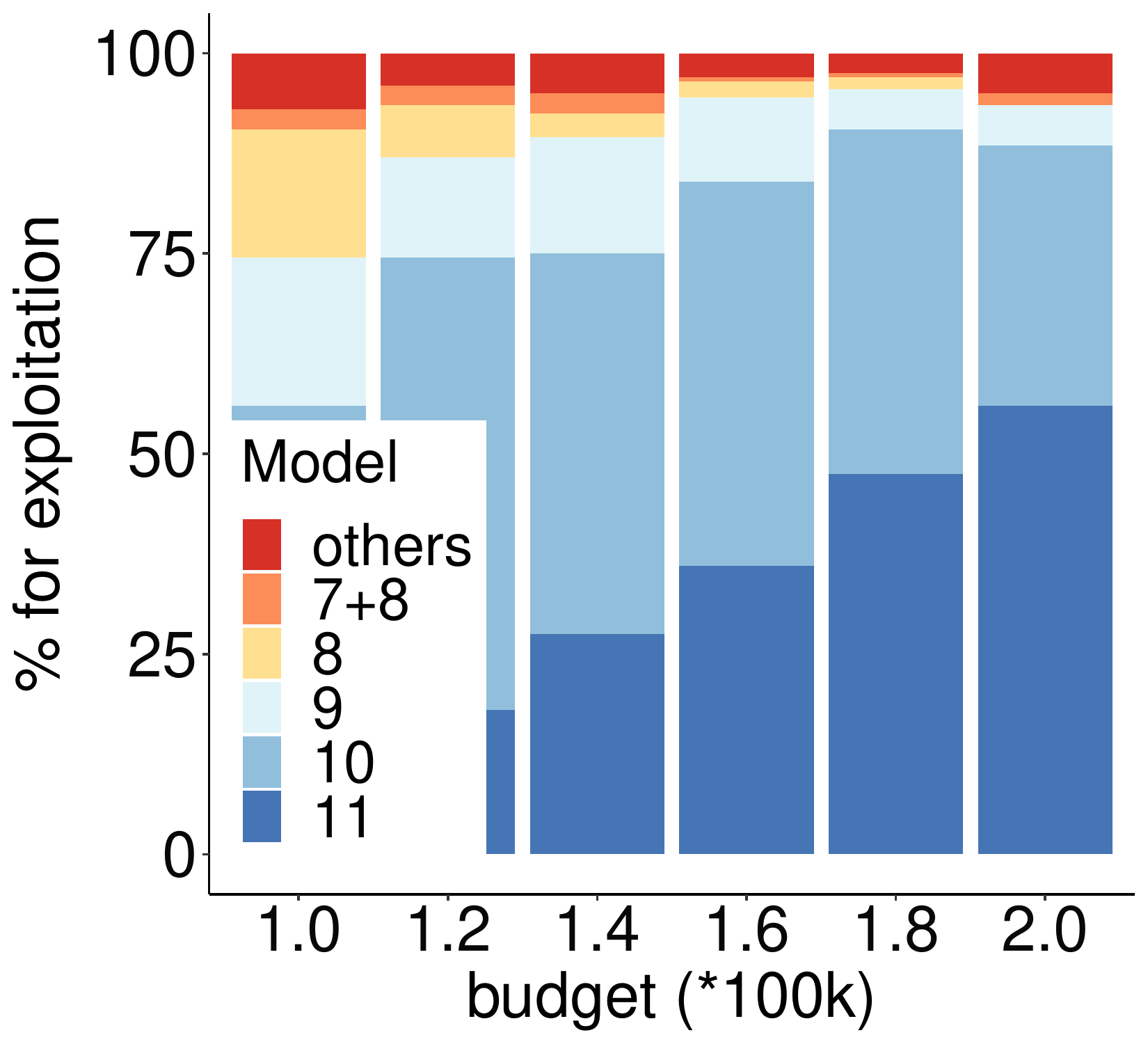}
  \caption{\small The ($\log_{10}$) mean-squared error of the estimators as the budget increases, with the $0.05$-$0.5$-$0.95$ quantiles plotted for the AETC algorithm to measure the uncertainty in the exploration (\textbf{left}). Comparison of the estimation errors of MC and the AETC algorithm by plotting instances of the error over each of $200$ experiments for a fixed total budget $B = 10^5$ (\textbf{middle}). Empirical probability (plotted as a percentage) of certain models being selected by AETC for exploitation (\textbf{right}).}
  \label{fig:111}
  \end{center}
\end{figure}

\section{Conclusions}\label{concl}

We have proposed a novel algorithm for the multifidelity problem based on concepts from bandit learning. Our proposed AETC procedures, Algorithms \ref{alg:aETC} and \ref{alg:aETC_mul}, operate under the assumption of a linear relationship between low-fidelity features and a high-fidelity output. Their exploration phase expends resources to learn about model relationships to discover an effective linear relationship. The exploitation phase leverages the cost savings by regressing on low-fidelity features to produce an LRMC estimate of the high-fidelity output.
Under the linear model assumption, we show the consistency of the LRMC, whose MSE can be partitioned into a component stemming from exploration, and another from exploitation. This partition allows us to construct the AETC algorithm, which is an adaptive procedure operating under a specified budget that decides how much effort to expend in exploration versus exploitation, and also identifies an effective linear regression low-fidelity model. We show that, for a large budget, our AETC algorithms explore and exploit optimally.

The main advantage of our approach is that no \textit{a priori} statistical information or hierarchical structure about models is required, and very little knowledge about model relationships is needed: AETC needs only identification of which model is the trusted high-fidelity one, along with a specification of the cost of sampling each model. The algorithm proceeds from this information alone, which is a distinguishing feature of our approach compared to alternative multifidelity and multilevel methods, and can be used to tackle situations when a natural cost versus accuracy hierarchy is difficult to identify, and when relationships between models are not known.

\section*{Acknowledgement}
We would like to thank the anonymous referees for their very helpful comments which significantly improved the presentation of the paper.

\section*{Appendices}

\begin{appendices}

\section{The LRMC reusing the exploration data}\label{ap:0}
 The goal of this section is to provide a quantitative discussion that accompanies Remark \ref{rem:recycle}. Our analysis in this paper computes coefficients $\widehat{\beta}_S$ in exploitation using only the $N_S$ samples that are taken during exploitation. This analysis neglects the possibility of reusing the $m$ samples from exploration in the estimation of $\widehat{\beta}_S$. All our numerical results fall into a regime where $m \ll N_S$, corresponding to the case when there are many more exploitation samples (of low-fidelity models) than exploration rounds (which require querying the high-fidelity model). This section demonstrates that in this multifidelity regime of interest, recycling of exploration samples during exploitation has negligible impact.

Fix a model $S\subseteq [n]$ .
The LRMC estimator with the exploration samples reused is defined as 
\begin{align}
\lrmc_{S,\re} = \frac{1}{N_S+m}\left(\underbrace{\sum_{\ell\in [N_S]}X_{S,\ell}^T}_{\text{exploitation samples}}+\underbrace{\sum_{j\in [m]}X^T_{\ex,j}|_S}_{\text{exploration samples reused}}\right)\widehat{\beta}_S.
\end{align}
It follows from a similar computation as in \eqref{ana} that the conditional mean-squared error of $\lrmc_{S,\re}$ on the exploration data (including both the exploration samples and the model noise) is 
\begin{align}
\frac{N_S}{(N_S+m)^2}\widehat{\beta}_S\Sigma_S\widehat{\beta}_S + \left(\widetilde{X}^T_m\widehat{\beta}_S-x_S^T\beta_S\right)^2,\label{reuse1}
\end{align} 
where
\begin{align*}
\widetilde{X}_{m} &= \frac{N_S}{N_S+m}x_S+\frac{1}{N_S+m}\sum_{j\in [m]}X_{\ex,j}|_S\\
&\stackrel{\eqref{e1}}{=} x_S + \underbrace{\frac{m}{N_S+m}(\widehat{x}_S(m)-x_S)}_{: = \Delta}.
\end{align*}
Note $\Delta = o(x_S)$ as $m\to\infty$ due to the law of large numbers. 
Averaging the model noise in \eqref{reuse1} yields the average conditional MSE of $\lrmc_{S,\re}$:
\begin{align}
&\E_{\e_S}[\eqref{reuse1}]\nonumber\\
\stackrel{\eqref{dep}, \eqref{pt}}{=}&\  \frac{N_S}{(N_S+m)^2}\left[\beta_S^T\Sigma_S\beta_S + \sigma_S^2\tr(\Sigma_S(Z_S^TZ_S)^{-1})\right]+\E_{\e_S}[(x^T_S(\widehat{\beta}_S-\beta_S))^2]\nonumber\\
\ \ \ \ \ \ \ \ & + 2\E_{\e_S}[x^T_S(\widehat{\beta}_S-\beta_S)\Delta^T\widehat{\beta}_S]+ \E_{\e_S}[(\Delta^T\widehat{\beta}_S)^2]\nonumber\\ 
\simeq&\  \frac{N_S}{(N_S+m)^2}\beta_S^T\Sigma_S\beta_S + \sigma^2_Sx_S^T(Z_S^TZ_S)^{-1}x_S +2\sigma^2_Sx_S^T(Z_S^TZ_S)^{-1}\Delta\nonumber\\
\ \ \ \ \ \ \ \ & + \left(\frac{m}{N_S+m}\right)^2\left[(\widehat{x}_S-x_S)^T\beta_S\right]^2\nonumber\\
\simeq&\  \frac{N_S}{(N_S+m)^2}\beta_S^T\Sigma_S\beta_S + \sigma^2_Sx_S^T(Z_S^TZ_S)^{-1}x_S  + \left(\frac{m}{N_S+m}\right)^2\left[(\widehat{x}_S-x_S)^T\beta_S\right]^2\nonumber\\
\simeq&\   \frac{N_S}{(N_S+m)^2}\beta_S^T\Sigma_S\beta_S + \frac{1}{m}\sigma^2_Sx^T_S\Lambda_S^{-1}x_S+\left(\frac{m}{N_S+m}\right)^2\left[(\widehat{x}_S-x_S)^T\beta_S\right]^2.\label{emini}
\end{align}
Thus, one can verify that
\begin{align}
&\frac{1}{N_S}\beta_S^T\Sigma_S\beta_S + \frac{1}{m}\sigma_S^2\tr(x_Sx_S^T\Lambda_S^{-1})\nonumber\\
\leq&\ \left(1+\frac{m}{N_S}\right)^2\left[\frac{N_S}{(N_S+m)^2}\beta_S^T\Sigma_S\beta_S + \frac{1}{m}\sigma^2_Sx^T_S\Lambda_S^{-1}x_S\right]\nonumber\\
\leq&\ \left(1+\frac{m}{N_S}\right)^2\cdot\eqref{emini}.\label{peru}
\end{align}
While according to \eqref{optm}, the leftmost term in \eqref{peru} converges to $\overline{\mse}_S|_{\widehat{\beta}_S}$ almost surely. 
Thus, providing $m/N_S\ll 1$, reusing exploration samples results in the estimate \eqref{emini}, which asymptotically has a negligible advantage over the estimate $\overline{\mse}_S|_{\widehat{\beta}_S}$ that does not reuse samples.

\section{Proof of Theorem~\ref{thm:conss}}\label{ap:1}
We start by conditioning on $Z_S$. Note that 
\begin{align}
  \widehat{\beta}_S - \beta_S =  Z_S^\dagger\eta_S,\label{pt}
\end{align}
 where $\eta_S$ is the model noise vector with each component being the model noise $\e_S$ in the corresponding exploration sample, i.e., $\eta_S/\sigma_S\in\R^m$ is an isotropic sub-Gaussian random vector with $\E[\eta_S \eta_S^T] = \sigma_S^2I_m$. Substituting \eqref{pt} into \eqref{ana} and applying the Cauchy-Schwarz inequality ($|2\langle x,y\rangle|\leq 2\|x\|_2\|y\|_2\leq \|x\|^2_2+\|y\|_2^2$) yields
\begin{align}
  \mse_S|_{\widehat{\beta}_S} &= \frac{1}{N_S}\left(\beta_S^T\Sigma_S\beta_S+\eta_S^T(Z_S^\dagger)^T\Sigma_SZ_S^\dagger\eta_S+2\beta_S^T\Sigma_S Z_S^\dagger\eta_S\right) + \eta_S^T(Z_S^\dagger)^Tx_Sx_S^TZ_S^\dagger\eta_S\nonumber\\
  &\leq \frac{2}{N_S}\left(\beta_S^T\Sigma_S\beta_S+\eta_S^T(Z_S^\dagger)^T\Sigma_SZ_S^\dagger\eta_S\right) + \eta_S^T(Z_S^\dagger)^Tx_Sx_S^TZ_S^\dagger\eta_S\nonumber\\
  &\leq \frac{2}{N_S}\left(\beta_S^T\Sigma_S\beta_S+\|\Sigma_S\|_2\|Z_S^\dagger\eta_S\|_2^2\right) + \|B_S\eta_S\|_2^2\ \ \ \ \ \ \ \ \ \ \ \ \ B_S = \sqrt{x_Sx_S^T}Z_S^\dagger.\label{exp1}
\end{align}
Both $\|Z_S^\dagger\eta_S\|_2^2$ and $\|B_S\eta_S\|_2^2$ are the quadratic forms of sub-Gaussian random vectors and can be bounded with high probability using the Hanson-Wright inequality \cite[Theorem 6.3.2]{Vershynin_2018}: 
\begin{align}
\P\left(\|Z_S^\dagger\eta_S\|^2_2>C_1\sigma^2_S\|Z_S^\dagger\|^2_F\log m\right)\leq \frac{1}{3m^2}\nonumber\\
\P\left(\|B_S\eta_S\|^2_2>C_1\sigma^2_S\|B_S\|^2_F\log m\right)\leq \frac{1}{3m^2},\label{lalala}
\end{align}
where $C_1$ is an absolute constant depending only on the sub-Gaussian norm of $\e_S/\sigma_S$, and 
\begin{align}
\|Z_S^\dagger\|_F^2 &= \tr(Z_S^\dagger(Z_S^\dagger)^T) = \frac{1}{m}\tr((m^{-1}Z_S^TZ_S)^{-1})\nonumber\\
\|B_S\|_F^2 &= \tr(B_S^TB_S) = \frac{1}{m}\tr(\left(m^{-1}Z_S^TZ_S\right)^{-1} x_Sx_S^T).\label{voc}
\end{align}
Since $Z_S^TZ_S$ is a sum of i.i.d. outer products, 
appealing to a deviation result in \cite[Theorem 2.1]{Mendelson_2006}, we obtain that under assumption \eqref{...} and for $m> s+1$, 
\begin{align}
&\P\left(\left\|\frac{1}{m}Z_S^TZ_S -\Lambda_S\right\|_2>C_2\frac{K^2\log^5 m}{\sqrt{m}}\right)\nonumber\\
=&\ \P\left(\left\|\frac{1}{m}\sum_{\ell\in [m]}X_{\ex,\ell}|_SX^T_{\ex,\ell}|_S -\E[X_SX_S^T]\right\|_2> C_2\frac{K^2\log^5 m}{\sqrt{m}}\right)<\frac{1}{3m^2},\label{cov}
\end{align}
where $C_2$ is an absolute constant. 
Combining \eqref{cov} with the matrix perturbation equality for an invertible matrix $A$ $$(A+\Delta A)^{-1} =  A^{-1}-A^{-1}\Delta AA^{-1} + o(\|\Delta A\|_2)$$ 
yields with high probability,
\begin{align}
&\left(m^{-1}Z_S^TZ_S\right)^{-1} = \Lambda_S^{-1} + E_S& \|E_S\|_2\lesssim \frac{K^2\log^5 m}{\sigma^2_{\min}(\Lambda_S)\sqrt{m}},\label{emm}
\end{align}
where $\sigma_{\min}(\Lambda_S)$ is the smallest singular value of $\Lambda_S$. 
Putting \eqref{emm}, \eqref{cov}, \eqref{voc}, \eqref{lalala} and \eqref{exp1} together, we have with probability at least $1-m^{-2}$, 
\begin{align}
  \mse_S|_{\widehat{\beta}_S}&\lesssim \frac{2}{N_S}\left[\beta_S^T\Sigma_S\beta_S+C_1\sigma_S^2\|\Sigma_S\|_2\tr(\Lambda_S^{-1} + \frac{K^2\log^5 m}{\sigma^2_{\min}(\Lambda_S)\sqrt{m}}I_s)\frac{\log m}{m}\right]\nonumber\\
&\ \ \ \ \ \  + C_1\sigma_S^2\tr(\Lambda_S^{-1}x_Sx_S^T + \frac{K^2\log^5 m}{\sigma^2_{\min}(\Lambda_S)\sqrt{m}}x_Sx_S^T)\frac{\log m}{m}\nonumber\\
&\lesssim\frac{1}{N_S}\beta_S^T\Sigma_S\beta_S +\sigma_S^2\tr(\Lambda_S^{-1}x_Sx_S^T)\frac{\log m}{m}\nonumber\\
&\leq \frac{1}{N_S}\beta_S^T\Sigma_S\beta_S +\sigma_S^2\tr(\Lambda_S^{-1}(x_Sx_S^T+\Sigma_S))\frac{\log m}{m}\nonumber\\
& = \frac{1}{N_S}\beta_S^T\Sigma_S\beta_S +(s+1)\sigma_S^2\frac{\log m}{m}.
\end{align}
This finishes the proof.

\section{Proof of Theorem \ref{thm:AETC}} \label{ap:2}

We first show that $m(B)$ diverges as $B\to\infty$ almost surely. 
To this end, it suffices to show that with probability $1$, 
\begin{align}
\sup_{t>n+1}\max_{S\subseteq [n]}\widehat{\beta}^T_S(t)\widehat{\Sigma}_S(t)\widehat{\beta}_S(t)<\infty. \label{simkey}
\end{align}
Indeed, if \eqref{simkey} is true, by definition \eqref{0001}, for almost every realization $\omega$, there exists an $L(\omega)<\infty$ such that 
\begin{align}
  \sup_{t>n+1}\max_{S\subseteq [n]}k_{1,t}(S; \omega)<L(\omega),\label{good}
\end{align}
where $\omega$ is included to stress the quantity's dependence on realization. 
The exploration stopping criterion of Algorithm \ref{alg:aETC} requires that 
\begin{align*}
m(B; \omega)\geq m_{S^*(t; \omega)}(t; \omega) = \frac{B}{c_{\ex}+\sqrt{\frac{c_{\ex}k_{1,t}(S^*(t; \omega); \omega)}{k_{2,t}(S^*(t; \omega);\omega)}}}\stackrel{\eqref{good}}{\geq}\frac{B}{c_{\ex}+\sqrt{\frac{c_{\ex}L(\omega)}{\alpha_{m(B; \omega)}}}}
\end{align*} 
Note that $m(B;\omega)$ is nondecreasing in $B$. If $m(B;\omega)$ did not diverge, then there would exist an integer $C>0$ such that $\sup_{B}m(B;\omega)<C$. Meanwhile, this also suggests that $\inf_B\alpha_{m(B;\omega)}>\alpha_C>0$, forcing the right-hand side of the above inequality to diverge; hence $m(B;\omega)$. A contradiction.
Thus, 
\begin{align}
&\lim_{B\to\infty}m(B; \omega)=\infty. 
\label{ksl1}
\end{align}
To show \eqref{simkey}, note that $\widehat{\Sigma}_S(t)$ converges to $\Sigma_S$ almost surely due to the strong law of large numbers. 
For $\widehat{\beta}_S$, the sub-exponential assumption on the distribution of $X_S$ for all $S\subseteq [n]$ (condition \eqref{...}) ensures \eqref{cov} (with $m$ replaced by $t$) for all $S\subseteq [n]$, which combined with the Borel-Cantelli lemma implies that with probability 1, $\lambda_{\min, S}(t)\to\infty$ and
\begin{align*}
&\log\lambda_{\max, S}(t) = o(\lambda_{\min, S}(t))&\forall S\subseteq [n],
\end{align*} 
where $\lambda_{\max, S}(t)$ and $\lambda_{\min, S}(t)$ are respectively the largest and smallest eigenvalue of $Z_S^TZ_S$. 
Appealing to \cite[Theorem 1]{Lai_1982}, with probability 1, $\widehat{\beta}_S(t)\to\beta_S$ as $t\to\infty$ for all $S\subseteq [n]$. 
Thus we have proved \eqref{simkey}.
 
We now work with a fixed realization $\omega$ along which $m(B; \omega)\to\infty$ as $B\to\infty$, and all estimators in \eqref{e1} converge to the true parameters as $t\to\infty$. 
Recall the empirical MSE estimated for model $S$ in round $t$, as a function of the exploration round $m$:
\begin{align*}
\widehat{\overline{\mse}}_S|_{\widehat{\beta}_S}(m;t) = \frac{k_{1,t}(S)}{B-c_{\ex}m} + \frac{k_{2,t}(S)}{m}.
\end{align*}
Take $\delta<1/2$ sufficiently small. 
Since \eqref{>} is assumed to have a unique minimizer, by consistency and a continuity argument, there exists a sufficiently large $T(\delta; \omega)$ so that, for all $t\geq T(\delta, \omega)$,  
\begin{align}
\max_{(1-\delta)m_{S^*}\leq m\leq (1+\delta)m_{S^*} }\widehat{\overline{\mse}}_{S^*}|_{\widehat{\beta}_{S^*}}(m;t)&<\min_{S\subseteq [n], S\neq S^*}\min_{0<m<B/c_\ex}\widehat{\overline{\mse}}_S|_{\widehat{\beta}_S}(m;t)\label{l1}\\
1-\delta &\leq\frac{m_S(t; \omega)}{m_S}\leq 1+\delta\ \ \ \ \ \ \ \ \ \forall S\subseteq [n].\label{l2} 
\end{align}
Since $m_{S}$ scales linearly in $B$ and $m(B;\omega)$ diverges as $B\uparrow\infty$, there exists a sufficiently large $B(\delta;\omega)$ such that for $B>B(\delta;\omega)$, 
\begin{align}
\min\left\{\min_{S\subseteq [n]}m_{S}, m(B;\omega)\right\} > 2T(\delta;\omega).\label{mkj}
\end{align}
\eqref{l2}, \eqref{mkj} and $\delta<1/2$ together imply that, at $T(\omega)$, for all $S\subseteq [n]$,
\begin{align}
m_S(T(\delta;\omega); \omega)\stackrel{\eqref{l2}, \delta<1/2}{\geq} \frac{m_S}{2}\stackrel{\eqref{mkj}}{>}T(\delta; \omega).\label{bozhong}
\end{align}
Combining \eqref{bozhong} with \eqref{l1} yields that, at $T(\omega)$, the algorithm chooses $S^*$ as the optimal model, and the corresponding estimated optimal exploration rate is larger than $T(\delta;\omega)$.
By the design of our algorithm, more exploration is needed, and \eqref{l1} and \eqref{l2} further ensure that $S^*$ will be chosen in the subsequent exploration until the stopping criterion is met, proving \eqref{rr1}.
Moreover, when the algorithm stops,  
\begin{align*}
1-\delta \leq \frac{m(B; \omega)}{m_{S^*}}\leq 1+\delta +\frac{1}{m_{S^*}}.
\end{align*}
Taking $B\to\infty$ followed by $\delta\to 0$ yields \eqref{rr2}, finishing the proof.  


\section{Proof of Theorem \ref{thm:consv}}\label{ap:3}
In this section, we provide some omitted details of the analysis in the case of vector-valued high-fidelity models. 
We first complete the proof of Theorem \ref{thm:consv}, which establishes a similar nonasymptotic convergence rate of the LRMC estimator under the $Q$-risk. Then, we give a detailed description of AETC algorithm for vector-valued high-fidelity models, and examine the estimation efficiency in the exploration phase. We end this section by providing a numerical example where very high-dimensional high-fidelity output is considered. 

For simplicity, we assume that noises $\e^{(1)}_{S}, \cdots, \e^{(k_0)}_{S}$ are jointly Gaussian; the sub-Gaussian case can be discussed similarly using analogous concentration inequalities. 
 
 \subsection{Proof of Theorem \ref{thm:consv}}
Let $ r_S = \rank(\Gamma_S)$ and $Q\in\R^{q\times k_0}$.
Similar to the calculation in \eqref{exp1}, we first rewrite $\risk_S|_{\widehat{\beta}_S}$ as 
\begin{align}
\risk_S|_{\widehat{\beta}_S} &= \frac{1}{N_S}\tr(\Sigma_S\widehat{\beta}_S^TQ^TQ\widehat{\beta}_S) + x_S^T(\widehat{\beta}_S-\beta_S)^TQ^TQ(\widehat{\beta}_S-\beta_S)x_S\nonumber\\
&\leq \frac{2}{N_S}\left[\tr(\Sigma_S\beta_S^TQ^TQ\beta_S)+\tr(\Sigma_S(\widehat{\beta}_S-\beta_S)^TQ^TQ(\widehat{\beta}_S-\beta_S))\right]\nonumber\\
&\ \ \  + x_S^T(\widehat{\beta}_S-\beta_S)^TQ^TQ(\widehat{\beta}_S-\beta_S)x_S.\label{002}
\end{align}
Conditional on $Z_S$,
\begin{align}
Q(\widehat{\beta}_S-\beta_S) = Q(\e_{S,1}, \cdots, \e_{S, m})(Z_S^\dagger)^T,\label{haoo}
\end{align}
where $\e_{S,\ell}\sim\mathcal N(0,\Gamma_S)$ are i.i.d. random vectors representing the noise vector in the $\ell$-th exploration sample. 
Since $r_S = \rank(\Gamma_S)$, by the properties of multivariate normal distributions, there exists a $P_S\in\R^{k_0\times r_S}$ such that \begin{align}
&\e_{S,\ell} = P_S\xi_{S, \ell}&P_SP_S^T = \Gamma_S\label{lkj}
\end{align}
where $\xi_{S,\ell}$ are i.i.d. normal distributions $\mathcal N(0, I_{r_S})$. 
Thus, 
\begin{align}
(\widehat{\beta}_S-\beta_S)^TQ^TQ(\widehat{\beta}_S-\beta_S)=Z_S^\dagger(\xi_{S,1}, \cdots, \xi_{S, m})^TP_S^TQ^TQP_S(\xi_{S,1}, \cdots, \xi_{S, m})(Z_S^\dagger)^T.\label{haha}
\end{align}
Note that $(\xi_{S,1}, \cdots, \xi_{S, m})$ is a Gaussian random matrix, and multiplying it on the left by a unitary matrix yields a matrix whose rows are i.i.d. $\mathcal N(0, I_m)$. 
Thus, taking the eigendecomposition $P_S^TQ^TQP_S = V\text{diag}(\lambda_1, \cdots, \lambda_{r_S})V^*$ and plugging it into \eqref{haha} yields   
\begin{align}
(\widehat{\beta}_S-\beta_S)^TQ^TQ(\widehat{\beta}_S-\beta_S)^T &= \sum_{i=1}^{r_S}\lambda_iZ_S^\dagger g_ig_i^T(Z_S^\dagger)^T\label{hap}
\end{align}
where $g_i$ are i.i.d. $\mathcal N(0, I_m)$. 
Substituting \eqref{hap} into \eqref{002} yields
\begin{align*}
\risk_S|_{\widehat{\beta}_S}\leq \frac{2}{N_S}\left(\tr(\Sigma_S\beta_S^TQ^TQ\beta_S)+\|\Sigma_S\|_2\sum_{i=1}^{r_S}\lambda_i\|Z_S^\dagger g_i\|_2^2\right)+ \sum_{i=1}^{r_S}\lambda_i\|B_S g_i\|_2^2,
\end{align*}
where $B_S$ is as defined in \eqref{exp1}. 
For $\|Z_S^\dagger g_i\|_2^2$ and $\|B_Sg_i\|_2^2$, the proof of Theorem \ref{thm:conss} tells us that for large $m$, with probability at least $1-m^{-2}$, 
\begin{align}
&\|Z_S^\dagger g_i\|_2^2\lesssim \frac{\log m}{m} = o(1)&\|B_Sg_i\|_2^2\lesssim (s+1)\frac{\log m}{m}\ \ \ \ \ \ \ \ \ \ \ \forall i\in [r_S].\label{xuemama}
\end{align}  
The proof is complete by observing $\sum_{i=1}^{r_S}\lambda_i = \tr(P_S^TQ^TQP_S) = \tr(Q\Gamma_SQ^T)$.

\subsection{AETC for vector-valued high-fidelity models}
We now describe the analog of Algorithm \ref{alg:aETC}. 
The only modification occurs when we estimate the model variance structure. 
Instead of estimating the model variance $\sigma^2_S$ in \eqref{e1}, we need to estimate the model covariance matrix $\Gamma_S$, which can be computed as the sample covariance of the residuals:
\begin{align}
\widehat{\Gamma}_S(t) = \frac{1}{t-s-1}\sum_{\ell\in [t]}\left(f(Y_\ell) - \widehat{\beta}_SX_{\ex,\ell}|_S \right)\left(f(Y_\ell) - \widehat{\beta}_SX_{\ex,\ell}|_S\right)^T. \label{rep-noise}
\end{align}

Similar empirical estimates in \eqref{0001} can be defined as 
\begin{align}
\tilde{k}_{1,t}(S) &=c_{\ext}(S)\tr(\widehat{\Sigma}_S(t)\widehat{\beta}_S^T(t)Q^TQ\widehat{\beta}_S(t))\label{<<<}\\
\tilde{k}_{2,t}(S) &= \tr(Q\widehat{\Gamma}_S(t)Q^T)\tr(\widehat{x}_S(t)\widehat{x}_S^T(t)\widehat{\Lambda}_S^{-1}(t))\nonumber
\end{align}
and
\begin{align}
&\tilde{m}_S(t) = \frac{B}{c_{\ex}+\sqrt{\frac{c_{\ex}\tilde{k}_{1,t}(S)}{\tilde{k}_{2,t}(S)}}}, \ \ \ \widehat{\overline{\risk}}_S|_{\widehat{\beta}_S}(m;t) = \frac{\tilde{k}_{1,t}(S)}{B-c_{\ex}m}+ \frac{\tilde{k}_{2,t}(S)}{m}\label{00011},\\
&\overline{\risk}_S^*|_{\widehat{\beta}_S}(t) = \frac{\left(\sqrt{\tilde{k}_{1,t}(S)}+\sqrt{c_{\ex}\tilde{k}_{2,t}(S)}\right)^2}{B}.\label{00010}
\end{align}
The analog of Algorithm \ref{alg:aETC} for vector-valued high-fidelity models can be summarized as follows:
\medskip

\begin{algorithm}[H]
\hspace*{\algorithmicindent} \textbf{Input}: $B$: total budget, $c_i$: cost parameters, $\alpha_t\downarrow 0$: regularization parameters \\
    \hspace*{\algorithmicindent} \textbf{Output}: $(\pi_t)_t$
 \begin{algorithmic}[1]
\STATE compute the maximum exploration round $M = \lfloor B/c_{\ex}\rfloor$
   \FOR{$t \in [n+2]$}{
 \STATE $\pi_t = a_{\ex}([n])$
 }
\ENDFOR
 \WHILE{$n+2\leq t\leq M$}
 \FOR{$S\subseteq [n]$}{
 \STATE{compute $\tilde{k}_{1,t}(S), \tilde{k}_{2,t}(S)$ using \eqref{e1}, \eqref{rep-noise} and \eqref{<<<}, and set $\tilde{k}_{2,t}(S)\gets \tilde{k}_{2,t}(S) +\alpha_t$}
 \STATE{compute $\tilde{m}_S(t)$ using \eqref{00011}}
 \STATE{compute the optimal $Q$-risk using \eqref{00011}: $h_S(t) =\widehat{\overline{\risk}}_S|_{\widehat{\beta}_S}(\tilde{m}_S(t)\vee t; t)$}
 }
 \ENDFOR
 \STATE {find the optimal model $S^*(t) = \argmin_{S\subseteq [n]}h_S(t)$}
\IF{$\tilde{m}_{S^*(t)}(t)>t$}
        \STATE $\pi_{t+1} = a_{\ex}([n])$ and $t \gets t + 1$
    \ELSE
        \STATE  $\pi_{t+1} = a_{\ext}(S^*(t))$ and $t \gets M+1$
     \ENDIF
 \ENDWHILE
 \end{algorithmic}
\caption{AETC algorithm for multifidelity approximation (vector-valued case)} 
\label{alg:aETC_mul}
\end{algorithm}

Applying a similar argument as the proof of Theorem \ref{thm:AETC}, one can show that the exploitation model produced by Algorithm \ref{alg:aETC_mul} converges almost surely to
\begin{align}
\tilde{S}^* =\argmin_{S\subseteq [n]} \left(\sqrt{\tilde{k}_{1,t}(S)}+\sqrt{c_{\ex}\tilde{k}_{2,t}(S)}\right)^2\label{><}
\end{align}
with optimal exploration as $B\to\infty$, where
\begin{align*}
\tilde{k}_1(S) &=c_{\ext}(S)\tr(\Sigma_S(t)\beta_S^T(t)Q^TQ\beta_S(t))\\
\tilde{k}_2(S) &= \tr(Q\Gamma_S(t)Q^T)\tr(x_S(t)x_S^T(t)\Lambda_S^{-1}(t)).
\end{align*}

\subsection{Efficient estimation}
For large $k_0$, both $\beta_S$ and $\Gamma_S$ are high-dimensional, which may not admit a good global estimation for small exploration rate $m$. However, the parameters $\tilde{k}_1(S)$ and $\tilde{k}_2(S)$ used in decision-making are scalar-valued and only involve marginals of the high-dimensional parameters. 
In the following theorems, we will justify that relative accuracy of the plug-in estimators for both quantities defined in \eqref{<<<} is independent of $k_0$ under suitable assumptions.

\begin{Th}\label{thm:est1}  
Under the same condition as Theorem \ref{thm:consv} and for large $t> \max\{s+1,5\}$, it holds with probability at least $1-t^{-2}$ that 
\begin{align}
\frac{\left |\tr(Q\widehat{\Gamma}_S(t)Q^T) - \tr(Q\Gamma_SQ^T)\right |}{\tr(Q\Gamma_SQ^T)}&\lesssim \frac{\log t}{\sqrt{t-s-1}},\label{0a}
\end{align}
where the implicit constant in $\lesssim$ is universal. 
\end{Th}

For the estimation of $\tr(Q\beta_S\beta_S^TQ)$, for convenience we consider a model $S$ without intercept:

\begin{Th}\label{thm:est}  
Assume that \eqref{...} holds for the $2$-Orlicz norm, and
\begin{align}
\frac{\tr(Q\Gamma_SQ^T)}{\tr(Q\beta_S\beta_S^TQ)}\lesssim\mathcal O(1). \label{last}
\end{align}
Suppose that $S$ does not include the intercept term and the corresponding covariance matrix $\Sigma_S$ is nonsingular. 
Under the same condition as Theorem \ref{thm:consv} and for large $t$, it holds with probability at least $1-t^{-2}$ that 
\begin{align}
\frac{\left|\tr(\widehat{\Sigma}_S(t)\widehat{\beta}_S^T(t)Q^TQ\widehat{\beta}_S(t))-\tr(\Sigma_S\beta_S^TQ^TQ\beta_S)\right|}{\tr(\Sigma_S\beta_S^TQ^TQ\beta_S)}&\lesssim \kappa(\Sigma_S)\sqrt{\frac{\log t}{t}},\label{0b}
\end{align}
where $\kappa(\Sigma_S)$ is the condition number of $\Sigma_S$, and the implicit constant is independent of $t$ and $k_0$.  
\end{Th}

\begin{Rem}
Under additional assumption on the Frobenius norm of $Q\beta_S$ and its submatrices, 
similar results can be obtained for $S$ containing the intercept, with $\kappa(\Sigma_S)$ replaced by $\kappa(M_{1,1}(\Sigma_{S}))$, where $M_{i,j}(\cdot )$ denotes the $(i,j)$ minor of a matrix.  
\end{Rem}

\begin{Rem}
The constant $\kappa(\Sigma_S)$ in \eqref{0b}, despite depending on the singular values of $\Sigma_S$, is independent of $k_0$.  This combined with \eqref{0a} implies that the estimation in Algorithm \ref{alg:aETC_mul} is relatively efficient regardless of the response dimension $k_0$. 
\end{Rem}

\begin{proof}[Proof of Theorem \ref{thm:est1}]
Rewriting \eqref{rep-noise} using \eqref{haoo} and \eqref{newass}, we have
\begin{align}
Q\widehat{\Gamma}_S(t)Q^T &= \frac{1}{t-s-1}Q(\e_{S,1}, \cdots, \e_{S, t})(I_t - Z_SZ_S^\dagger)(\e_{S,1}, \cdots, \e_{S, t})^TQ^T\nonumber\\
&\stackrel{\eqref{lkj}}{ = }\frac{1}{t-s-1}QP_S(\xi_{S,1}, \cdots, \xi_{S, t})(I_t - Z_SZ_S^\dagger)(\xi_{S,1}, \cdots, \xi_{S, t})^TP^T_SQ^T.\label{sleep}
\end{align}
Note that $I_t - Z_SZ_S^\dagger$ is a $(t-s-1)$-dimensional (random) orthogonal projection matrix, i.e., $I_t - Z_SZ_S^\dagger = UU^T$ for some $U\in\R^{t\times (t-s-1)}$ with orthonormal columns. 
Thus, by the rotational invariance of multivariate normal distributions,
\begin{align}
&(\xi_{S,1}, \cdots, \xi_{S, t})(I_t - Z_SZ_S^\dagger)(\xi_{S,1}, \cdots, \xi_{S, t})^T\nonumber\\
\stackrel{\mathcal D}{=}&\ (\zeta_{S,1}, \cdots, \zeta_{S, t-s-1})(\zeta_{S,1}, \cdots, \zeta_{S, t-s-1})^T,\label{cry}
\end{align}
where $\zeta_{S,\ell}, \ell\in [t-s-1]$ are (fixed) i.i.d. normal vectors $\mathcal N(0, I_{r_S})$, and $\stackrel{\mathcal D}{=}$ denotes equality in distribution. 
Substituting \eqref{cry} into \eqref{sleep} and taking the trace yields
\begin{align*}
\tr(Q\widehat{\Gamma}_S(t)Q^T)&\stackrel{\mathcal D}{=} \frac{1}{t-s-1}\tr(QP_S(\zeta_{S,1}, \cdots, \zeta_{S, t-s-1})(\zeta_{S,1}, \cdots, \zeta_{S, t-s-1})^TP^T_SQ^T)\\
& = \frac{1}{t-s-1}\tr((\zeta_{S,1}, \cdots, \zeta_{S, t-s-1})^TP^T_SQ^TQP_S(\zeta_{S,1}, \cdots, \zeta_{S, t-s-1}))\\
&\stackrel{}{=}\frac{1}{t-s-1}\sum_{i=1}^{t-s-1}\zeta_{S,i}^TP_S^TQ^TQP_S\zeta_{S,i}\\
& = \frac{1}{t-s-1}\bm\zeta_{S,t-s-1}^T\bm\Omega_{t-s-1}\bm\zeta_{S,t-s-1}
\end{align*}
where $\bm {\zeta}_{S,t-s-1} = (\zeta_{S,1}^T, \cdots, \zeta_{S,t-s-1}^T)^T$ is a standard multivariate normal vector in $\R^{r_S(t-s-1)}$ and
\begin{align}
\bm\Omega_{t-s-1} = \underbrace{\begin{pmatrix}
P_S^TQ^TQP_S\\
&&\ddots\\
&&&P_S^TQ^TQP_S\\
\end{pmatrix}}_{\text{$t-s-1$ diagonal blocks}}
\end{align}
Apply the Hanson-Wright inequality \cite[Theorem 6.2.1]{Vershynin_2018} to $\bm\zeta_{S,t-s-1}^T\bm\Omega_{t-s-1}\bm\zeta_{S,t-s-1}$ and we yield that for every $t>s+1$ and $\delta\geq 4\|\bm\Omega_{t-s-1}\|_F$, 
\begin{align*}
&\P\left(|\bm\zeta_{S,t-s-1}^T\bm\Omega_{t-s-1}\bm\zeta_{S,t-s-1} - (t-s-1)\tr(P_S^TQ^TQP_S)|>\delta\right)\\
&\ \leq 2\exp\left(-\frac{C\delta}{4\|\bm\Omega_{t-s-1}\|_F}\right),
\end{align*} 
where $C\leq 1$ is an absolute constant, and $4$ comes from an upper bound for the square of the sub-Gaussian norm of the standard normal distribution. 
For $t>\max\{s+1, 5\}$, taking $$\delta =C^{-1}\|\bm\Omega_{t-s-1}\|_F\log(2t^2)$$ yields that with probability at least $1-t^{-2}$, 
\begin{align*}
\left |\bm\zeta_{S,t-s-1}^T\bm\Omega_{t-s-1}\bm\zeta_{S,t-s-1} - (t-s-1)\tr(P_S^TQ^TQP_S)\right |\leq C^{-1}\log (2t^2)\|\bm\Omega_{t-s-1}\|_F\label{awd}
\end{align*}
Dividing both sides by $(t-s-1)\tr(Q\Gamma_SQ^T)$ and using $\tr(P_S^TQ^TQP_S) = \tr(Q\Gamma_SQ^T)$ finishes the proof of \eqref{0a}:
\begin{align}
\frac{\left |\tr(Q\widehat{\Gamma}_S(t)Q^T) - \tr(Q\Gamma_SQ^T)\right |}{\tr(Q\Gamma_SQ^T)}&\leq \frac{\|\bm\Omega_{t-s-1}\|_F}{\tr(Q\Gamma_SQ^T)\sqrt{t-s-1}}\frac{C^{-1}\log (2t^2)}{\sqrt{t-s-1}}\nonumber\\
& = \frac{\|Q\Gamma_SQ^T\|_F}{\tr(Q\Gamma_SQ^T)}\frac{C^{-1}\log (2t^2)}{\sqrt{t-s-1}}\nonumber\\
& \leq  \frac{3C^{-1}\log t}{\sqrt{t-s-1}}.\label{aiden}
\end{align}
The proof is complete.  
\end{proof}

\begin{proof}[Proof of Theorem \ref{thm:est}]
By the triangle inequality and Cauchy-Schwarz inequality, 
\begin{align}
& \left|\tr(\widehat{\Sigma}_S(t)\widehat{\beta}_S^T(t)Q^TQ\widehat{\beta}_S(t))-\tr(\Sigma_S\beta_S^TQ^TQ\beta_S)\right|\nonumber\\
\leq &\ \left|\tr(\widehat{\Sigma}_S(t)\widehat{\beta}_S^T(t)Q^TQ\widehat{\beta}_S(t))-\tr(\widehat{\Sigma}_S(t)\beta_S^TQ^TQ\beta_S)\right|\nonumber\\
&\ +\left|\tr(\widehat{\Sigma}_S(t)\beta_S^TQ^TQ\beta_S)-\tr(\Sigma_S\beta_S^TQ^TQ\beta_S)\right|\nonumber\\
\leq &\ \|\widehat{\Sigma}_S(t)\|_F\|\widehat{\beta}_S^T(t)Q^TQ\widehat{\beta}_S(t)-\beta_S^TQ^TQ\beta_S\|_F+\|\widehat{\Sigma}_S(t)-\Sigma_S\|_F\|\beta_S^TQ^TQ\beta_S\|_F\nonumber\\
\leq &\ \|\widehat{\Sigma}_S(t)\|_F\|\widehat{\beta}_S^T(t)Q^TQ\widehat{\beta}_S(t)-\beta_S^TQ^TQ\beta_S\|_F+\|\widehat{\Sigma}_S(t)-\Sigma_S\|_F\tr(\beta_S^TQ^TQ\beta_S)\nonumber\\
\leq &\ \sqrt{s}\left(\|\widehat{\Sigma}_S(t)\|_2\|\widehat{\beta}_S^T(t)Q^TQ\widehat{\beta}_S(t)-\beta_S^TQ^TQ\beta_S\|_F+\|\widehat{\Sigma}_S(t)-\Sigma_S\|_2\tr(\beta_S^TQ^TQ\beta_S)\right).\label{mor}
\end{align}
Let $a_i$ and $b_i$ denote the $i$-th column of $Q\beta_S$ and $Q\widehat{\beta}_S(t)$, respectively, i.e., $Q\beta_S = (a_1, \cdots, a_{s})$ and $Q\widehat{\beta}_S(t) = (b_1, \cdots, b_{s})$. Then,
\begin{align*}
&\|\widehat{\beta}_S^T(t)Q^TQ\widehat{\beta}_S(t)-\beta_S^TQ^TQ\beta_S\|_F^2 \\
=&\  \sum_{i, j \in [s]}\left(\langle a_i, a_j\rangle - \langle b_i, b_j\rangle\right)^2\\
=&\ \sum_{i, j \in [s]}\left(\langle a_i-b_i, a_j\rangle +\langle b_i, a_j-b_j\rangle\right)^2\\
\leq &\ 2\sum_{i, j \in [s]}\left(\langle a_i-b_i, a_j\rangle^2 +\langle b_i, a_j-b_j\rangle^2\right)\\
\leq &\ 2\sum_{i, j \in [s]}\left(\|a_i-b_i\|_2^2\|a_j\|_2^2 + \|b_i\|_2^2\|a_j-b_j\|_2^2\right)\\
\leq &\ 2\sum_{i, j \in [s]}\left(\|a_i-b_i\|_2^2\|Q\beta_S\|_2^2 + \|Q\widehat{\beta}_S(t)\|_2^2\|a_j-b_j\|_2^2\right)\\
\leq &\ 4s\left(\|Q\beta_S\|_2+\|Q\widehat{\beta}_S(t)\|_2\right)^2\|Q\beta_S - Q\widehat{\beta}_S(t)\|_F^2.
\end{align*}
Substituting this into \eqref{mor} yields
\begin{align}
&\left|\tr(\widehat{\Sigma}_S(t)\widehat{\beta}_S^T(t)Q^TQ\widehat{\beta}_S(t))-\tr(\Sigma_S\beta_S^TQ^TQ\beta_S)\right|\nonumber\\
\leq&\ 2s\|\widehat{\Sigma}_S(t)\|_2 \left(\|Q\beta_S\|_2+\|Q\widehat{\beta}_S(t)\|_2\right)\|Q\beta_S - Q\widehat{\beta}_S(t)\|_F\nonumber\\
&\ +\sqrt{s}\|\widehat{\Sigma}_S(t)-\Sigma_S\|_2\tr(\beta_S^TQ^TQ\beta_S).\label{feel}
\end{align}
Under the strengthened assumption of \eqref{...}, we can obtain a better bound for the relative error of the sample covariance estimator $\widehat{\Sigma}_S$ by appealing to \cite[Exercise 9.2.5]{Vershynin_2018}:
With probability at least $1-1/2t^2$, $\|\widehat{\Sigma}_S(t)-\Sigma_S\|_2/\|\Sigma_S\|_2\lesssim \sqrt{\log t/t}$, i.e., $\|\Sigma_S\|_2 + \|\widehat{\Sigma}_S(t)\|_2\lesssim \|\Sigma_S\|_2$. 
On the other hand,  
\begin{align}
\|Q\beta_S - Q\widehat{\beta}_S(t)\|_F &= \sqrt{\tr((\widehat{\beta}_S(t)-\beta_S)^TQ^TQ(\widehat{\beta}_S(t)-\beta_S))}\nonumber\\
&\stackrel{\eqref{hap}, \eqref{xuemama}}{\lesssim}\sqrt{\tr(Q\Gamma_SQ^T)}\sqrt{\frac{\log t}{t}}\label{00002xx},
\end{align}
where \eqref{00002xx} holds with probability at least $1-1/2t^2$ for large $t$.
In this case we can verify that,
\begin{align}
\|Q\beta_S\|_2+\|Q\widehat{\beta}_S(t)\|_2&\leq 2\|Q\beta_S\|_F+\|Q\widehat{\beta}_S(t)-Q\beta_S\|_F\\\label{00003}
&\lesssim \|Q\beta_S\|_F\left(2+\sqrt{\frac{\tr(Q\Gamma_SQ^T)}{\tr(Q\beta_S\beta_S^TQ)}\frac{\log t}{t}}\right)\nonumber\\
&\stackrel{\eqref{last}}{\lesssim}\|Q\beta_S\|_F = \sqrt{\tr(\beta_SQ^TQ\beta_S)}\nonumber.
\end{align}
Thus, with probability at least $1-t^{-2}$, 
\begin{align}
&\left|\tr(\widehat{\Sigma}_S(t)\widehat{\beta}_S^T(t)Q^TQ\widehat{\beta}_S(t))-\tr(\Sigma_S\beta_S^TQ^TQ\beta_S)\right|\nonumber\\
\lesssim&\ \|\Sigma_S\|_2\sqrt{\tr(\beta_SQ^TQ\beta_S)}\sqrt{\tr(Q\Gamma_SQ^T)}\sqrt{\frac{\log t}{t}}+\tr(\beta_S^TQ^TQ\beta_S)\|\Sigma_S\|_2\sqrt{\frac{\log t}{t}}\nonumber\\
&\stackrel{\eqref{last}}{\lesssim}\ \|\Sigma_S\|_2\tr(\beta_S^TQ^TQ\beta_S)\sqrt{\frac{\log t}{t}}.\label{layton}
\end{align}
Dividing \eqref{layton} by $\tr(\Sigma_S\beta_S^TQ^TQ\beta_S)$ finishes the proof:
\begin{align*}
\frac{\left|\tr(\widehat{\Sigma}_S(t)\widehat{\beta}_S^T(t)Q^TQ\widehat{\beta}_S(t))-\tr(\Sigma_S\beta_S^TQ^TQ\beta_S)\right|}{\tr(\Sigma_S\beta_S^TQ^TQ\beta_S)}&\lesssim\frac{\|\Sigma_S\|_2\tr(\beta_S^TQ^TQ\beta_S)}{\tr(\Sigma_S\beta_S^TQ^TQ\beta_S)}\sqrt{\frac{\log t}{t}}\\
&\leq\kappa(\Sigma_S)\sqrt{\frac{\log t}{t}}.
\end{align*}
\end{proof}

\section{Numerical simulation of Section \ref{ppde}}\label{app:num-setup-1}

In Section \ref{ppde}, the partial differential equation \eqref{elliptic_pde_stochastic} for a fixed $\bm p$ is frequently solved with the Finite Element Method. We use standard bilinear square isotropic finite elements on a rectangular mesh. After spatial discretization, this results in a linear system in the form of 
\begin{align}\label{eq:fem}
\bm K \bm u = \bm f
\end{align}
where $\bm K$ and $\bm f$ are stiffness matrix and force vector, respectively, and $\bm u$ is the vector of nodal displacements. Our output quantity of interest is the scalar compliance, defined as
\begin{align}\label{eq:compliance}
  \C = \bm u^T \bm K \bm u.
\end{align}

In this example, we form a multifidelity hierarchy through coarsening of the discretization: Consider $n=7$ mesh resolutions with mesh sizes $h = \{1/(2^{8-L})\}_{L=1}^{7}$ where $L$ denotes the level. According to this hierarchy of meshes, the mesh associated with $L=1$ yields the most accurate model (highest fidelity), which is taken as the high-fidelity model in our experiments. The Poisson's ratio is $\nu=0.3$.
The finite element computations in this paper are performed in MATLAB using part of the publicly available code for topology optimization~\cite{Andreassen2011}. 
      
To model the uncertainty, we consider a random field for the elastic modulus via the Karhunen-Lo\'{e}ve (KL) expansion \eqref{KL_numerical}, where $\delta=0.5, E_0=1$ are constants. The random variables $p_i$ are uniformly distributed on $[-1,1]$ and the eigenvalues $\lambda_i$ and basis functions $E_i$ are taken from the analytical expressions for the eigenpairs of an exponential kernel on $D=[0,1]^2$. In one dimension, i.e., $D_1=[0,1]$, the eigenpairs are 
\begin{equation}\label{eigenpair}
\lambda^{1D}_i = \frac{2}{w^2_i+1} \quad b^{1D}_i = A_i (\sin(w_i \bm x)+ w_i \cos (w_i \bm x)) \qquad i \in \mathbb{N}
\end{equation}
where $w_i$ are the positive ordered solutions to
\begin{equation}
\tan(w)=\frac{2w}{w^2-1},
\end{equation}
See, e.g., \cite{Teckentrup2015}. In our simulations, we use the approximation $w_i \approx i \pi$, which is the asymptotic behavior of these solutions. We generate two-dimensional $D=[0,1]^2$ eigenpairs via tensorization of the one-dimensional pairs, 
\begin{equation}
\lambda_{ij} = {\lambda^{1D}_i}{\lambda^{1D}_j}  \quad E_{ij} = b_i^{1D} \otimes b_j^{1D} \qquad i,j \in \mathbb{N}
\end{equation}
where $\otimes$ denotes the tensor product. 
In this example, $d=4$ corresponds to the tensor-product indices $(i,j) \in [2] \times [2]$. 
When $d=20$, we consider the first 20 terms of the total-order polynomial indices $\{\{0, 0\}, \{10, 01\}, \{20, 11, 02\}, · · · \}$
To compute the elastic modulus for different resolutions, we evaluate the analytical basis functions in \eqref{eigenpair} at different resolutions.

\section{Low-fidelity models of Section \ref{ssec:rom-gp}}\label{app:num-setup-2}

Models $X^{(i)}$ for $i \in [6]$ in Section \ref{ssec:rom-gp} are formed as GP predictors. We construct a GP emulator for $C$ as a function of the parameters $\bm p$, and use the GP mean as a low-fidelity emulator.
  Given $\mathcal{D}^{\circ} = \{(\bm p^{\circ}_i, \bm Y(p^{\circ}_i))\}_{i=1}^{n_T}$ training samples with $\bm P^{\circ} =\{\bm p^{\circ}_i\}_{i=1}^{n_T}$ the sampling nodes and $\bm Y^{\circ} =\{\bm Y(p^\circ_i)\}_{i=1}^{n_T}$ the observational data, let $\bm K_{GP}$ denote a kernel matrix with entries $[\bm K_{GP}]_{ij} = k(\bm p_i, \bm p_j, \bm \theta)$ for a given kernel function $k(\bm p, \bm p'): \mathbb{R}^d \times \mathbb{R}^d \rightarrow \mathbb{R}$ is the kernel function and $\bm \theta$ is a hyperparameter that tunes the kernel. We adopt a standard GP training procedure, which determines $\bm \theta$ by maximizing a log-likelihood objective function,
\begin{equation}\label{log-lke}
\begin{array}{l l}
    \mathcal{L}(\bm \theta) & = \log p(\bm Y| \bm P) = \log \mathcal{N}(\bm Y|\bm 0, \bm K_{GP}(\bm P, \bm \theta) + \lambda\bm I) \\
    \\
    & =\displaystyle \frac{1}{2} \bm Y^{T} (\bm K_{GP} + \lambda\bm I)^{-1} \bm Y   + \frac{1}{2} \log| \bm K_{GP} + \lambda\bm I | + \frac{n}{2} \log(2\pi) \\
    \end{array}
\end{equation}
where $\bm I$ is the identity matrix and $\lambda \geq 0$ is a constant. The parameter $\lambda$ is introduced to model the effect of noise in the data, but in this example we do not consider any noise and therefore we set $\lambda=0$. We choose an exponential kernel function,
\begin{equation}
k(\bm p, \bm p') = \exp (-\|\bm p- \bm p' \|_2/\theta)
\end{equation}
with a scalar hyperparameter $\theta$. Once the optimal hyperparameter $\theta^{\ast}$ is obtained from optimization of the log likelihood \eqref{log-lke}, we compute the approximate GP sample on the test node $\bm p^{\ast}$ via
\begin{equation}\label{GP_regr}
  X(p^\ast) = \bar{\bm Y}^{\circ} + \bm K_{i\circ}^T \bm K_{GP}^{-1}(\bm Y^{\circ}-\bar{\bm Y}^{\circ}),
\end{equation}
where $\bm K_{i\circ} = k(\bm p^{\ast}, \bm P^{\circ}) $ is a vector obtained by evaluating the kernel function with the training samples $\bm P^{\circ}$ and the particular test sample $\bm p^{\ast}$, and $\bar{\bm Y}^{\circ}$ is the sample mean of the training data.

    Our training data are generated via compliance samples from the high-fidelity model. The fidelity of each GP emulator in this case are determined by the number of training samples $n_T$ as well as the optimal/nonoptimal choice of hyperparameters. We generate three fidelities by choosing $n_T=10,~100,~1000$ training samples and compute the optimal hyperparameter $\theta$ via the minimization of \eqref{log-lke}. The approximated models from the optimal hyperparameter calculation are indexed by $X^{(1)},~X^{(3)},~X^{(5)}$ for $n_T = 10, 100, 1000$, respectively. We then generate three more models with the same number of training samples as before, but with a nonoptimal hyperparameter as $\theta^{\ast} \gets 0.1 \theta^{\ast}$. We index these nonoptimal models as $X^{(2)},~X^{(4)},~X^{(6)}$. The cost of all 6 of these models is given by the cost of training, optimization, and evaluation of \eqref{GP_regr} averaged over $2 \times 10^4$ different values of $\bm p^\ast$. 

Models $X^{(i)}$ for $i = 7, \ldots, 12$ are defined via projection-based model reduction using proper orthogonal decomposition (POD). 
Let the matrix $$\bm U^H = [\bm u^H(\bm p^{(1)}) | \bm u^H(\bm p^{(2)}) | \ldots | \bm u^H(\bm p^{(n)}) ] \in \mathbb{R}^{N_h \times N_s}$$ be comprised of high-fidelity nodal displacement vectors on parametric samples $\bm P = \{\bm p^{(i)} \}_{i=1}^{N_s}$. In this notation, and the current example, $N_h$ is the number of high-fidelity finite element degrees of freedom and $N_s$ is the number of parametric samples. 
We form emulators by projecting $\bm u^H$ onto $k$ basis vectors collected as columns vectors into a matrix $\bm V_k$,
\begin{align}
  \bm u^H(\bm p^{(i)}) &\approx \bm V_k \bm w(\bm p^{(i)}), & \bm V_k \in \R^{N_h \times k}.
\end{align}
As in typical POD approaches, we choose $\bm V_k$ as the dominant $k$ left-singular vector of $\bm U^H$.
The POD coefficients $\bm w$ are computed as a Galerkin projection of the original high-fidelity finite element system, 
\begin{equation}
\begin{array}{l l l}
(\bm V_k^T \bm K \bm V_k) \bm w &=& \bm V^T_k \bm f.
\end{array}
\end{equation}
Using the notation $\bm K_k \equiv \bm V_k^T \bm K \bm V_k$, we compute an approximate compliance via $\C = \bm U^T \bm K \bm U \simeq \bm w^T \bm K_k \bm w$.

We generate six more models, $X^{(7)},\ldots,~X^{(12)}$, using the procedure above by making five choices for the reduced dimension $k$: $k = \{1,~2,~3,~4,~5,~10\}$.
We ignore the cost of generating the matrix $\bm U^H$ and take as cost only the computation time for solving the linear system  $\bm K_k \bm w = \bm F_k$ as well as the approximate compliance calculation $\C \simeq \bm w^T \bm K_k \bm w$ averaged $2 \times 10^4$ realizations of $\bm p$ as in the GP case.

\section{A larger scale problem}\label{aG}
Under the same setup as the numerical experiments in Section \ref{ppde},
we now apply Algorithm \ref{alg:aETC_mul} to a larger scale problem where the high-fidelity output is high-dimensional. 
Fix the budget as $B = 2\times 10^5$, and let 
\begin{align*}
&f(Y) =  \bar{\bm u}^{(0)}\in\R_+^{2601}&Q = I_{2601}.
\end{align*}
In this case, the number of affordable samples by the (high-fidelity) MC estimator is $\lfloor B/4096\rfloor = 48$. 
We compare AETC and MC algorithms for approximating $\E[f(Y)]$. 
In our experiment, the number of exploration rounds $m$ chosen by the AETC algorithm in the case of the square domain and the L-shape domain is $26$ and $12$, respectively.
The respective exploitation models are $f(Y)\sim X^{(4)}+X^{(5)}+X^{(6)}+\text{intercept}$ and $f(Y)\sim X^{(3)}+X^{(4)}+X^{(5)}+X^{(6)}+\text{intercept}$, and the affordable exploitation samples for the chosen model are $2762$ and $1581$, respectively. 
$\frac{\tr(Q\Gamma_SQ^T)}{\tr(Q\beta_S\beta_S^TQ)}$ in both cases are of order $10^{-6}$, which satisfies the assumption \eqref{last}. 
The condition number of the $(1,1)$ minor of $\Sigma_S$ are $712.03$ and $6486.57$ for the square and the L-shape domain, respectively. 
Although these relatively large quantities appear in the worst-case estimate \eqref{0b}, we observe that the performance of the algorithm in practice is more accurate than these estimates suggest.

We investigate exploitation error conditioned on the exploration above: We apply the trained model (the model selected during exploration) for exploitation $500$ times, and for each estimate we compute the total error (computed in relation to a ground truth MC run over $50000$ independent high-fidelity samples, which is visualized in Figure \ref{fig:24}), i.e., the squared $\ell_2$ norm of the difference between $\E[f(Y)]$ and its estimate.  
We then find the $0.05$-$0.5$-$0.95$ quantiles of this scalar total error, and plot the spatial pointwise error at the realization identified by the total error quantiles.
The results are compared to the budget-$B$ MC estimator in Figure \ref{fig:2}. 
Within the uncertainty, the AETC algorithm enjoys a superior performance overall, on both geometries.

\begin{figure}[htbp]
\begin{center}
  \includegraphics[width=0.45\textwidth]{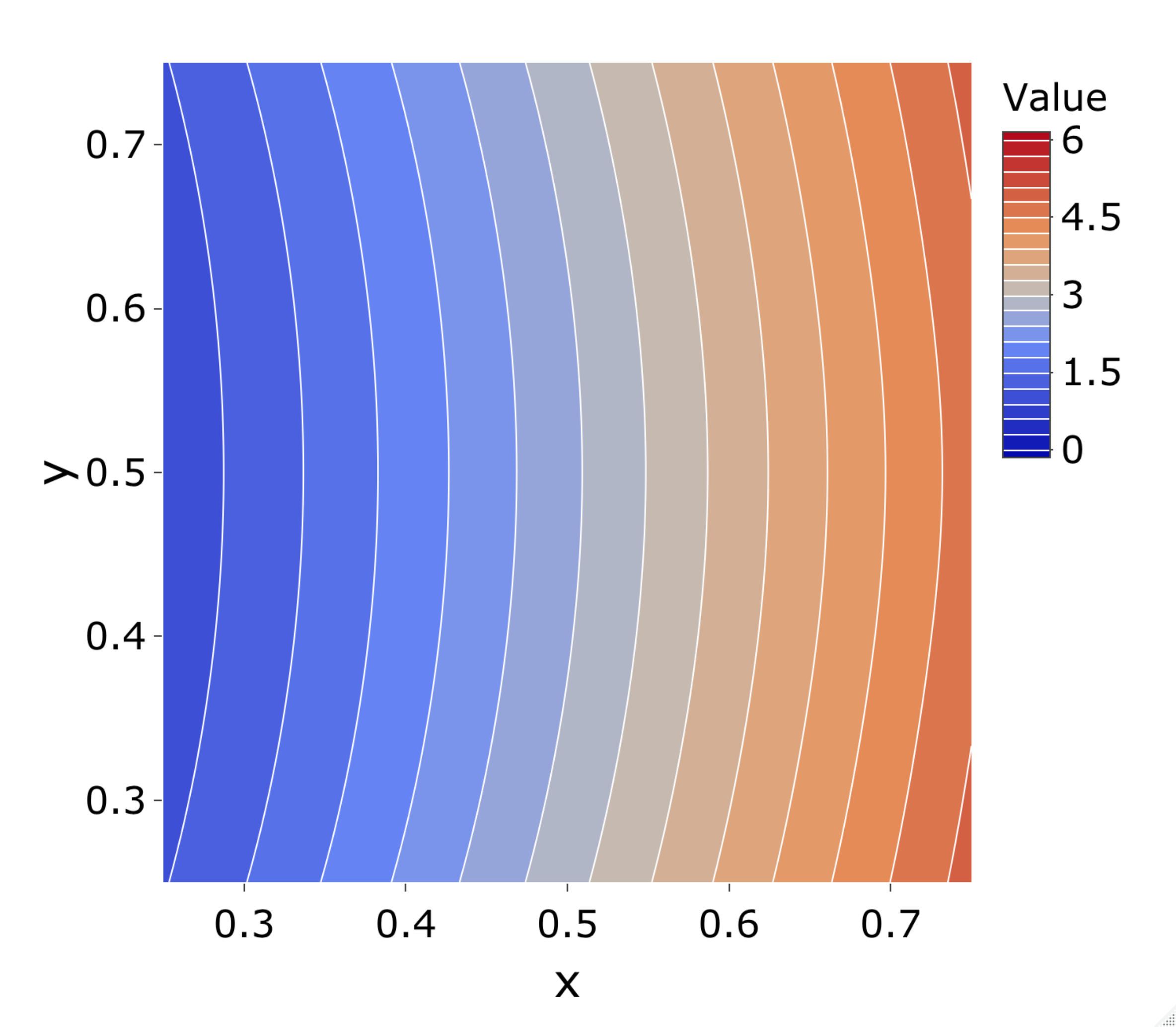}\hspace{1cm}
  \includegraphics[width=0.45\textwidth]{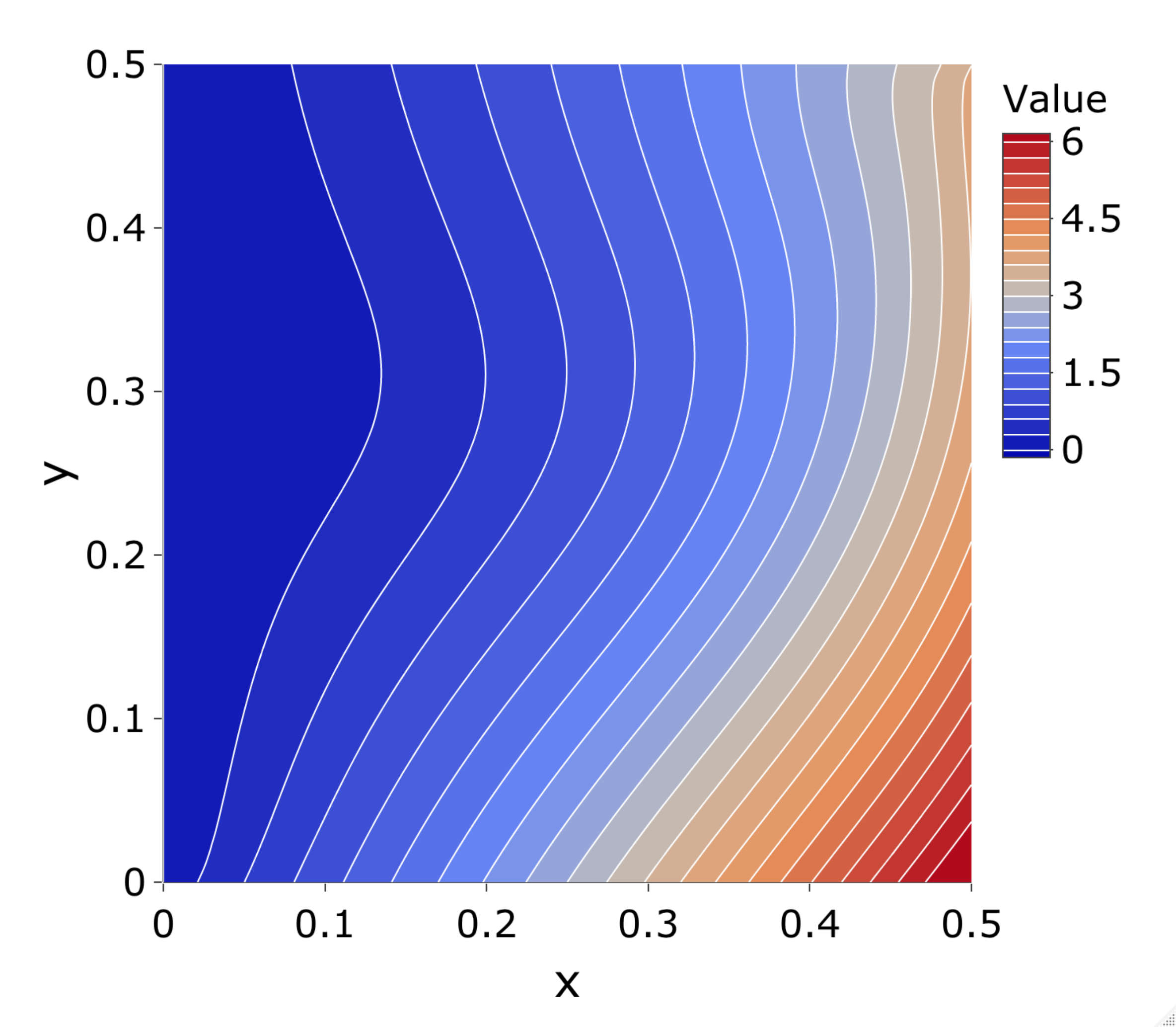}
  \caption{\small Visualization of the ground truth $\E[f(Y)]$ computed by the MC over $50000$ independent samples in the case of the 
square domain (\textbf{Left}) and the L-shape domain (\textbf{Right}).}
  \label{fig:24}
  \end{center}
\end{figure}

\begin{figure}[htbp]
\begin{center}
  \includegraphics[width=0.32\textwidth]{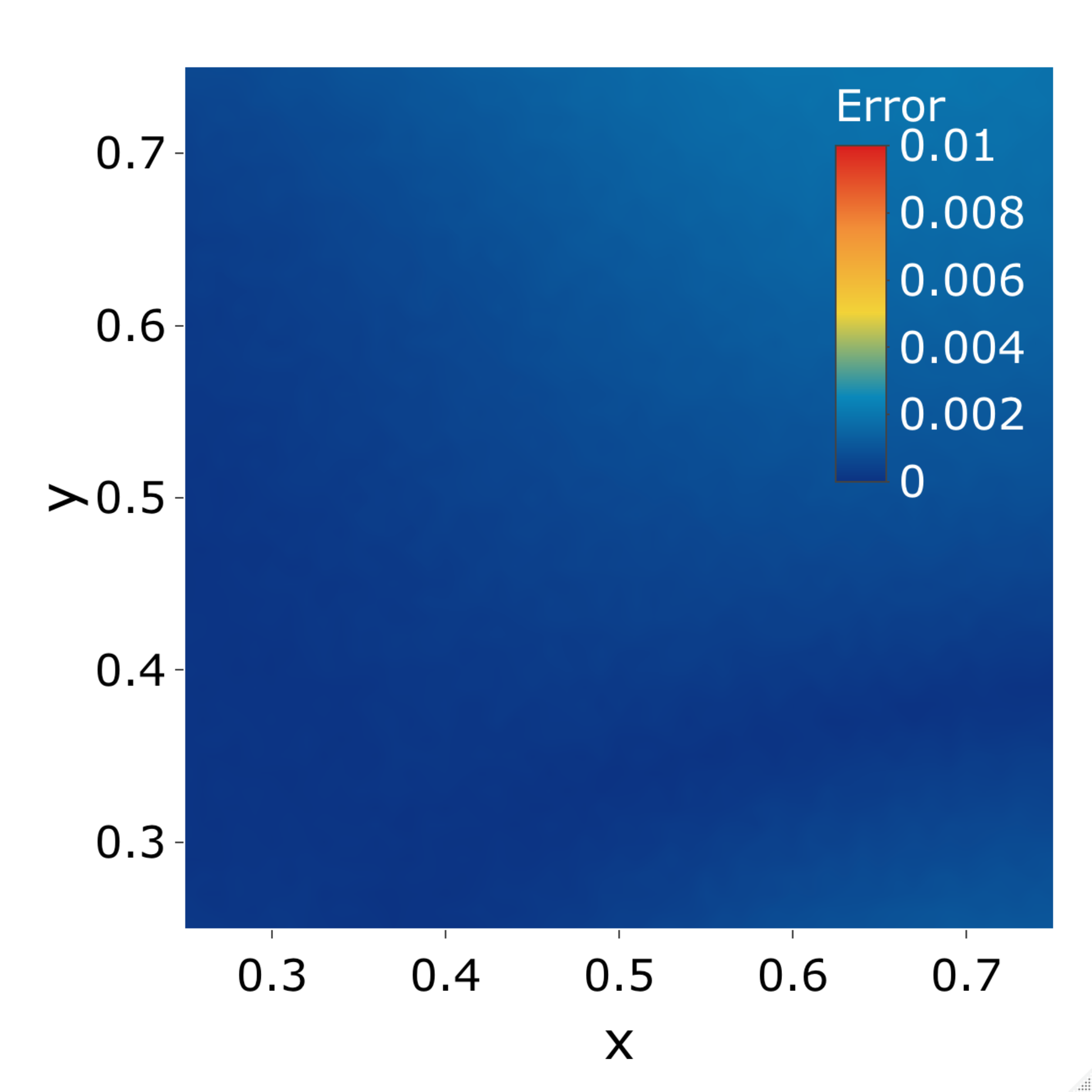}
  \includegraphics[width=0.32\textwidth]{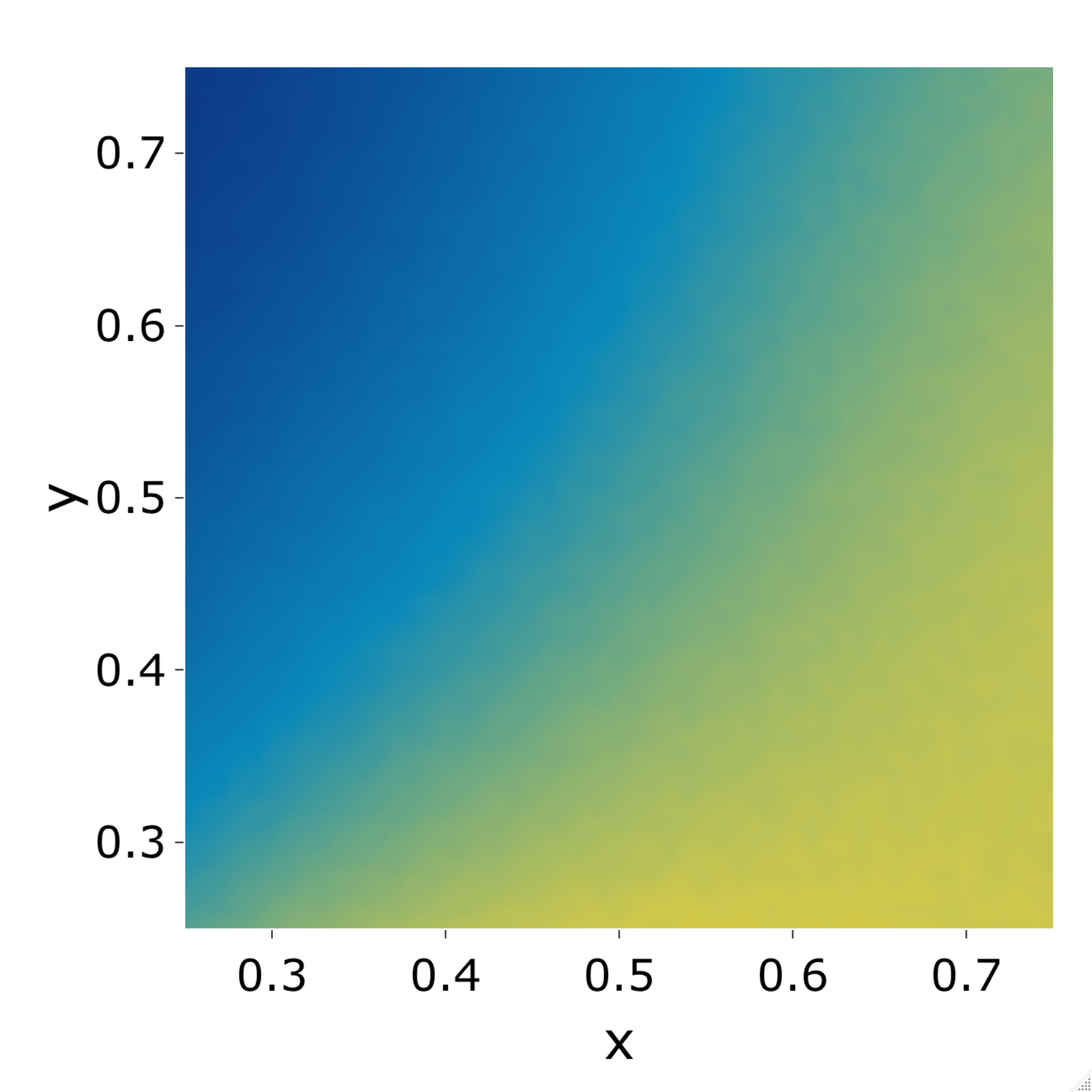}
  \includegraphics[width=0.32\textwidth]{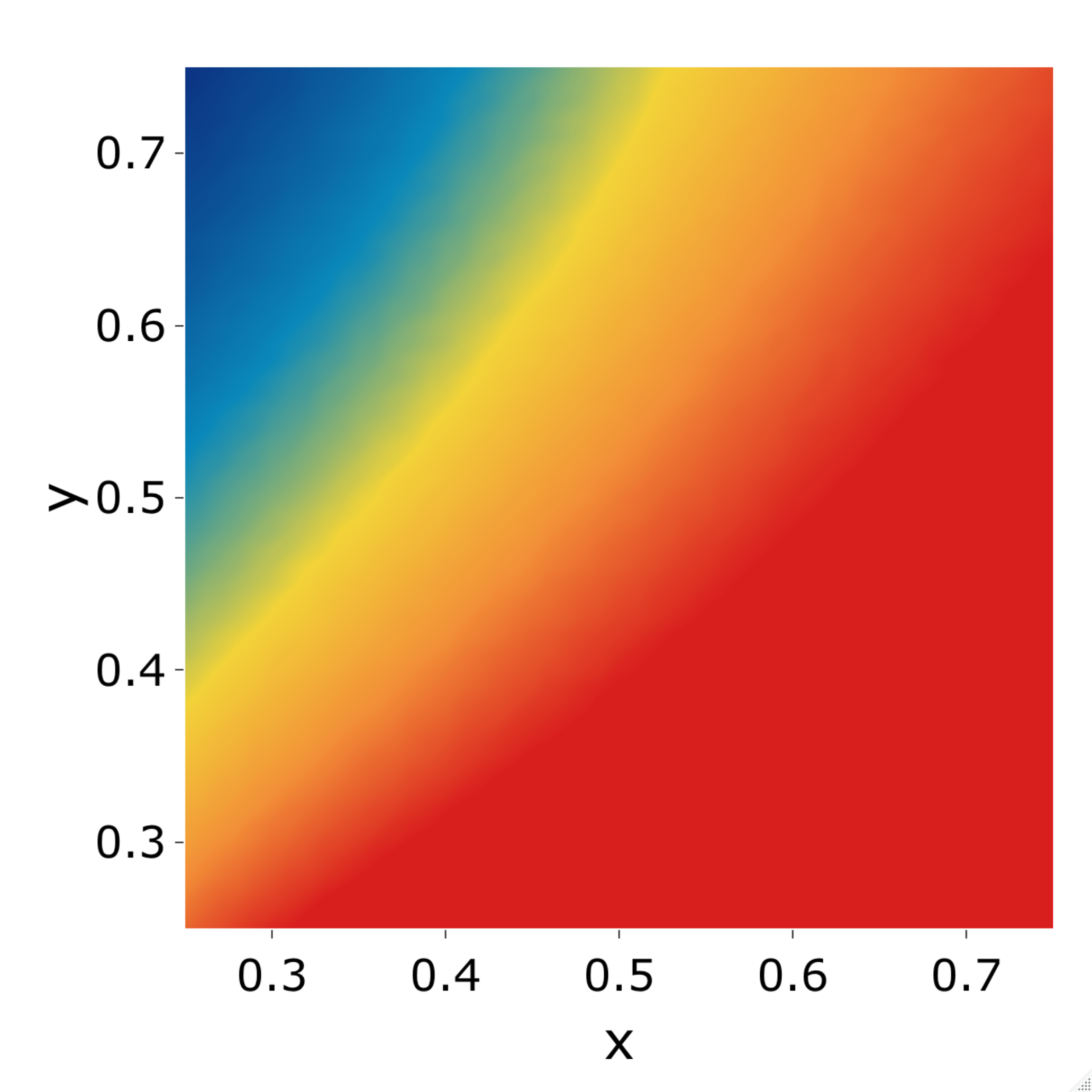}
  \includegraphics[width=0.32\textwidth]{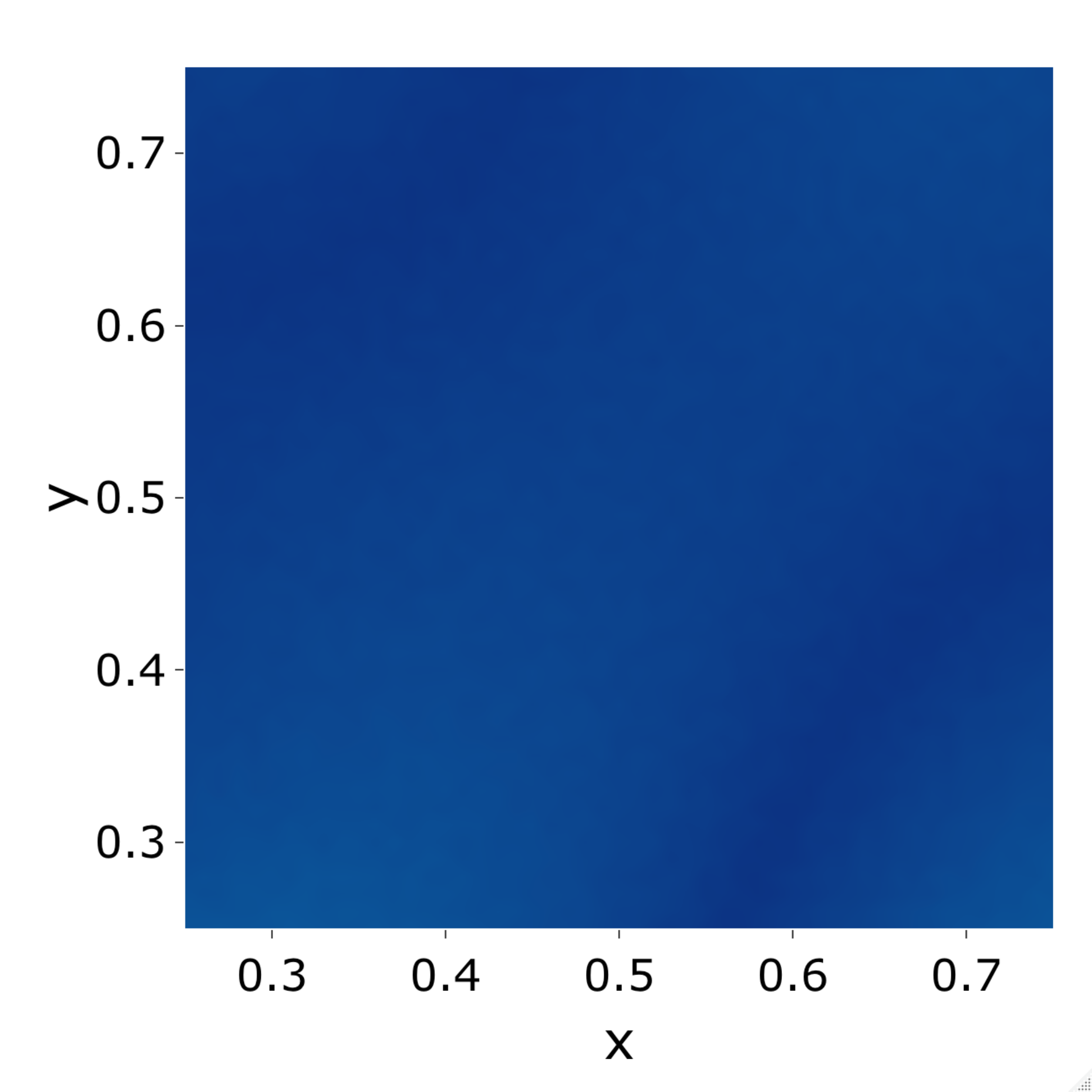}
  \includegraphics[width=0.32\textwidth]{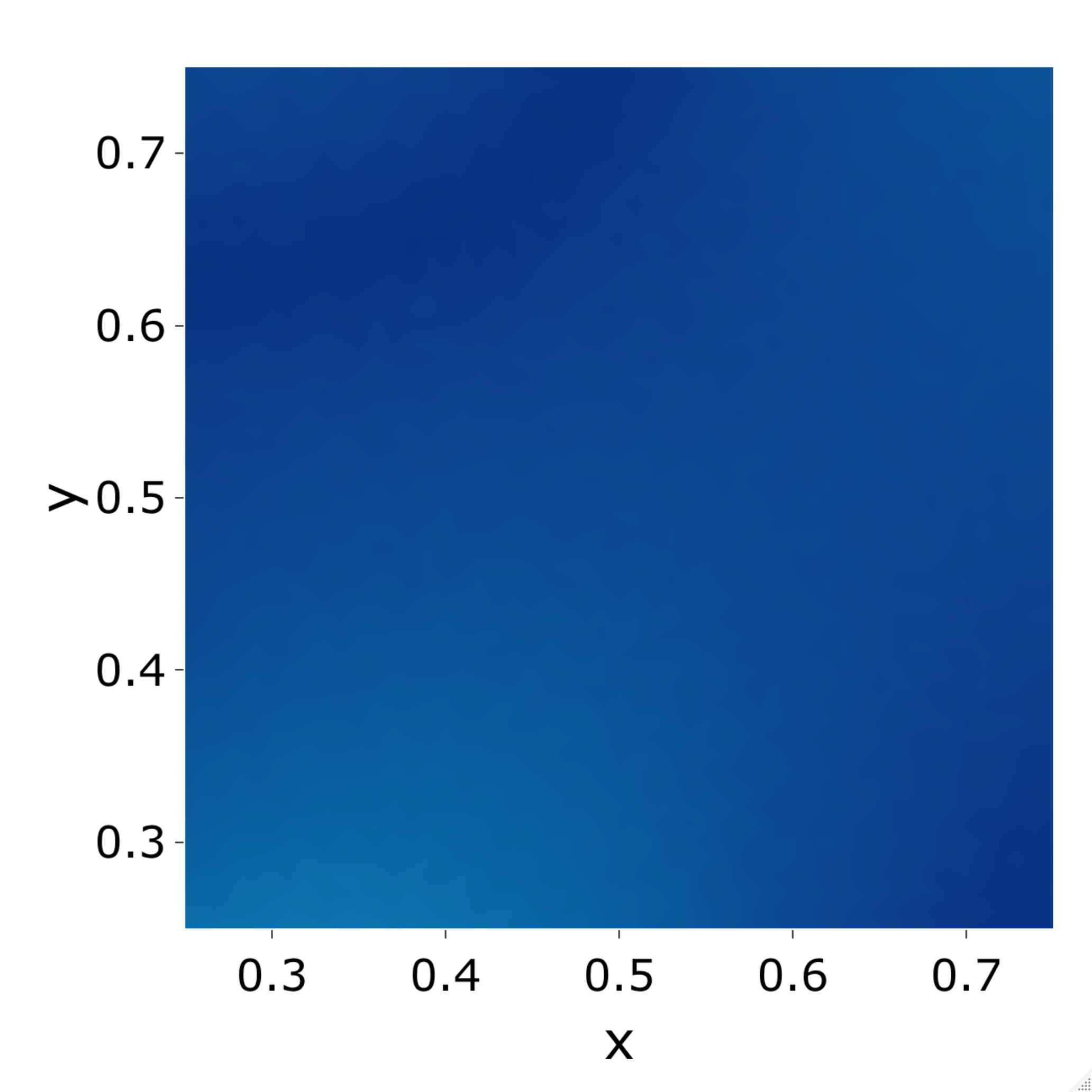}
  \includegraphics[width=0.32\textwidth]{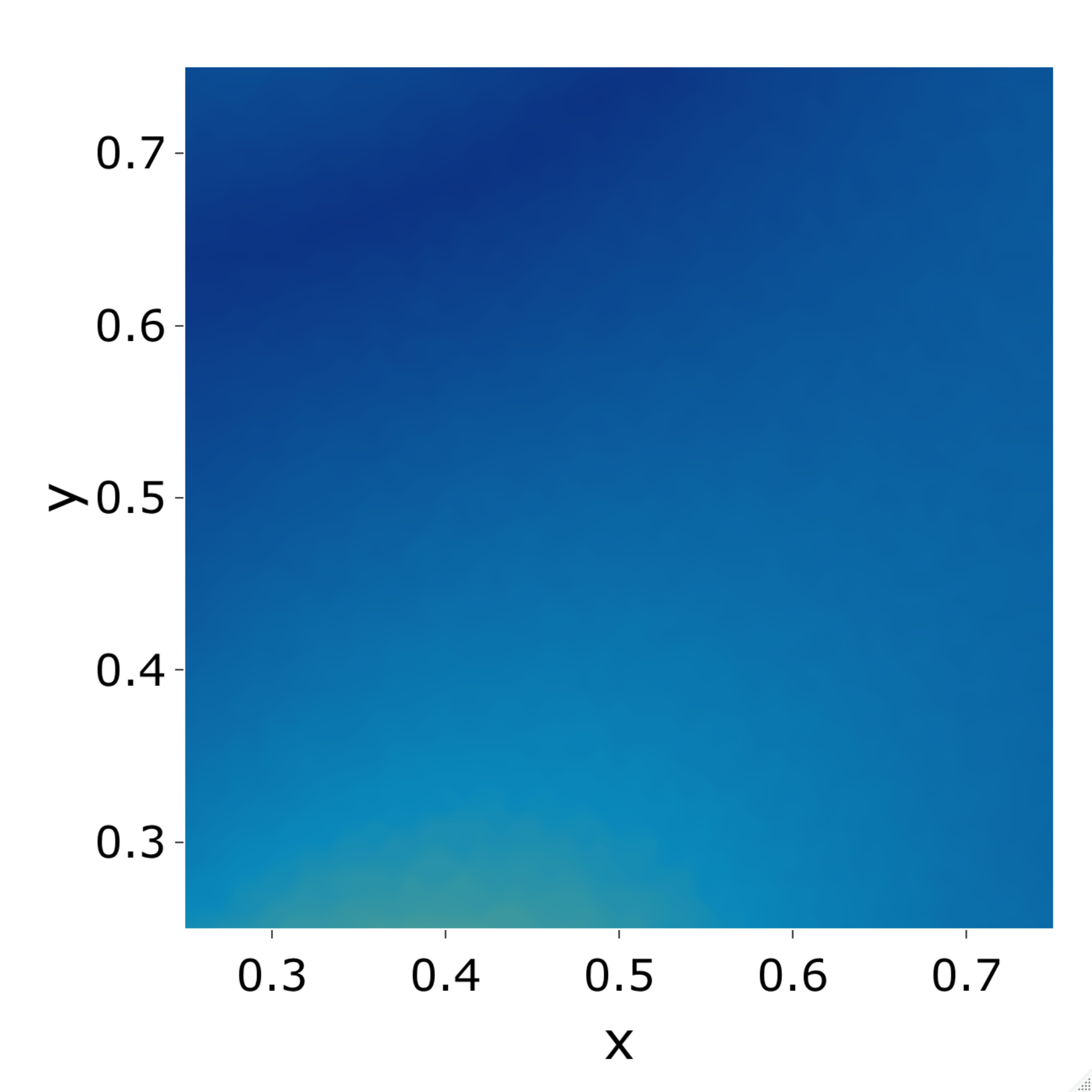}
  \noindent\makebox[\linewidth]{\rule{\textwidth}{0.6 pt}}
  \includegraphics[width=0.32\textwidth]{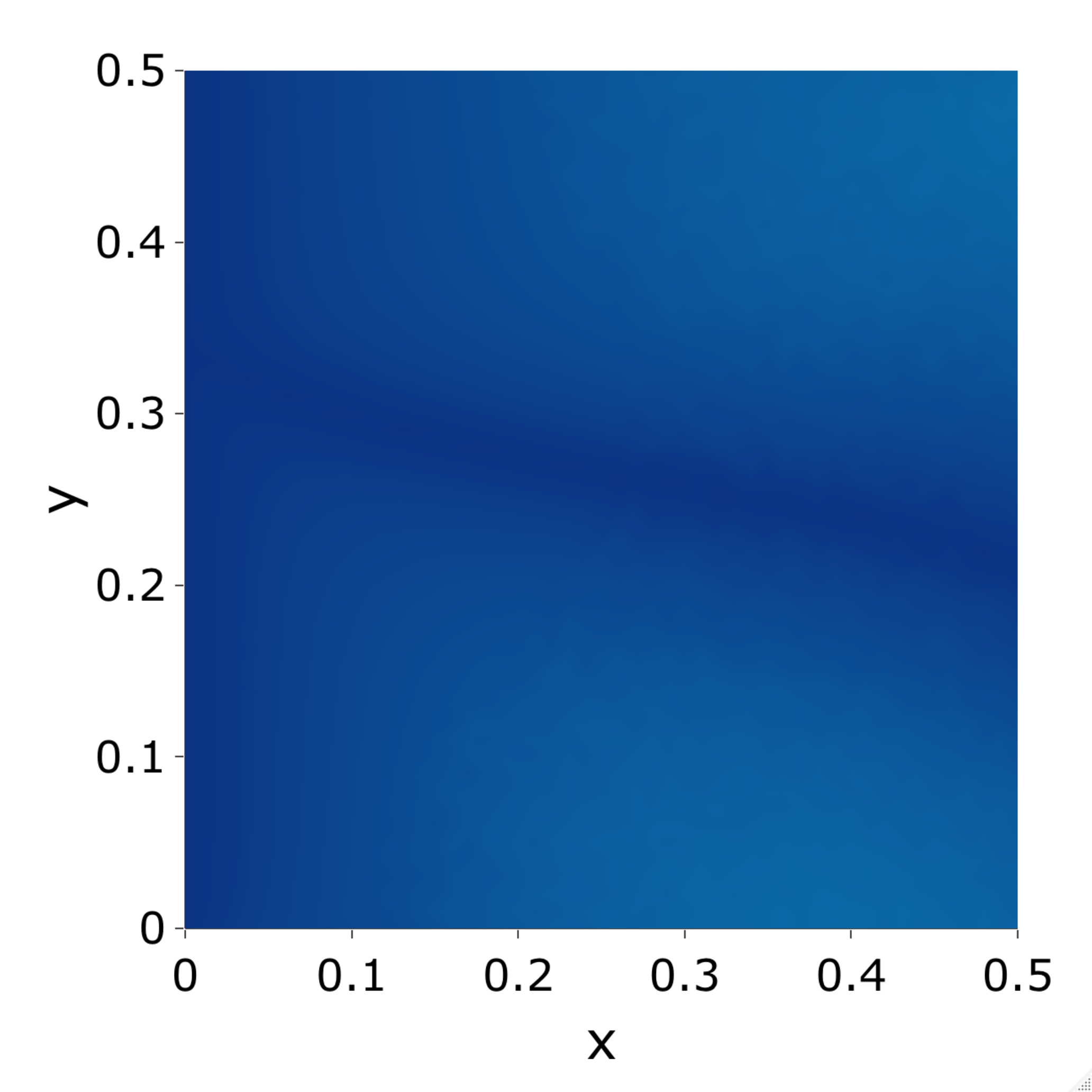}
  \includegraphics[width=0.32\textwidth]{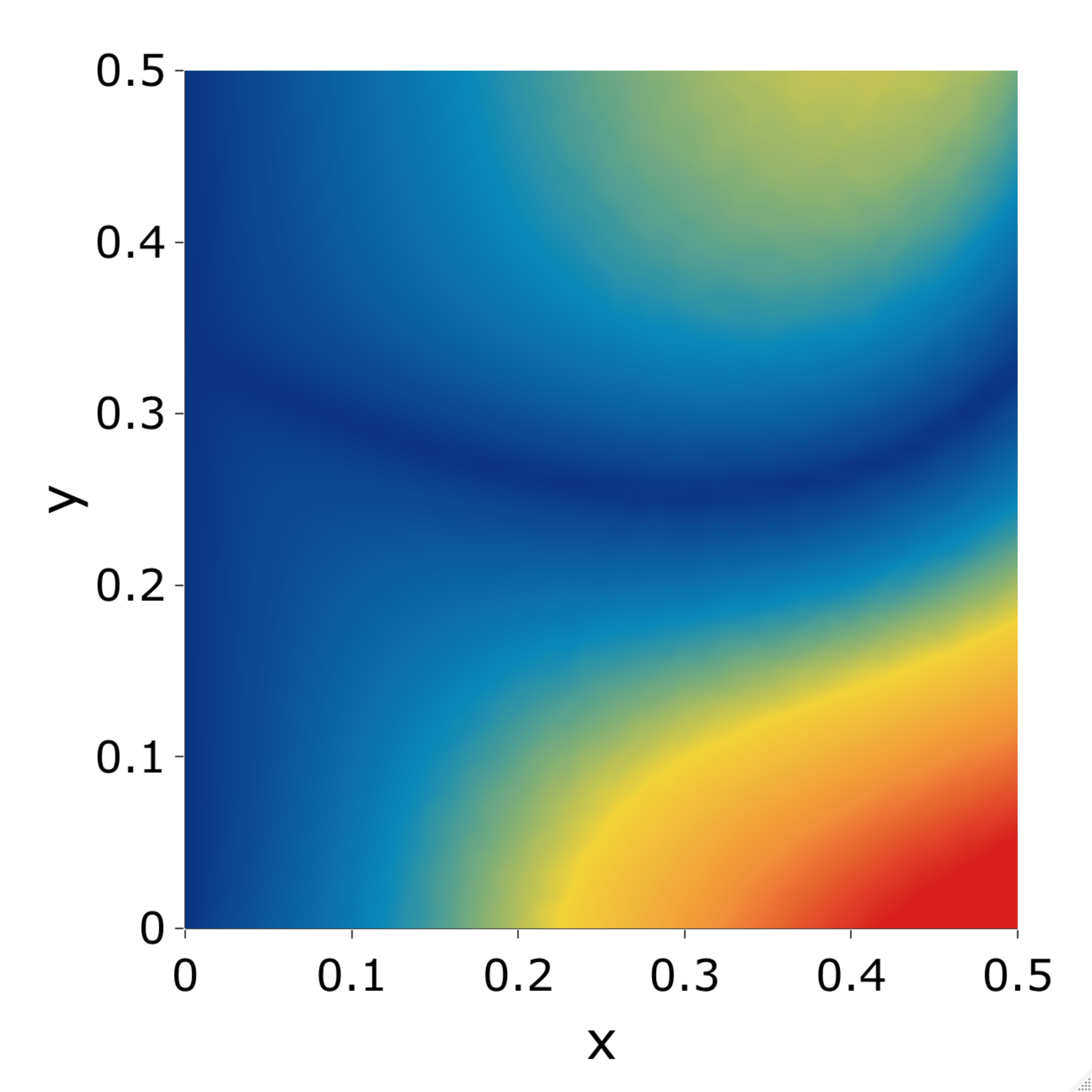}
  \includegraphics[width=0.32\textwidth]{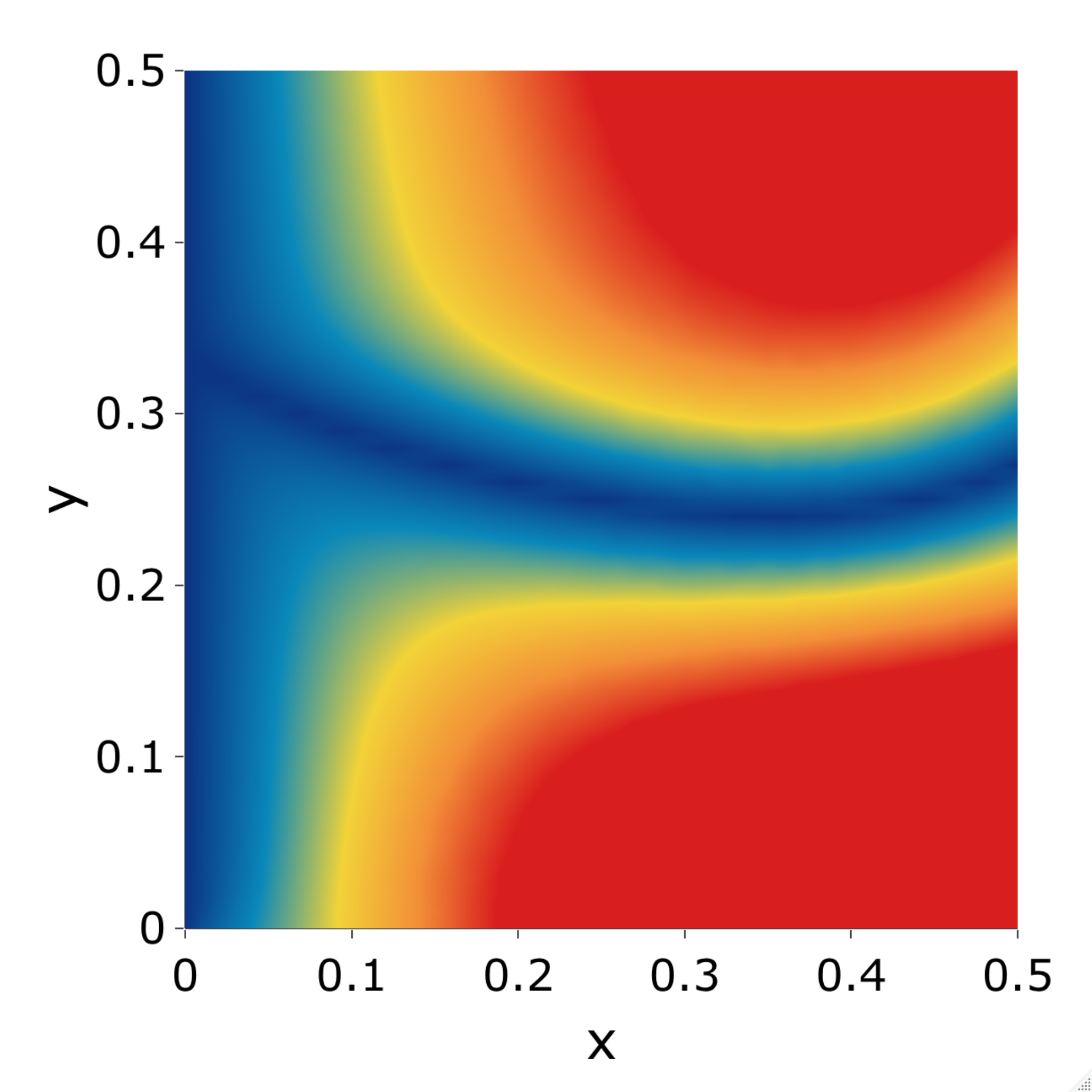}
  \includegraphics[width=0.32\textwidth]{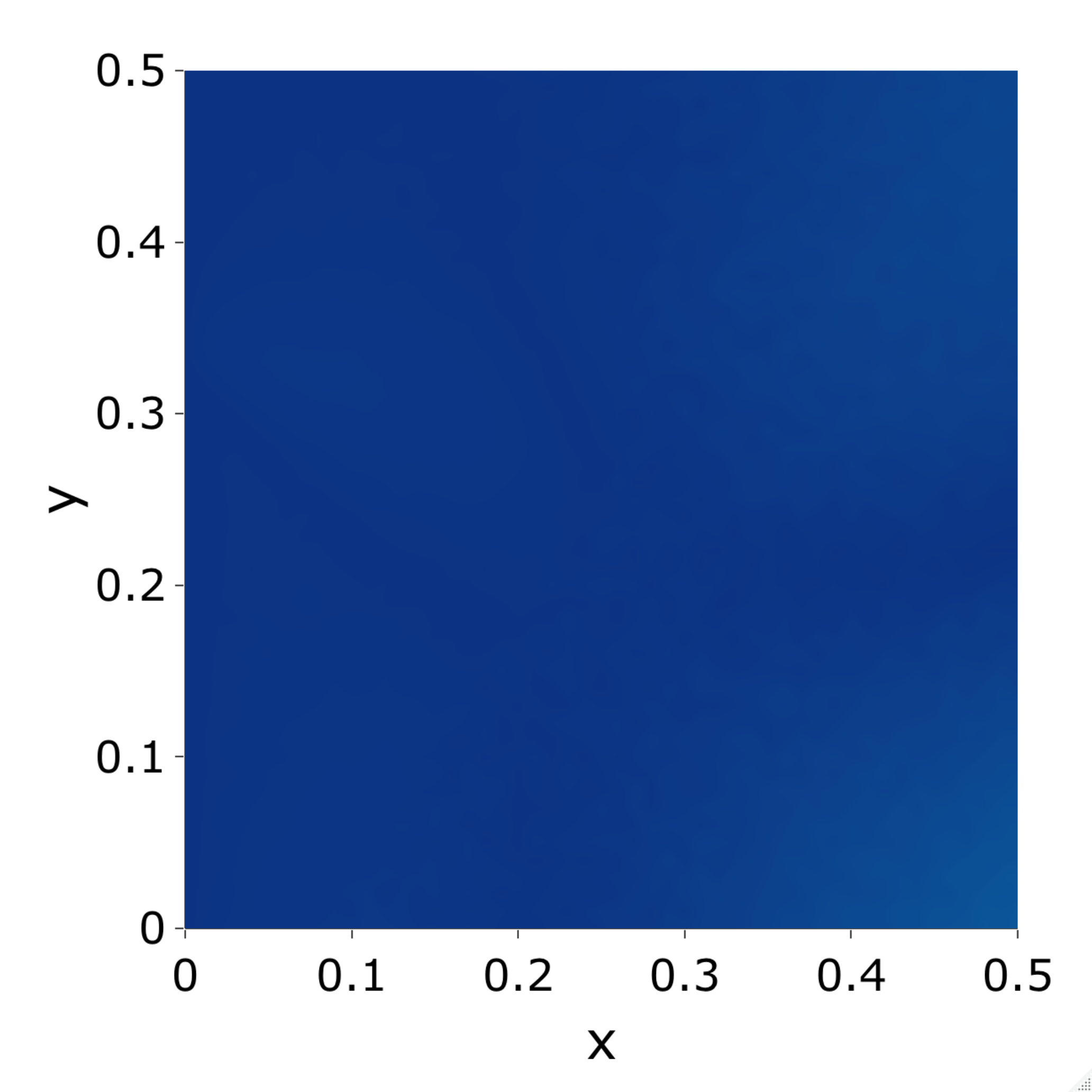}
  \includegraphics[width=0.32\textwidth]{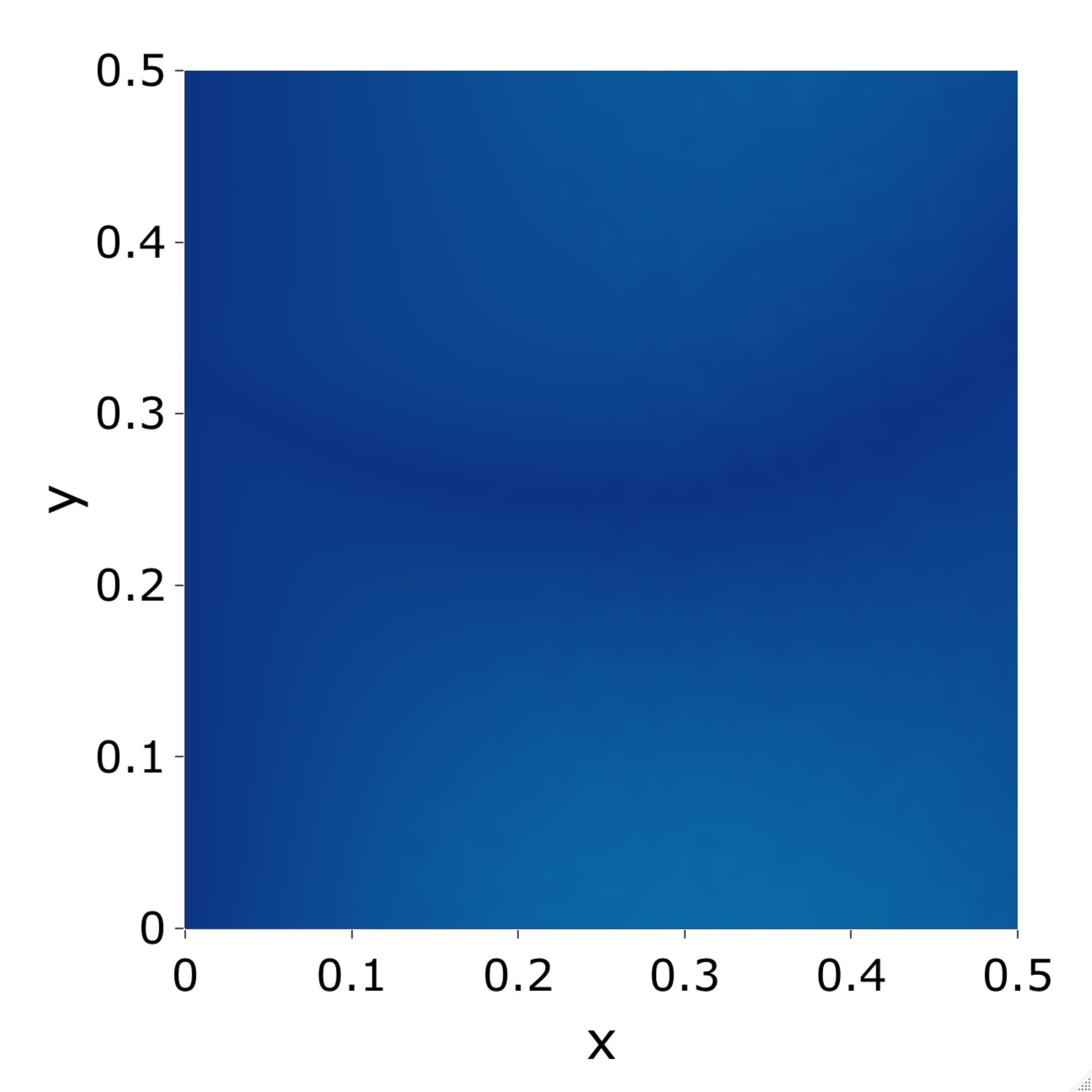}
  \includegraphics[width=0.32\textwidth]{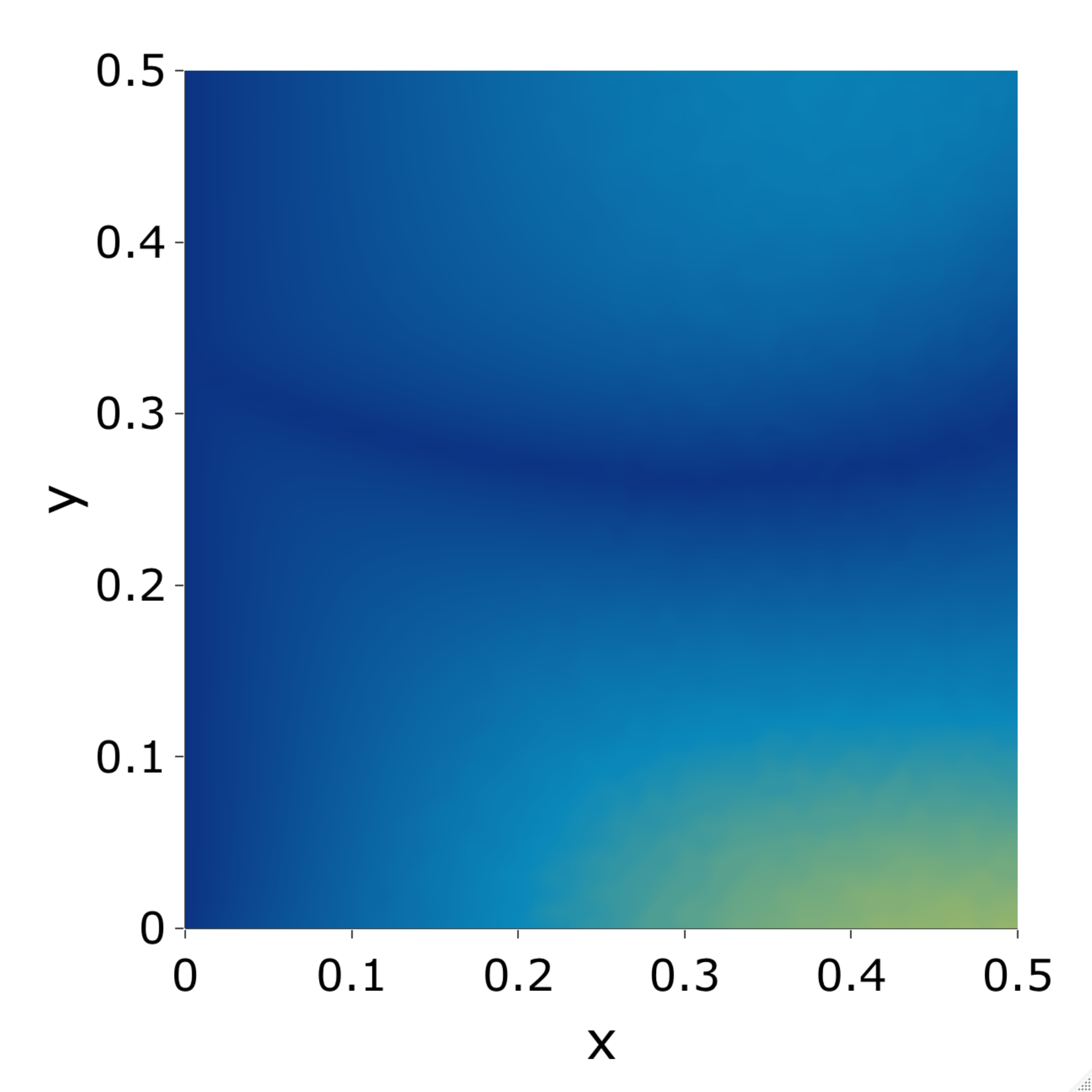}
   \caption{\small 
  Error comparison of the estimated $\E[f(Y)]$ given by the MC and the AETC algorithms. Top two rows: square domain. Bottom two rows: L-shape domain. Within each two-row block, the top row corresponds to budget-$B$ MC, and the bottom row to budget-$B$ AETC. (Left, middle, right) columns: pointwise spatial errors corresponding to $(0.05, 0.5, 0.95)$ quantiles of the total scalar $\ell_2$ error. The color limits are uniform for every plot, and are quantified in the top-left plot.
  }
  \label{fig:2}
  \end{center}
\end{figure}

\end{appendices}

\bibliographystyle{siamplain}
\bibliography{ref}

\end{document}